\documentclass[12pt]{amsart}

\usepackage{amsmath}
\usepackage{amsthm}
\usepackage{stmaryrd}
\usepackage{amsfonts,mathrsfs,amssymb}
\usepackage{times}
\usepackage{fullpage}
\title{$L^2$ estimates for the $\dbar$ operator}
\author{Jeffery D. McNeal}
\email{mcneal@math.ohio-state.edu}
\address{
Department of Mathematics \newline \indent
Ohio State University \newline \indent
Columbus, OH 43210-1174}
\author{Dror Varolin} 
\email{dror@math.sunysb.edu}
\address{
Department of Mathematics \newline \indent
Stony Brook University \newline \indent
Stony Brook, NY 11794-3651}

\newcommand{\noi}{\noindent}

\newcommand{\cd}{{\mathcal D}}

\newcommand{\co}{{\mathcal O}}

\newcommand{\sa}{{\mathscr A}}
\newcommand{\sB}{{\mathscr B}}

\newcommand{\sd}{{\mathscr D}}

\newcommand{\sg}{{\mathscr G}}
\newcommand{\sh}{{\mathscr H}}

\newcommand{\sr}{{\mathscr R}}

\newcommand{\ve}{\varepsilon}
 
\newcommand{\vp}{\varphi} 

\newcommand{\B}{{\mathbb B}}
\newcommand{\C}{{\mathbb C}}
\newcommand{\D}{{\mathbb D}}

\newcommand{\p}{{\mathbb P}}

\newcommand{\R}{{\mathbb R}}

\newcommand{\Z}{{\mathbb Z}}

\newcommand{\red}{\hfill $\diamond$}

\newcommand{\di}{\partial}
\newcommand{\dbar}{\bar \partial}

\newcommand{\re}{{\rm Re\ }}

\newcommand{\relcomp}{\subset \subset}

\newcommand{\gradbar}{{\rm grad}^{0,1}}
\newcommand{\ii}{\sqrt{-1}}

\newcommand{\emb}{\hookrightarrow}
\newcommand{\tensor}{\otimes}

\def\XXint#1#2#3{{\setbox0=\hbox{$#1{#2#3}{\int}$} 
\vcenter{\hbox{$#2#3$}}\kern-.5\wd0}}

\usepackage{hyperref}

\begin{document}

\theoremstyle{plain}
\newtheorem{thm}{\sc Theorem}
\newtheorem*{s-thm}{\sc Theorem}
\newtheorem{lem}{\sc Lemma}[subsection]
\newtheorem{d-thm}[lem]{\sc Theorem}
\newtheorem{prop}[lem]{\sc Proposition}
\newtheorem{cor}[lem]{\sc Corollary}

\theoremstyle{definition}
\newtheorem{conj}[lem]{\sc Conjecture}
\newtheorem{prob}[lem]{\sc Open Problem}
\newtheorem{defn}[lem]{\sc Definition}
\newtheorem{qn}[lem]{\sc Question}
\newtheorem{ex}[lem]{\sc Example}
\newtheorem{rmk}[lem]{\sc Remark}
\newtheorem{rmks}[lem]{\sc Remarks}
\newtheorem*{ack}{\sc Acknowledgment}
\newtheorem*{s-rmk}{\sc Remark}




\maketitle

\setcounter{tocdepth}1



\vskip .3in



\begin{abstract}
This is a survey article about $L^2$ estimates for the $\bar \partial$ operator.  After a review of the basic approach that has come to be called the ``Bochner-Kodaira Technique", the focus is on twisted techniques and their applications to estimates for $\bar \partial$, to $L^2$ extension theorems, and to other problems in complex analysis and geometry including invariant metric estimates and the $\dbar$-Neumann Problem.
\end{abstract}

\section*{Introduction}

Holomorphic functions of several complex variables can be defined as those functions that satisfy the so-called Cauchy-Riemann equations.  Because the Cauchy-Riemann equations form an elliptic system, holomorphic functions are subject to various kinds of rigidity, both local and global. The simplest striking example of this rigidity is the identity principle, which states that a holomorphic function is completely determined by its values on {\it any} open subset.  Another simple yet profound rigidity is manifest in the maximum principle, which states that a holomorphic function that assumes an interior local maximum on a connected open subset of $\C ^n$ must be constant on that subset.  As a consequence, there are no non-constant holomorphic functions on any compact connected complex manifold.  A third striking kind of rigidity, present only when the complex dimension of the underlying space is at least 2, is the Hartogs Phenomenon: the simplest example of this phenomenon occurs if $\Omega$ is the complement of a compact subset of a Euclidean ball $B$, in which case the Hartogs Phenomenon is that any holomorphic function on $\Omega$ is the restriction of a unique holomorphic function in $B$.

Because of their rigidity, holomorphic functions with specified geometric properties are often hard to construct.  A fundamental technique used  in their construction is the solution of the inhomogeneous Cauchy-Riemann equations with estimates.  If the right sorts of estimates are available, one can construct a smooth function with the desired property (and this is often easy to do), and then correct this smooth function to be holomorphic by adding to it an appropriate solution of a certain inhomogeneous Cauchy-Riemann equation.

While the most natural estimates one might want are uniform estimates, they are often very hard or impossible to obtain.  On the other hand, it turns out that $L^2$ estimates are often much easier, or at least possible, to obtain if certain geometric conditions underlying the problem are satisfied.

Close cousins of holomorphic functions are harmonic functions, which have been the subjects of extensive study in complex analysis since its birth.  Not long after the concept of manifolds became part of the mathematical psyche (and to some extent even earlier), mathematicians began to extract geometric information from the behavior of harmonic functions and differential forms.  In the 1940s, Bochner introduced his technique for getting topological information from the behavior of harmonic functions.  Around the same time, Hodge began to extract topological information about algebraic varieties from Bochner's ideas.  In the hands of Kodaira, the Bochner Technique saw incredible applications to algebraic geometry, including the celebrated  Kodaira Embedding Theorem.  

About twenty years earlier, mathematicians began to study holomorphic functions from another angle, more akin to Perron's work in one complex variable.  Among the most notable of these was Oka, who realized that plurisubharmonic functions were fundamental tools in bringing out the properties of holomorphic functions and the natural spaces on which they occur.  

In the late 1950s and early 1960s, the approaches of Bochner-Kodaira and Oka began to merge into a single, and very deep approach based in partial differential equations.  The theory was initiated by Schiffer and Spencer, who were working on Riemann surfaces.  Spencer defined the $\dbar$-Neumann problem, and intensive research by Andreotti and Vesentini, H\"ormander, Morrey, Kohn, and others began to take hold.  It was Kohn who finally formulated and solved the $\dbar$-Neumann problem on strictly pseudoconvex domains, after a crucial piece of work by Morrey.  Kohn's work, which should be viewed as the starting point for the Hodge Theorem on manifolds with boundary, provided $L^2$ estimates and regularity up to the boundary for the solution of the Cauchy-Riemann equations having minimal $L^2$ norm.  In the opinion of the authors, Kohn's work is one of the incredible achievements in twentieth century mathematics. 

Shortly after Kohn's work, H\"ormander and Andreotti-Vesentini, independently and almost simultaneously,  obtained weighted $L^2$ estimates for the inhomogeneous Cauchy-Riemann equations.  The harmonic theory of Bochner-Kodaira and the plurisubharmonic theory of Oka fit perfectly into the setting of weighted $L^2$ estimates for the inhomogeneous Cauchy-Riemann equations.  The theorem also gave a way of getting interesting information about the Bergman projection: the integral operator that orthogonally projects $L^2$ functions onto the subspace of holomorphic $L^2$ functions.

The applications of the Andreotti-H\"ormander-Vesentini Theorem were fast and numerous.  There is no way we could mention all of them here, 
and neither is it our intention to do so.  Our story, though it will have some elements from this era, begins a little later.

In the mid 1980s, Donnelly and Fefferman discovered a technique that allowed an important improvement of the Bochner-Kodaira technique, which we call "twisting".  Ohsawa and Takegoshi adapted this technique to prove a powerful and general extension theorem for holomorphic functions.  Extension is a way of constructing holomorphic functions by induction on dimension.  The $L^2$ extension theorem, as well as the twisted technique directly, has been used in a number of important problems in complex analysis and geometry, but it is our feeling that there are many more applications to be had.

Thus we come to the purpose of this article: it is meant to lie somewhere between a survey and a lecture on past work, both by us and others. 

The paper is split into two parts.  In the first part, we shall explain the Bochner-Kodaira-Morrey-Kohn-H\"ormander technique, and its twisted analogue.  In the second, considerably longer part, we shall demonstrate some of the applications that these techniques have had.  We shall discuss improvements to the H\"ormander theorem, $L^2$ extension theorems,   invariant metric estimates, and applications to the $\dbar$-Neumann problem on domains that are not necessarily strictly pseudoconvex (though they are still somewhat restricted, depending on what one proving).  We do not provide all proofs, but where we do not prove something, we provide a reference.

There are many topics that have been omitted, but which could have naturally been included here.  We chose to focus on the analytic techniques that lie behind these results, with the goal of equipping a reader with the understanding needed to easily learn about these topics from the original papers, and to apply these techniques to other problems.

After this paper was written, we received the interesting expository paper \cite{blocki-BMS}, which has some overlap with our article. In both papers, the aim is to explain $L^2$ estimates on $\dbar$ that have been derived after the work of
Andreotti-H\"ormander-Vesentini and to apply them to certain problems. But the differences between B\l ocki's and our paper are significant. We discuss the basic apparatus of the $\dbar$ estimates in differential-geometric terms, which allows us to present
the estimates on $(p,q)$-forms values in holomorphic bundles over domains in K\" ahler manifolds starting in subsection \ref{SS:kahler}, and continuing thereafter. This generality of set-up yields $L^2$ extension theorems of Ohsawa-Takegoshi type of much wider applicability in Section \ref{S:extension}, e.g., in subsection \ref{SS:denominators}.  There is no doubt, however, that such a level of generality could also have been achieved in \cite{blocki-BMS}, had the author wanted to include it.  A more significant difference between the two papers is conceptual: we view the idea of twisting the $\dbar$ complex as a basic method, one that has led to new $L^2$ estimates and provides a framework to obtain additional estimates. The point of view taken in B\l ocki's paper is that the H\"ormander (or more correctly, Skoda) estimate is primary and can be manipulated, ex post facto, to yield new inequalities.  B\l ocki writes that our methods are "more complicated", but we do not agree with this characterization, which seems to be a matter of taste.  Indeed, the methods we use are equal in complexity to the methods needed to prove H\"ormander-Skoda's estimates, so it seems to us that any mathematician who wants to understand the entire picture will find no greater economy in one or the other approaches. 

While it is certain that both approaches have their pedagogical benefits, we wanted to highlight how our approach shows that certain curvature conditions lead directly to useful new estimates.  The path from the curvature hypotheses to the estimates that uses H\"ormander-Skoda's Theorem as a black box seemed to us to be more ad hoc.  We also give an extended discussion in subsection \ref{SS:twist_more} about how the new inequalities -- whether derived via twisting or by manipulating H\"ormander's inequality -- do give estimates that are genuinely stronger than H\" ormander's basic estimate alone. In terms of applications, aside from our homage to B\l ocki's and others' work on the resolution of the Suita conjecture (cf. Section \ref{S:optimal_constants})  the overlap of the two papers is only around the basic Ohsawa-Takegoshi extension theorem. B\l ocki's very interesting applications go more in the potential theory direction -- singularities of plurisbharmonic functions and the pluricomplex Green's function -- while ours go more towards the $\dbar$-Neumann problem -- compactness and subelliptic estimates, and pointwise estimates on the Bergman metric. There is no doubt in our mind that the present article and B\l ocki's paper supplement each other, providing both different conceptual points of view and different applications.

\begin{ack}
The authors are grateful to the many mathematicians whose work is surveyed and discussed in the present article.  Many of these friends are active researchers with whom one or both of us has interacted on many occasions, and we have learned much from them over the years.  We are especially grateful to Bo Berndtsson, Dave Catlin, Jean-Pierre Demailly, Charles Fefferman, Joe Kohn, Takeo Ohsawa, and Yum-Tong Siu, who have been inspiring to us throughout our career, and from whom we have learned immeasurably much.
\red
\end{ack}

\newpage

\begin{center}
{\large {\sc Part I.  The Basic Identity and Estimate, and its Twisted Relatives}}
\end{center}

\section{The Bochner-Kodaira-Morrey-Kohn Identity}

We begin by recalling an important identity for the $\dbar$-Laplace Beltrami operator.  We shall first state it in the simplest case of $(p,q)$-forms on domains in $\C ^n$, and then for the most general case of $(p,q)$-forms with values in a holomorphic vector bundle over a domain in a K\"ahler manifold.

\subsection{\bf Domains in $\C ^n$}\label{domains-par}

Let us begin with the simplest situation of a bounded domain $\Omega \relcomp \C ^n$.  We assume that $\Omega$ has a smooth boundary $\di \Omega$ that is a real hypersurface in $\C ^n$.  We fix a function $\rho$ in a neighborhood $U$ of $\di \Omega$ such that 
\[
U \cap \Omega = \{ z\in \C ^n\ ;\ \rho (z) < 0\}, \quad \di \Omega = \{ z\in \C ^n \ ;\ \rho (z) = 0\} \quad \text{and} \quad |\di \rho|\equiv 1 \text{ on } \di \Omega.
\]
(A function satisfying the first two conditions is called a {\it defining function} for $\Omega$, and a defining function normalized by the third condition is called a {\it Levi defining function}.)  Suppose also that $e^{-\vp}$ is a smooth weight function on $\Omega$.  

Here and below, we use the standard summation convention on $(p,q)$-forms, which means we sum over repeated upper and lower indices (of the same type).  We employ the usual multi-index notation
\[
dz^I = dz^{i_1} \wedge ... \wedge dz^{i_p} \quad \text{and} \quad d\bar z ^J = d\bar z^{j_1} \wedge ... \wedge d\bar z^{j_q},
\]
and write
\[
\alpha ^{I \bar J} = \alpha _{K\bar L} \delta ^{i_1\bar \ell _1} \cdots \delta ^{i_p\bar \ell _p} \delta ^{k_1\bar j_1}\cdots \delta ^{k_q\bar j_q}.
\]
We define 
\[
\left < \alpha , \beta \right > := \alpha _{I\bar J} \overline{\beta ^{J\bar I}} \quad \text{and} \quad |\alpha|^2 = \alpha _{I\bar J} \overline{\alpha ^{J \bar I}}.
\]

Using the pointwise inner product $\left < \cdot , \cdot \right >$ and the weight function $e^{-\vp}$, we can define an inner product on the space of smooth $(p,q)$-forms by the formula 
\[
(\alpha ,\beta )_{\vp} := \int _{\Omega} \left < \alpha , \beta \right > e^{-\vp} dV.
\]
We define $L^2_{(p,q)}(\Omega, e^{-\vp})$ to be the Hilbert space closure of the set of all smooth $(p,q)$-forms $\alpha = \alpha _{I\bar J} dz^I \wedge d\bar z ^J$ on a neighborhood of $\Omega$.  As usual, these spaces consist of $(p,q)$-forms with coefficients that are square-integrable on $\Omega$ with respect to the measure $e^{-\vp}dV$.

On a smooth $(p,q)$-form one has the so-called $\dbar$ operator (also called the Cauchy-Riemann operator) defined by 
\[
\dbar \alpha = \frac{\di \alpha _{I\bar J}}{\di \bar z ^k} d\bar z^k \wedge dz^I \wedge d\bar z ^J = (-1)^p  \frac{\di \alpha _{I\bar J}}{\di \bar z ^k} dz^I \wedge d\bar z^k \wedge d\bar z ^J = (-1)^p \ve _{\bar K}^{\bar k \bar J} \frac{\di \alpha _{I\bar J}}{\di \bar z ^k} dz^I \wedge d\bar z^K,
\]
where $\ve ^M_N$ denotes the sign of the Permutation taking $M$ to $N$.  Evidently $\dbar$ maps smooth $(p,q)$-forms to smooth $(p,q+1)$-forms and satisfies the compatibility condition $\dbar ^2 = 0$.

We can now define the formal adjoint $\dbar ^*_{\vp}$ of $\dbar$ as follows:  if $\alpha$ is a smooth $(p,q)$-form on $\Omega$, then the formal adjoint satisfies 
\[
(\dbar ^*_{\vp}\alpha , \beta) = (\alpha ,\dbar \beta)
\]
for all smooth $(p,q-1)$-forms $\beta$ with compact support in $\Omega$.  A simple integration-by-parts argument shows that 
\begin{equation}\label{formal-adjoint-cn}
(\dbar ^*_{\vp} \alpha)_{I\bar J} = (-1)^{p-1} e^{\vp} \delta ^{k\bar j}\frac{\di}{\di z^{k}}\left ( e^{-\vp}\alpha _{I\overline{j J}}\right ).
\end{equation}
If $\alpha$ is a smooth form, one can directly compute that
\begin{equation}\label{formal-bk-cn}
(\dbar \dbar ^*_{\vp} + \dbar ^*_{\vp}\dbar )\alpha =  \delta ^{i \bar j}\bar \nabla ^* _i \bar \nabla _j \alpha +  \sum _{k=1} ^q \delta ^{i \bar s}\frac{\di ^2 \vp}{\di z^{i}\di \bar z ^{j_k}} \alpha _{I \bar j_1...(\bar s )_{k}...\bar j_q},
\end{equation}
where $\bar \nabla _{j} = \frac{\di}{\di \bar z ^j}$ and $\bar \nabla ^*$ is the formal adjoint of $\bar \nabla$.  

\begin{rmk}
The geometric meaning of the important formal identity \eqref{formal-bk-cn} will be expanded on in the next paragraph.  For the time being, we would like to make the following comment.  The tensor $\alpha$ can either be thought of as a differential form, or a section of the vector bundle $\Lambda ^{p,q}_{\Omega} \to \Omega$.  (We view both of these bundles as having a non-trivial Hermitian metric induced by the weight $e^{-\vp}$.)  As a differential form, one can act on it with the $\dbar$-Laplace Beltrami operator $\dbar \dbar ^* + \dbar ^*\dbar$, but as a section of $\Lambda ^{p,q}_{\Omega}$, one has to use the "covariant $\dbar$ operator $\bar \nabla$, because the latter bundle is not holomorphic and therefore doesn't have a canonical choice of $\dbar$ operator.  The main thrust of the formula \eqref{formal-bk-cn} is that these two, a priori nonnegative $\dbar$-Laplace operators are related by the complex Hessian of $\vp$ (as it acts on $(p,q)$-forms).
\red
\end{rmk}

If in \eqref{formal-bk-cn} we assume $\alpha$ has compact support in $\Omega$, then taking inner product with $\alpha$ and integrating-by-parts yields
\begin{equation}\label{formal-bk-cn-integrated}
||\dbar ^*_{\vp}\alpha||^2_{\vp} + ||\dbar \alpha ||^2_{\vp} = \int _{\Omega} |\bar \nabla \alpha|^2 e^{-\vp} dV + \int _{\Omega}\sum _{k=1} ^q \delta ^{i \bar s}\frac{\di ^2 \vp}{\di z^{i}\di \bar z ^{j_k}} \alpha _{I \bar j_1...(\bar s )_{k}...\bar j_q} \overline{\alpha ^{\bar I j_1 ... j_k ... j_q}} e^{-\vp} dV.
\end{equation}
In particular, if there is a constant $c>0$ such that 
\[
\frac{\di ^2\vp}{\di z^i \di \bar z ^j} \ge c \delta _{i\bar j},
\]
then we get the inequality 
\begin{equation}\label{basic-est-not}
||\dbar ^*_{\vp}\alpha||^2_{\vp} + ||\dbar \alpha ||^2_{\vp} \ge c ||\alpha||^2
\end{equation}
for all smooth $(p,q)$-forms $\alpha$ with compact support.  But since smooth forms with compact support do not form a dense subset of $L^2_{(p,q)}(\Omega , e^{-\vp})$ with respect to the so-called {\it graph norm} $||\alpha||+||D\alpha||$, we cannot take advantage of this estimate.  In fact, the estimate is not true without additional assumptions on $\Omega$. 

To clarify the situation, we must develop the theory of the $\dbar$ operator a little further, and in particular, extend it to a significantly larger subset of the space $L^2_{(p,q)}(\Omega , e^{-\vp})$.  Such a development requires some of the theory of unbounded operators and their adjoints, which we now outline in the case of $\dbar$.

The operator $\dbar$ is extended to $L^2_{(p,q)}(\Omega , e^{-\vp})$ as follows.  First, it is considered as an operator in the sense of currents.  But as such, the image of $L^2_{(p,q)}(\Omega , e^{-\vp})$ is a set of currents that properly contains $L^2_{(p,q+1)}(\Omega , e^{-\vp})$.  We therefore limit the domain of $\dbar$ by defining 
\[
{\rm Domain}(\dbar) := \{ \alpha \in L^2_{(p,q)}(\Omega , e^{-\vp})\ ;\ \dbar \alpha \in L^2_{(p,q+1)}(\Omega , e^{-\vp})\}.
\]
Of course, this Hilbert space operator extending $\dbar$ (which we continue to denote by $\dbar$) is not bounded on $L^2_{(p,q)}(\Omega , e^{-\vp})$, but nevertheless it has two important properties.
\begin{enumerate}
\item[(i)] It is densely defined: indeed, ${\rm Domain}(\dbar)$ is a dense subset of $L^2_{(p,q)}(\Omega , e^{-\vp})$ since it contains all the smooth forms on a neighborhood of $\Omega$.
\item[(ii)] It is a closed operator, i.e., the Graph  $\{ (\alpha , \dbar \alpha )\ ;\ \alpha \in {\rm Domain}(\dbar)\}$ of $\dbar$, is a closed subset of $L^2_{(p,q)}(\Omega , e^{-\vp}) \times L^2_{(p,q+1)}(\Omega , e^{-\vp})$.
\end{enumerate}
Much of the theory of bounded operators extends to the class of closed, densely defined operators.  For example, the Hilbert space adjoint of a closed, densely defined operator is itself a closed densely defined operator.

Though we will not define the Hilbert space adjoint of a closed densely defined operator in general, the definition should be clear from what we do for $\dbar$.  First, one defines the domain of $\dbar ^*_{\vp}$ to be 
\[
{\rm Domain}(\dbar ^*_{\vp}) := \{ \alpha \in L^2_{(p,q)}(\Omega , e^{-\vp})\ ;\ \ell _{\alpha}: \beta \mapsto (\dbar \beta , \alpha)_{\vp} \text{ is  bounded on }{\rm Domain}(\dbar)\}.
\]
Since ${\rm Domain}(\dbar)$ is dense in $L^2_{(p,q-1)}(\Omega , e^{-\vp})$, the linear functional $\ell _{\alpha}$ extends to a unique element of $L^2_{(p,q-1)}(\Omega , e^{-\vp})^*$.  By the Riesz Representation Theorem there is a unique $\gamma_{\alpha} \in L^2_{(p,q-1)}(\Omega , e^{-\vp})$ such that 
\[
(\beta , \gamma _{\alpha})_{\vp} = (\dbar \beta , \alpha)_{\vp}.
\]
We then define 
\[
\dbar ^*_{\vp} \alpha := \gamma _{\alpha}.
\]

A natural and important problem that arises is to characterize those smooth forms on a neighborhood of $\Omega$ that are in the domain of the Hilbert space adjoint $\dbar ^*_{\vp}$.  If we take such a smooth form $\alpha$, then it is in the domain of $\dbar ^*_{\vp}$ if and only if
\[
(\alpha , \dbar \beta)_{\vp} = (\dbar ^*_{\vp} \alpha , \beta)_{\vp}
\]
for all smooth forms $\beta$.  Indeed, since compactly supported $\beta$ are dense in $L^2_{(p,q)}(\Omega, e^{-\vp})$, the Hilbert space adjoint must act on $\alpha$ in the same way as the formal adjoint.  On the other hand, if $\beta$ does not have compact support, integration-by-parts yields 
\[
(\alpha , \dbar \beta)_{\vp} = (\dbar ^*_{\vp} \alpha , \beta)_{\vp} + \frac{(-1) ^{p-1}}{p!(q-1)!} \int _{\di \Omega} \delta ^{s\bar t} \frac{\di \rho }{\di z^s} \alpha _{I \bar t \bar J}\overline{\beta ^{I\bar J}}e^{-\vp} dS_{\di \Omega}.
\]
It follows that a smooth form $\alpha$ is in the domain of $\dbar ^*_{\vp}$ if and only if 
\begin{equation}\label{dbar-Neumann-bc}
\delta ^{s\bar t} \frac{\di \rho }{\di z^s} \alpha _{I \bar t \bar J} \equiv 0 \text{ on }\di \Omega.
\end{equation}

\begin{defn}
The boundary condition \eqref{dbar-Neumann-bc} is called the $\dbar$-Neumann boundary condition.  
\red
\end{defn}

For smooth forms satisfying the $\dbar$-Neumann boundary condition, the identity \eqref{formal-bk-cn-integrated} generalizes as the following theorem, proved by C.B. Morrey for $(0,1)$-forms, and generalized to $(p,q)$-forms by J.J. Kohn.

\begin{d-thm}\label{bk-thm-cn}
Let $\Omega$ be a domain with smooth real codimension-$1$ boundary with Levi defining function $\rho$, and let $e^{-\vp}$ be a smooth weight function.  Then for any smooth $(p,q)$-form $\alpha$ in the domain of $\dbar ^*_{\vp}$, one has the so-called Bochner-Kodaira-Morrey-Kohn identity
\begin{eqnarray} \label{bk-cn}
||\dbar ^*_{\vp}\alpha||^2_{\vp} + ||\dbar \alpha ||^2_{\vp} &=& \int _{\Omega} \sum _{k=1} ^q \delta ^{i\bar s} \frac{\di ^2\vp}{\di z^i \di \bar z ^{j_k}} \alpha _{I\bar j_1...(\bar s)_k...\bar j_q} \overline{\alpha ^{\bar I j_1...j_k...j_q}}e^{-\vp} dV +\int _{\Omega} |\bar \nabla \alpha|^2e^{-\vp} dV \\
\nonumber && \quad  + \int _{\di \Omega} \sum _{k=1} ^q \delta ^{i\bar s} \frac{\di ^2\rho}{\di z^i \di \bar z ^{j_k}} \alpha _{I\bar j_1...(\bar s)_k...\bar j_q} \overline{\alpha ^{\bar I j_1...j_k...j_q}}e^{-\vp} dS_{\di \Omega}.
\end{eqnarray}
\end{d-thm}

In addition to the Bochner-Kodaira-Morrey-Kohn identity, we need two pieces of information.  The first is the following theorem.

\begin{d-thm}[Graph norm density of smooth forms]\label{friedrichs}
The smooth forms in the domain of $\dbar ^*_{\vp}$ are dense in the space ${\rm Domain}(\dbar) \cap {\rm Domain}(\dbar ^*_{\vp})$ with respect to the graph norm 
\[
|||\alpha|||_{\vp} := ||\alpha||_{\vp} + ||\dbar \alpha||_{\vp} +||\dbar ^*_{\vp}\alpha||_{\vp}.
\]
\end{d-thm}

\begin{s-rmk}
Theorem \ref{friedrichs} relies heavily on the smoothness of the weight function $\vp$.
\red
\end{s-rmk}

Our next goal is to obtain an estimate like \eqref{basic-est-not} from Theorem \ref{bk-thm-cn} under some assumption that the Hermitian matrix
\begin{equation}\label{pos-curv-flat}
\left (\frac{\di ^2\vp}{\di z^i \di \bar z ^j}\right );
\end{equation}
for example, that it is positive definite.  It is reasonable to believe (and not hard to verify) that the first term on the right hand side of \eqref{bk-cn} can be controlled by such an assumption, and the second term is clearly non-negative.  But the third term, namely the boundary integral, can present a problem.  And indeed, it is here that the $\dbar$ Neumann boundary condition enters for a second time, to indicate the definition of Pseudoconvexity.

To understand pseudoconvexity properly, it is useful to look again at the complex structure of $\C ^n$, thought of as a real $2n$-dimensional manifold.  If $z^1,...,z^n$ are the complex coordinates in $\C ^n$ and we write 
\[
z^i = x^i + \ii y^i, \quad 1 \le i \le n
\]
then multiplication by $\ii$ induces a real linear operator $J$ that acts on real tangent vectors by 
\[
J \frac{\di}{\di x^i} = \frac{\di}{\di y^i} \quad \text{and} \quad J \frac{\di}{\di y^i} = - \frac{\di}{\di x^i}, \qquad 1 \le i \le n.
\]
To diagonalize $J$, whose eigenvalues are $\pm \ii$ with equal multiplicity (as can be seen by the fact that complex conjugation commutes with $J$), we must complexify the real vector space $T_{\C^n}$, i.e., look at $T_{\C^n} ^{\C} = T_{\C^n} \tensor _{\R} \C$.  Then we have a decomposition 
\[
T^{\C}_{\C^n} \cong T^{1,0}_{\C^n} \oplus T^{0,1}_{\C^n}
\]
where, with $\frac{\di}{\di z^i} = \frac{1}{2} \left ( \frac{\di}{\di x^i} - \ii \frac{\di}{\di y^i}\right )$ and $\frac{\di}{\di \bar z^i} = \frac{1}{2} \left ( \frac{\di}{\di x^i} + \ii \frac{\di}{\di y^i}\right )$, 
\[
T^{1,0}_{\C^n} = {\rm Span}_{\C} \left \{ \frac{\di}{\di z ^1},...,\frac{\di}{\di z^n} \right \}  \quad \text{and} \quad T^{0,1}_{\C^n} = {\rm Span}_{\C} \left \{ \frac{\di}{\di \bar z ^1},...,\frac{\di}{\di \bar z^n} \right \} .
\]
The reader can check that $J v = \ii v$ for $v \in T^{1,0}_{\C^n}$ and $Jv = - \ii v$ for $v \in T^{0,1}_{\C^n}$.

Let us now turn our attention to real tangent vectors in $\C ^n$ that are also tangent to the boundary $\di \Omega$, i.e., vectors that are annihilated by $d\rho$.  In general, given such a vector $v$, $Jv$ will not be tangent to $\di \Omega$.  In fact, if we write 
\[
T^{1,0}_{\di \Omega} := T_{\di\Omega} \cap  J T_{\di \Omega},
\]
then a computation shows that 
\[
T^{1,0}_{\di \Omega} = {\rm Ker}(\di \rho).
\]

\begin{defn}
We say that (the boundary of) $\Omega$ is pseudoconvex if the Hermitian form $\ii \di \dbar \rho$ is positive semi-definite on $T^{1,0}_{\di \Omega}$.  If this form is positive definite, we say $\Omega$ is strictly pseudoconvex.
\red
\end{defn}

It is now evident that if a smooth $(p,q)$-form $\alpha$ satisfies the $\dbar$-Neumann boundary conditions, then it takes values in 
\[
\left . \Lambda ^{p,0}_{\Omega} \right |_{\di \Omega} \wedge \Lambda ^{0,q}_{\di \Omega},
\]
where
\[
\Lambda ^{0,q}_{\di \Omega} := \Lambda ^{0,q} (T^{1,0*}_{\di \Omega}) = \underbrace{\overline{T^{1,0*}_{\di \Omega}} \wedge ... \wedge \overline{T^{1,0*}_{\di \Omega}}}_{q \text{ times}}.
\]
That is to say, the $(1,0)$-covectors are unrestricted, but the $(0,1)$-covectors must be complex-tangent to the boundary.

Putting this all together, we have the following theorem.

\begin{d-thm}[Basic Estimate]\label{basic-est-cn}
Let $\Omega$ be a domain in $\C ^n$ with pseudoconvex boundary, and let $\vp$ be a function on $\Omega$ such that the Hermitian matrix \eqref{pos-curv-flat} is uniformly positive definite at each point of $\Omega$.  Then for all $(p,q)$-forms $\alpha \in {\rm Domain}(\dbar) \cap {\rm Domain}(\dbar ^*_{\vp})$, we have the estimate 
\[
||\dbar \alpha||^2_{\vp} + ||\dbar ^*_{\vp}\alpha||^2_{\vp} \ge C ||\alpha||^2_{\vp},
\]
where $C$ is the smallest eigenvalue of \eqref{pos-curv-flat}.
\end{d-thm}

\begin{s-rmk}
The hypothesis of pseudoconvexity of $\di \Omega$ and positive definiteness of the complex Hessian of $\vp$ in Theorem \ref{basic-est-cn} is sharp when $q=1$, but can be improved when $q \ge 2$.  We will clarify this point in the next paragraph, when we look at the notion of positivity of the action induced by a Hermitian form on a vector space $V$ on the space of $(p,q)$-multilinear forms on $V$.
\red
\end{s-rmk}

\subsection{A remark:  Some simplifying algebra and geometry}

There are some geometric and algebraic insights that can make the identity \eqref{bk-cn} easier to digest.  These ideas will also make a simple transition to the geometric picture in the next paragraph.

The first, and simpler, issue, is to understand the action of the Hessian of a function on forms.  To have a coordinate-free definition of this action, we use multilinear algebra.  If we have a Hermitian form $\sh$ on a finite-dimensional Hermitian vector space $V$ of complex dimension $n$ and with "background" positive definite Hermitian form $\sa$, then the form $\sh$ acts on vector spaces obtained from $V$ by multilinear operations.  The only case we are interested in here is the case $E \tensor \Lambda ^{0,q}(V^*)$, where $E$ is a vector space on which $\sh$ acts trivially, and 
\[
\Lambda ^{0,q}(V^*) := \underbrace{\bar V^* \wedge ... \wedge \bar V^*}_{q\text{ times}},
\]
where $V^*$ is the dual vector space of $V$.  (Here with think of a Hermitian form as an element $\sB \in \Lambda ^{1,1}(V^*)$ satisfying 
\[
\sB = \overline{\sB}  \quad \text{and} \quad \left < \sB , \ii x\wedge \bar x \right > \in \R \quad \text{ for all } x \in V.
\]

Let us examine how the Hermitian form $\sh$ acts on $E \tensor \Lambda ^{0,q}(V^*)$ relative to the positive definite Hermitian form $\sa$.  First, we diagonalize $\sh$ on $V$ relative to $\sa$.  That is to say, there exist real numbers $\lambda _1\le ... \le \lambda _r$ and independent vectors $v_1,...,v_r \in V$ (where $r = \dim_{\C}(V)$) such that 
\[
\sh (v_i, v_i) = \lambda _i \sa (v_i,v_i).
\]
We can choose $v_i$ such that $\sa (v_i,v_i)=2$ for all $i$, as we assume from now on.  

\begin{defn}
We say that a Hermitian form $\sh$ is $q$-positive with respect to a background Hermitian form $\sa$ if the sum of any $q$ eigenvalues of $\sh$ with respect to $\sa$, counting multiplicity, is non-negative.  If the sum is positive, we say $\sh$ is strictly $q$-positive with respect to $\sa$.
\red
\end{defn}

Let us denote by $\alpha ^1, ..., \alpha ^r$ the basis of $V^*$ dual to $v_1,...,v_r$, i.e., $\left <v_i ,\alpha ^j\right > = \delta _i ^j$.  Writing 
\[
\alpha^{\bar J} = \bar  \alpha ^{j_1} \wedge ... \wedge \bar \alpha^{j_q},
\]
we define 
\[
\{ \sh \} e \tensor \alpha ^{\bar J} = \{ \sh \} _{\sa} e \tensor \alpha ^{\bar J} : =  (\lambda _{j_1} + ... + \lambda _{j_q}) e \tensor \alpha ^{\bar J},
\]
where $J= (j_1,...,j_q)$, and extend the action to $E \tensor \Lambda ^{0,q}(V^*)$ by linearity.

This definition of the action of $\sh$, applied to the case $V = T^{0,1}_{\Omega, z}$ or $V= T^{0,1}_{\di \Omega, z}$, and the Hermitian form $\sh = \ii \di \dbar \vp$ or $\sh = \ii \di \dbar \rho$ restricted to $T^{0,1}_{\di \Omega, z}$ respectively, shows us that in fact, the psuedoconvexity and positive definiteness hypotheses in Theorem \ref{basic-est-cn} can be replaced by the weaker assumptions of positivity of the sum of the $q$ smallest eigenvalues of $\di \dbar \vp$ (and of $\di \dbar \rho$ restricted to $T^{1,0}_{\di \Omega}$).

\begin{s-rmk}
Of course, $q$-positivity holds if and only if the sum of the $q$ smallest eigenvalues, counting multiplicity, is positive.  Thus 
\begin{enumerate}
\item[(i)]$\sh$ is $1$-positive if and only if $\sh$ is positive definite,
\item[(ii)] positive-definiteness is a stronger condition than $q$-positivity for $q \ge 2$, and 
\item[(iii)] the notion of $1$-positive is independent of the background form $\sa$, but, as soon as $q \ge 2$, $q$-positivity is not independent of $\sa$. 
\red
\end{enumerate}
\end{s-rmk}

The multi-linear algebra just discussed gives us a way to understand the Hermitian geometry of the terms in the basic identity on domains in $\C^n$, and our next task is to import this understanding to more general domains in K\"ahler manifolds, and $(p,q)$-forms with values in a holomorphic vector bundle.  To accomplish this task, we will need a notion of $\dbar$ for such sections, and also of curvature of vector bundle metrics.

Given a holomorphic vector bundle $E \to X$ or rank $r$ over a complex manifold $X$, one can associate to $E$ a $\dbar$-operator, defined as follows.  If $e_1, ... ,e_r$ is a holomorphic frame for $E$ over some open subset $U$ of $X$, then any (not necessarily holomorphic) section $f$ of $E$ can be written over $U$ as 
\[
f = f^ie_i
\]
for some functions $f^1,...,f^r$ on $U$.  We then define 
\[
\dbar f := (\dbar f^i) \tensor e_i.
\]
Because a change of frame occurs by applying an invertible matrix of holomorphic functions, the operator $\dbar$ is well-defined, and maps sections of $E$ to sections of $T^{0,1*}_X \tensor E$.

In general, we may wish to differentiate sections of $E$.  The way to do so is through a choice of a connection, i.e., a map 
\[
\nabla : \Gamma (X, E) \to \Gamma (X, T^{\C}_X\tensor E),
\]
where $T_X^{\C} := T_X\tensor _{\R} \C$ is the complexified tangent space.  In general, there is no natural choice of connection, but if the vector bundle has some additional structure, we are able to narrow down the choices significantly.  For example, if $E$ is a holomorphic vector bundle, then in view of the decomposition $T^{\C}_X = T^{1,0}_X \tensor T^{0,1}_X$ into $\C$-linear and $\overline{\C}$-linear components, we can split any connection as 
\[
\nabla = \nabla ^{1,0} +\nabla ^{0,1}.  
\]
We can then ask that $\nabla ^{0,1}= \dbar$.  If $E$ also equipped with a Hermitian metric, we can ask that our connection be compatible with the Hermitian metric in the following sense:
\[
d\left < f, g\right > (\xi) = \left < \nabla _{\xi} f , g \right > + \left < f,\nabla _{\xi}g\right>, \qquad \xi \in T_X.
\]
A connection $\nabla$ that is compatible with the metric of $E\to X$ and satisfies $\nabla ^{0,1}=\dbar$ is called a {\it Chern connection}.  The following theorem is sometimes called the {\it Fundamental Theorem of Holomorphic Hermitian Geometry}.

\begin{d-thm}
Any holomorphic Hermitian vector bundle admits a unique Chern connection.
\end{d-thm}

\begin{rmk}
The Chern connection is in particular uniquely determined for the Hermitian vector bundle $T^{1,0}_X$, the $(1,0)$-tangent bundle of a Hermitian manifold $(X,\omega)$.  Now, the assignment of a $(1,0)$-vector $\xi$ to twice its real part gives an isomorphism of $T^{1,0}_X$ with $T_X$, and this isomorphism associates to a Hermitian metric the underlying Riemannian metric.  Thus we obtain a second canonical connection for $T^{1,0}_X$, namely the Levi-\v Civita connection.  In general, these two connections are different.  They coincide if and only if the underlying Hermitian manifold is K\"ahler.
\red
\end{rmk}

We can also discuss the notion of the curvature of a connection.  Indeed, thinking of $\nabla$ is an exterior derivative, we see that in general $\nabla \nabla \neq 0$.  The miracle of geometry is that the next best thing happens:  $\nabla \nabla$ is a $0^{\text th}$ order differential operator, and this operator is called the curvature of the connection $\nabla$.  In general, locally the curvature is given, in terms of a frame for $E$, by a matrix whose coefficients are differential $2$-forms with values in $E$:  
\[
\Theta (\nabla) := \nabla \nabla : \Gamma (X, E) \to \Gamma (X, E \tensor \Lambda ^2T_X).
\]
In the case of the Chern connection, we have 
\[
\nabla ^2 = (\nabla ^{1,0})^2 + \nabla ^{1,0} \dbar + \dbar \nabla ^{1,0} + \dbar ^2  = (\nabla ^{1,0})^2 + \nabla ^{1,0} \dbar + \dbar \nabla ^{1,0},
\]
but because the Chern connection preserves the metric (i.e., it is Hermitian) we must have $(\nabla ^{1,0})^2=0$.  We therefore have the formula
\[
\Theta (\nabla) = [\nabla ^{1,0}, \dbar]
\]
for the Chern connection.  (Here, because we are using differential forms, the commutator is "graded" according to the degree of the forms.  If we choose two tangent $(1,0)$-vector fields $\xi$ and $\eta$, then we have 
\[
\Theta (\nabla) (\xi , \eta)= [\nabla ^{1,0}_{\xi}, \dbar _{\bar \eta}]
\]
where now the commutator is the usual one.)

The example of the trivial line bundle $p_1: \Omega \times \C\to \Omega$ (where $p_1$ is projection to the first factor) with non-trivial metric $e^{-\vp}$ is already interesting.  In this case, a section is the graph of a function $f :\Omega \to \C$, and the metric for the trivial line bundle is given by 
\[
\left < f, g\right > (z)  := f(z) \overline{g(z)} e^{-\vp(z)}, \quad z\in \Omega.
\]
The $\dbar$-operator just corresponds to the usual $\dbar$ operator on functions.  Imposing metric compatibility gives 
\[
d\left < f,g\right > = df \bar g e^{-\vp} + f \overline{dg} e^{-\vp} + f\bar g (-d\vp)e^{-\vp} = (\di f - \di \vp f +\dbar f)\bar g e^{-\vp} + f\overline{(\di g -\di \vp g - \dbar g)}e^{-\vp},
\]
so the Chern connection is given by 
\[
\nabla ^{1,0} f = \di f - f \di \vp.
\]
The curvature of this connection is 
\[
\Theta (\nabla) f = \di (\dbar f) - \di \vp \wedge \dbar f + \dbar (\di f - \di \vp f) = (\di \dbar +\dbar \di) f - \di \vp \wedge \dbar f - \dbar f \wedge \di \vp - f \dbar \di \vp= (\di \dbar \vp)f,
\]
i.e., multiplication by the complex Hessian of $\vp$.  This is precisely the Hermitian form in \eqref{formal-bk-cn-integrated} and \eqref{bk-cn}.  (The action of this Hermitian form on $(p,q)$-forms was discussed in the first part of this paragraph.)

Finally, let us begin to clarify the remarks we made earlier regarding $(p,q)$-forms.  We note that while $\Lambda ^{p,q}_X \to X$ is in general not a holomorphic vector bundle, it is indeed holomorphic when $q=0$.  Writing  
\[
\Lambda ^{p,q}_X = \Lambda ^{p,0}_X \tensor \Lambda ^{0,q}_X,
\]
we can therefore treat $(p,q)$-forms as $(0,q)$-forms with values in the holomorphic vector bundle $\Lambda ^{p,0}_X$.  More generally, if $E \to X$ is a holomorphic vector bundle, then we can treat $E$-valued $(p,q)$-forms as $E \tensor \Lambda ^{p,0}_X$-valued $(0,q)$-forms on $X$.  For this reason, it is often advantageous to assume $p=0$.  

\begin{s-rmk}
Interestingly, there is also an advantage to assuming $p=n$.  We shall return to this point after we write down a geometric generalization of Theorem \ref{bk-thm-cn}, which is our next task.
\red
\end{s-rmk}

\subsection{$(p,q)$-forms with values in a holomorphic vector bundle over domains in a K\"ahler manifold}\label{SS:kahler}

Let $(X,\omega)$ be a K\"ahler manifold and let $\Omega \relcomp X$ be an open set whose boundary is a possibly empty, smooth compact real hypersurface in $X$.  Let us write 
\[
dV_{\omega} := \frac{\omega ^n}{n!},
\]
where $n$ is the complex dimension of $X$.  We assume there is a smooth, real-valued function $\rho$ defined on a neighborhood $U$ of $\di \Omega$ in $X$ such that 
\[
\Omega = \{ z \in U\ ;\ \rho (z) < 0\}, \quad \di \Omega = \{ z \in U\ ;\ \rho (z) = 0\}
\quad \text{and} \quad |\di \rho |_{\omega} \equiv 1 \text{ on }\di \Omega.
\]
Suppose also that there is a holomorphic vector bundle $E \to \overline{\Omega}$ with (smooth\footnote{There are a few notions of singular Hermitian metrics for vector bundles of higher rank, but the theory of singular Hermitian metrics, while rather developed for line bundles, is much less developed in the higher rank case.}) Hermitian metric $h$.  With this data, we can define a pointwise inner product on $E$-valued $(0,q)$-forms, which we denote 
\[
\left <\alpha ,\beta \right > _{\omega, h}.
\]
To give this notion a more concrete meaning, let us choose a frame $e_1,...,e_r$ for $E$ and local coordinates $z$.  If we write 
\[
\alpha = \alpha ^i_{\bar J} e_i \tensor d\bar z^J \quad \text{and} \quad \alpha = \beta ^i_{\bar J} e_i \tensor d\bar z^J,
\]
as well as 
\[
h_{i\bar j} := h(e_i, e_j) \quad \text{and} \quad \omega _{i\bar j} = \omega (\tfrac{\di}{\di z^i}, \tfrac{\di}{\di \bar z ^j}),
\]
then 
\[
\left <\alpha ,\beta \right > _{\omega, h} = \alpha _{\bar I} ^i \overline{\beta ^j _{\bar J}} h_{i\bar j} \omega ^{\bar J I},
\]
where $\omega ^{\bar j i}$ is the inverse transpose of $\omega _{i\bar j}$ and $\omega ^{\bar J I} = \omega ^{\bar j_1 i_i} \cdots \omega ^{\bar j_q i_q}$.  This pointwise inner product is easily seen to be globally defined, and as in the case of domains in $\C^n$, it induces an $L^2$ inner product on smooth $E$-valued $(0,q)$-forms by the formula
\[
(\alpha , \beta )_{\omega , h} := \int _{\Omega} \left <\alpha ,\beta \right > _{\omega, h} dV_{\omega}.
\]
If we carry out the natural analogues of the ideas from Paragraph \ref{domains-par}, we are led to define the domains of $\dbar$ and\footnote{Although the definition of $\dbar ^*_h$ also depends on $\omega$, we omit this dependence from the notation to keep things manageable.} $\dbar ^*_{h}$, the latter being given on smooth $E$-valued $(0,q)$-forms by the formula
\begin{equation}\label{formal-adjoint-vb}
(\dbar ^*_h \alpha)^i_{\bar J} = - h^{i\bar \ell} \omega _{M\bar J} \frac{\di}{\di z^k} \left ( h_{m \bar \ell} \alpha ^m_{\bar j L}\omega ^{\bar j k}\omega ^{\bar L M} \det(\omega) \right )  .
\end{equation}
We also have the analogue of the $\dbar$-Neumann boundary condition 
\begin{equation}\label{dbar-Neumann-bc-gen}
\omega ^{s\bar t} \frac{\di \rho}{\di z^s} \alpha^i _{\bar t \bar J} \equiv 0 \quad \text{on }\di \Omega.
\end{equation}
The result for compactly supported forms (which in this setting can be useful if one is working on compact K\"ahler manifolds) goes through in the same way, as does the modification to manifolds with boundary introduced by Morrey and Kohn.  To keep things brief, we content ourselves with stating the theorem that results.

\begin{d-thm}[Bochner-Kodaira-Morrey-Kohn Identity]\label{bk-id-thm-gen}
Let $(X,\omega)$ be a K\"ahler manifold and $E \to X$ a holomorphic vector bundle with Hermitian metric $h$.  Denote by ${\rm Ricci}(\omega)$ the Ricci curvature of $\omega$, and by $\Theta (h)$ the curvature of the Chern connection for $E$.  Then for any smooth $E$-valued $(0,q)$-form $\alpha$ in the domain of $\dbar ^*_h$, i.e., satisfying the $\dbar$-Neumann boundary condition \eqref{dbar-Neumann-bc-gen}, one has the identity
\begin{eqnarray} \label{bk-gen}
||\dbar ^*_{h}\alpha||^2_{h,\omega} + ||\dbar \alpha ||^2_{h,\omega} &=& \int _{\Omega} \sum _{k=1} ^q \omega  ^{i\bar s} (\Theta (h) + {\rm Ricci}(\omega))_{i,\bar j_k} \alpha ^i_{\bar j_1...(\bar s)_k...\bar j_q} \overline{\alpha_{\bar i} ^{j_1...j_k...j_q}} dV_{\omega}  \\
\nonumber && \quad  +\int _{\Omega} |\bar \nabla \alpha|^2_{\omega ,h} dV_{\omega} \\
\nonumber && \qquad + \int _{\di \Omega} \sum _{k=1} ^q \omega ^{i\bar s} \frac{\di ^2\rho}{\di z^i \di \bar z ^{j_k}} \alpha^i _{\bar j_1...(\bar s)_k...\bar j_q} \overline{\alpha_i ^{j_1...j_k...j_q}} dS_{\omega, \di \Omega}.
\end{eqnarray}
\end{d-thm}

\begin{rmks}\label{ricci-rmk}
A couple of remarks are in order.  
\begin{enumerate}
\item[(i)]  Some of the indices look to be in the wrong place; they are superscripts when they should be subscripts, or vice versa.  This is the standard notation for contraction (which is also called {\it raising/lowering}) with the relevant metric.

\item[(ii)] Although the Ricci curvature of a Riemannian metric is a well-known quantity, the Ricci curvature of a K\"ahler metric is even simpler.  A K\"ahler form $\omega$ induces a volume form $dV_{\omega}$, which can be seen as a metric for the the anticanonical bundle 
\[
-K_X = \det T^{1,0}_X.
\]
The curvature of this metric is precisely the Ricci curvature of $\omega$.  It is therefore given by the formula 
\[
{\rm Ricci}(\omega) = - \ii \di \dbar \log \det (\omega _{i\bar j}).
\]
The reader can check that ${\rm Ricci}(\omega)$ is independent of the choice of local coordinates.
\red
\end{enumerate}
\end{rmks}

Finally, we come to the statement made at the end of the previous paragraph, regarding the convenience of using $(n,q)$-forms instead of $(0,q)$-forms.  An $E$-valued $(n,q)$-form can be seen as an $E \tensor K_X$-valued $(0,q)$-form.  Locally, we can write such a form as 
\[
\alpha = \alpha ^i e_i \tensor dz^1 \wedge ... \wedge dz^n.
\]
We can use he metric $\frac{h}{\det \omega}$ for $E \tensor K_X$ to compute that 
\[
\left <\alpha ,\beta\right >_{\omega, \frac{h}{\det \omega}} = h_{i\bar j} \alpha ^i \overline{\beta ^j} \frac{\ii ^n}{2^n} dz ^1 \wedge d\bar z^1 \wedge ... \wedge dz ^n \wedge d\bar z^n,
\]
which is a complex measure on $\Omega$, and can thus be integrated without reference to a volume form.

Now, the metric $\frac{h}{\det \omega}$ has curvature
\[
\Theta (h) - {\rm Ricci}(\omega),
\]
and the second term cancels out the Ricci curvature in \eqref{bk-gen}.  We therefore get the following restatement of Theorem \ref{bk-id-thm-gen}

\begin{d-thm}[Bochner-Kodaira-Morrey-Kohn Identity for $(n,q)$-forms]\label{bk-id-thm-nq}
Let $(X,\omega)$ be a K\"ahler manifold and $E \to X$ a holomorphic vector bundle with Hermitian metric $h$.  Denote by $\Theta (h)$ the curvature of the Chern connection for $E$.  Then for any smooth $E$-valued $(n,q)$-form $\alpha$ in the domain of $\dbar ^*_h$, i.e., satisfying the $\dbar$-Neumann boundary condition \eqref{dbar-Neumann-bc-gen}, one has the identity
\begin{eqnarray} \label{bk-gen}
||\dbar ^*_{h}\alpha||^2_{h,\omega} + ||\dbar \alpha ||^2_{h,\omega} &=& \int _{\Omega} \sum _{k=1} ^q \omega  ^{i\bar s} \Theta (h) _{i,\bar j_k} \alpha ^i_{\bar j_1...(\bar s)_k...\bar j_q}\wedge  \overline{\alpha_{\bar i} ^{j_1...j_k...j_q}} \\
\nonumber && \quad  +\int _{\Omega} |\bar \nabla \alpha|^2_{\omega ,\frac{h}{\det \omega}} dV_{\omega} \\
\nonumber && \qquad + \int _{\di \Omega} \sum _{k=1} ^q \omega ^{i\bar s} \frac{\di ^2\rho}{\di z^i \di \bar z ^{j_k}} \alpha^i _{\bar j_1...(\bar s)_k...\bar j_q} \overline{\alpha_i ^{j_1...j_k...j_q}}\frac{dS_{\omega, \di \Omega}}{\det \omega}.
\end{eqnarray}
\end{d-thm}

The notion of pseudoconvexity goes over to the case of domains in K\"ahler manifolds without change, but we can also introduce the notion of $q$-positive domains in a K\"ahler manifold.  Although the notion of $q$-positivity can be defined for a vector bundle, we restrict ourselves to the case of line bundles, since this is the main situation we will be interested in.

\begin{defn}
Let $(X,\omega)$ be a K\"ahler manifold.
\begin{enumerate}
\item[(i)]A smoothly bounded domain $\Omega \relcomp X$ with defining function $\rho$ is said to be $q$-positive if the Hermitian form $\ii \di \dbar \rho$, restricted to $T^{1,0}_{\Omega}$, is $q$-positive with respect to $\omega$ restricted to $T^{1,0}_{\Omega}$.
\item[(ii)]  Let $L \to X$ be a holomorphic line bundle.  We say that a Hermitian metric for $L \to X$ is $q$-positively curved with respect to $\omega$ if the Chern curvature $\ii \di \dbar \vp$ is $q$-positive with respect to $\omega$.
\red
\end{enumerate}
\end{defn}

\begin{s-rmk}
Although the notion of $q$-positive domain does depend on the ambient K\"ahler metric $\omega$, it does not depend on the choice of defining function $\rho$, as the reader can easily verify.
\red
\end{s-rmk}

With these notions in hand, we can now obtain a generalization of Theorem \ref{basic-est-cn} to the setting of domains in K\"ahler manifolds.  Again we will stick to $(0,q)$-forms with values in a line bundle.  

\begin{d-thm}[Basic Estimate]\label{basic-est-gen}
Let $\Omega$ be a smoothly bounded relatively compact domain in a K\"ahler manifold $(X,\omega)$, and assume the boundary of $\Omega$ is $q$-positive with respect to $\omega$.  Let $L \to X$ be a holomorphic line bundle with smooth Hermitian metric $e^{-\vp}$ such that the Hermitian form 
\[
\ii \left (\di \dbar \vp +{\rm Ricci}(\omega)\right )
\]
is uniformly strictly $q$-positive with respect to $\omega$.  Then for all $L$-valued $(0,q)$-forms $\alpha \in {\rm Domain}(\dbar) \cap {\rm Domain}(\dbar ^*_{\vp})$, we have the estimate 
\[
||\dbar \alpha||^2_{\omega,\vp} + ||\dbar ^*_{\vp}\alpha||^2_{\omega,\vp} \ge C ||\alpha||^2_{\omega, \vp},
\]
where $C$ is infimum over $\Omega$ of the sum of the $q$ smallest eigenvalue of $\ii \di \dbar \vp +\ii {\rm Ricci}(\omega)$ with respect to $\omega$.
\end{d-thm}

\section{The twisted Bochner-Kodaira-Morrey-Kohn Identity}\label{twisted-prelims-section}

\subsection{The Identity; two versions}

We stay in the setting of a K\"ahler manifold $(X,\omega)$ and a domain $\Omega \relcomp X$ with $q$-positive boundary.   Suppose we have a holomorphic line bundle $L \to X$ with Hermitian metric $e^{-\vp}$.  Let us split the metric into a product 
\[
e^{-\vp} = \tau e^{-\psi}
\]
where $\tau$ is a positive function and thus $e^{-\psi}$ is also a metric for $L$.   In view of the formula \eqref{formal-adjoint-vb} we have

\[
\dbar ^*_{\vp} \alpha = \dbar ^*_{\psi}\alpha - \tau ^{-1} \gradbar (\tau)
\]
where $\gradbar(\tau)$ is the $(0,1)$-vector field defined by 
\[
\overline{\di \tau (\bar \xi)} = \omega (\xi, \gradbar (\tau)), \quad \xi \in T^{0,1}_{\Omega}.
\]
We also have the curvature formula 
\[
\di \dbar \vp = \di \dbar \psi - \frac{\di \dbar \tau}{\tau} + \frac{\di \tau \wedge \dbar \tau}{\tau ^2}.
\]
Substitution into \eqref{bk-gen} yields the following theorem, in which we use the more global, and somewhat more suggestive, notation than that used in \eqref{bk-gen}.

\begin{d-thm}[Twisted Bochner-Kodaira-Morrey-Kohn Identity]\label{tbk-thm}
For all smooth $L$-valued $(0,q)$-forms in the domain of $\dbar ^*_{\psi}$,  one has the identity
\begin{eqnarray}\label{twisted-bk-id}
\nonumber&& \int _{\Omega} \tau |\dbar ^*_{\psi} \beta|^2 _{\omega} e^{-\psi} dV_{\omega} + \int _{\Omega} \tau |\dbar \beta|^2 _{\omega} e^{-\psi}dV_{\omega} \\
&& \qquad = \int _{\Omega} \left < \{ \tau \ii (\di \dbar \vp + {\rm Ricci}(\omega) )  - \ii \di \dbar \tau \}\beta, \beta \right > _{\omega} e^{-\psi}dV_{\omega} \\
\nonumber && \qquad \qquad + \int _{\Omega} \tau |\overline{\nabla} \beta|^2 _{\omega} e^{-\psi}dV_{\omega}  + \int _{\di \Omega}  \left < \{ \tau \ii\di \dbar \rho\} \beta, \beta \right > _{\omega} e^{-\psi} dS _{\di \Omega}\\
\nonumber && \qquad \qquad \qquad + 2 \re \int _{\Omega} \left < \dbar ^* _{\psi} \beta , \gradbar \tau \rfloor \beta \right > _{\omega} e^{-\psi} dV_{\Omega}.
\end{eqnarray}
\end{d-thm}

On the other hand, we can also apply integration by parts to the term 
\[
\int _{\Omega} \left < \dbar ^* _{\psi} \beta , \gradbar \tau \rfloor \beta \right > _{\omega} e^{-\psi} dV_{\Omega},
\]
to obtain 
\[
\int _{\Omega} \left < \dbar ^* _{\psi} \beta , \gradbar \tau \rfloor \beta \right > _{\omega} e^{-\psi} dV_{\Omega} = -\int _{\Omega} \tau \left < \dbar \dbar ^*_{\psi} \beta , \beta \right >_{\omega}e^{-\psi}dV_{\omega} +\int _{\Omega}\tau |\dbar ^*_{\psi}\beta|^2_{\omega}e^{-\psi}dV_{\omega}.
\]
Substitution into \eqref{twisted-bk-id} yields the following theorem, due to Berndtsson.

\begin{d-thm}[Dual version of the Twisted Bochner-Kodaira-Morrey-Kohn Identity]\label{tbk-thm-2}
For all smooth $L$-valued $(0,q)$-forms $\beta$ in the domain of $\dbar ^*_{\psi}$,  one has the identity
\begin{eqnarray}\label{twisted-bk-id-bo}
\nonumber&& 2\re \int _{\Omega} \tau \left < \dbar \dbar ^*_{\psi} \beta , \beta \right >_{\omega}e^{-\psi}dV_{\omega}+  \int _{\Omega} \tau |\dbar \beta|^2 _{\omega} e^{-\psi}dV_{\omega} \\
&& \qquad = \int _{\Omega} \tau |\dbar ^*_{\psi} \beta|^2 _{\omega} e^{-\psi} dV_{\omega} + \int _{\Omega} \left < \{ \tau \ii (\di \dbar \vp + {\rm Ricci}(\omega) )  - \ii \di \dbar \tau \}\beta, \beta \right > _{\omega} e^{-\psi}dV_{\omega} \\
\nonumber && \qquad \qquad + \int _{\Omega} \tau |\overline{\nabla} \beta|^2 _{\omega} e^{-\psi}dV_{\omega}  + \int _{\di \Omega}  \left < \{\tau\ii \di \dbar \rho\} \beta, \beta \right > _{\omega} e^{-\psi} dS _{\di \Omega}.
\end{eqnarray}
\end{d-thm}

\subsection{Twisted basic estimate}

By applying the Cauchy-Schwarz Inequality to the second integral on the right hand side of \eqref{twisted-bk-id}, followed by the inequality $ab \le Aa^2+A^{-1}b^2$, one obtains
\[
2\re  \left < \dbar ^* _{\psi} \beta , \gradbar \tau \rfloor \beta \right > _{\omega} \le A |\dbar ^*_{\psi}\beta |^2_{\omega} + A^{-1} \left < \{ \ii \di \tau \wedge \dbar \tau\} \beta ,\beta\right > _{\omega}.
\] 
Thus, the following inequality holds:

\begin{d-thm}[Twisted Basic Estimate] \label{tbe}
Let $(X,\omega)$ be a K\"ahler manifold and let $L \to X$ be a holomorphic line bundle with smooth Hermitian metric $e^{-\psi}$.  Fix a smoothly bounded domain $\Omega \relcomp X$ such that $\di \Omega$ is pseudoconvex. Let $A$ and $\tau$ be positive functions on a neighborhood of $\overline{\Omega}$ with $\tau$ smooth.    Then for any smooth $L$-valued $(0,q)$-form $\beta$ in the domain of $\dbar ^*_{\psi}$ one has the estimate 
\begin{eqnarray}\label{tbest}
\nonumber && \int _{\Omega} (\tau +A)|\dbar ^*_{\psi} \beta |^2_{\omega} e^{-\psi} dV_{\omega} + \int _{\Omega} \tau|\dbar  \beta |^2_{\omega} e^{-\psi} dV_{\omega} \\
&& \qquad \ge \int _{\Omega} \left <\left \{ \ii \left (\tau (\di \dbar \psi + {\rm Ricci}(\omega) )  - \di \dbar \tau - A^{-1}\di \tau \wedge \dbar \tau \right ) \right  \}\beta, \beta \right > _{\omega} e^{-\psi}dV_{\omega}.
\end{eqnarray}
\end{d-thm}

\subsection{A posteriori estimate}
An application of the big constant-small constant inequality to the left-most term of identity \eqref{twisted-bk-id-bo} in Theorem \ref{tbk-thm-2} yields the following estimate.

\begin{d-thm}[A posteriori estimate]\label{Berndtsson-est}
Let $(X,\omega)$ be a K\"ahler manifold and let $L \to X$ be a holomorphic line bundle with smooth Hermitian metric $e^{-\psi}$.  Fix a smoothly bounded domain $\Omega \relcomp X$ such that $\di \Omega$ is pseudoconvex. Let $\tau$ be a smooth positive function on a neighborhood of $\overline{\Omega}$, and let $\ii \Theta$ be a non-negative Hermitian $(1,1)$-form that is strictly positive almost everywhere.   Then for any smooth $L$-valued $(0,q+1)$-form $\beta$ in the domain of $\dbar ^*_{\psi}$ one has the estimate 
\begin{eqnarray}\label{twisted-apest}
\nonumber &&  \int _{\Omega} \tau |\dbar \dbar ^*_{\psi} \beta|^2_{\Theta,\omega}e^{-\psi}dV_{\omega}+  \int _{\Omega} \tau |\dbar \beta|^2 _{\omega} e^{-\psi}dV_{\omega} \\
&& \qquad \ge \int _{\Omega} \tau |\dbar ^*_{\psi} \beta|^2 _{\omega} e^{-\psi} dV_{\omega} + \int _{\Omega} \left <\left \{  \ii \left ( \tau (\di \dbar \psi + {\rm Ricci}(\omega) -\Theta )  - \di \dbar \tau \right ) \right \}\beta, \beta \right > _{\omega} e^{-\psi}dV_{\omega} \\
\nonumber && \qquad \qquad + \int _{\Omega} \tau |\overline{\nabla} \beta|^2 _{\omega} e^{-\psi}dV_{\omega}  + \int _{\di \Omega}  \left < \{\tau\ii \di \dbar \rho\} \beta, \beta \right > _{\omega} e^{-\psi} dS _{\di \Omega}.
\end{eqnarray}
\end{d-thm}

\begin{rmk}
The norm $|\cdot |_{\Theta , \omega}$ appearing in \eqref{twisted-bk-id-bo} requires a little explanation.  The Hermitian matrix $\ii \Theta$ can be seen as a metric for $X$ (almost everywhere), and this metric induces a metric on $(0,q)$-forms.  This is the metric appearing in the second term on the second line of \eqref{twisted-bk-id-bo}.  Since the inequality is obtain from an application of Cauchy-Schwarz and the big constant/small constant inequality, the metric $|\cdot |_{\Theta , \omega}$ corresponds to the "inverse metric".  However, the inverse transpose of the matrix of $\Omega$ produces a $(1,1)$-vector.  To identify this vector with a $(1,1)$-form, we must lower the indices, which we do using the form $\omega$.  If we write the resulting $(1,1)$-form as $\Theta ^{-1}_{\omega}$, then 
\[
|\alpha |^2_{\Theta , \omega} := \left < \{\ii \Theta ^{-1}_{\omega}\}\beta , \beta\right >_{\omega}.
\]
Note that if $\alpha$ is a $(0,1)$-form, $|\alpha |^2_{\Theta , \omega}= |\alpha |^2_{\Theta}$, but for $q\ge 2$, the two norms are different.

Analogous statements hold for $(0,q)$-forms with values in a vector bundle.  In that case, we denote the resulting metric 
\[
|\alpha|^2_{\Theta , \omega ; h},
\]
noting that when the vector bundle has rank $1$ and $h = e^{-\vp}$, then $|\alpha|^2_{\Theta , \omega ; h} = |\alpha|^2_{\Theta , \omega}e^{-\vp}$.
\red
\end{rmk}

\newpage

\begin{center}
{\large {\sc Part II.  Applications}}
\end{center}

\section{$\dbar$ theorems}

Our next goal is to exploit the various basic estimates we have established so far.  We begin with deriving a H\"ormander-type estimate from \eqref{bk-gen}, and then proceed to introduce twists and obtain other types of estimates in two ways.  One method uses the {\it a priori} twisted basic estimate \eqref{tbest}, while a second method combines Kohn's work on the $\dbar$-Neumann problem with the {\it a posteriori} estimate \eqref{twisted-apest}.  Finally, we will discuss some examples showing what sorts of improvements one obtains from the twisting techniques.

\subsection{H\" ormander-type}

The first result we present, which has seen an enormous number of applications in complex analysis and geometry, is the so-called H\"ormander Theorem.  The statement is as follows.

\begin{d-thm}[H\"ormander, Andreotti-Vesentini, Skoda]\label{h-type-thm}
Let $(X,\omega)$ be a K\"ahler manifold of complex dimension $n$, and $E \to X$ a holomorphic vector bundle with Hermitian metric $h$.  Fix $q \in \{1,...,n\}$.  Let $\Omega \relcomp X$ be a smoothly bounded domain whose boundary is $q$-positive with respect to $\omega$.  Assume there is a $(1,1)$-form $\Xi$ on $X$ such that such that 
\[
\ii(\Theta (h) +{\rm Ricci}(\omega))- \Xi 
\]
is $q$-positive with respect to $\omega$.  Then for any $E$-valued $(0,q)$-form $f$ on $\Omega$ such that 
\[
\dbar f = 0 \quad \text{and} \quad \int _{\Omega} |f|^2_{\Xi, \omega ; h} dV_{\omega} < +\infty
\]
there exists an $E$-valued $(0,q-1)$-form $u$ such that 
\[
\dbar u = f \quad \text{and} \quad \int _{\Omega} |u|^2_{\omega, h} dV_{\omega} \le \int _{\Omega} |f|^2_{\Xi, \omega ; h} dV_{\omega}.
\]
\end{d-thm}

\begin{rmk}
When $\Xi = c \omega$ for some positive constant $c$, Theorem \ref{h-type-thm} was proved independently and almost simultaneously by Andreotti-Vesentini and H\"ormander.  The general case is due to Skoda.
\red
\end{rmk}

\begin{proof}[Proof of Theorem \ref{h-type-thm}]
The standard proof of Theorem \ref{h-type-thm} uses the so-called Lax-Milgram lemma, but we will give an analogous, though less standard, proof that passes through the $\dbar$-Laplace-Beltrami operator.  

To this end, let us define the Hilbert space 
\[
L^2_q(\Xi) := \left \{\alpha \text{ mearsurable }(0,q)-form\ ;\ \int _{\Omega} \left < \Xi \alpha, \alpha \right >_{\omega, h} dV_{\omega} < +\infty \right \},
\]
We have a vector subspace $\sh _q \subset L^2_{q}(\Xi)$ defined by 
\[
\sh_q = {\rm Domain}(\dbar) \cap {\rm Domain}(\dbar ^*_h).
\]
Because smooth forms are dense in the graph norm, the space $\sh_q$ becomes a Hilbert space with respect to the norm 
\[
||\alpha||_{\sh_q}:= \left ( ||\dbar \alpha||^2 _{h,\omega} +||\dbar \alpha||^2 _{h,\omega}\right )^{1/2}.
\]
In view of Theorem \ref{bk-id-thm-gen} and the hypotheses of Theorem \ref{h-type-thm} the inclusion of Hilbert spaces 
\[
\iota : \sh_q \emb L^2_{0,q}(\Xi)
\]
is a bounded linear operator.

Now let $f$ be an $E$-valued $(0,q)$-form satisfying the hypotheses of Theorem \ref{h-type-thm}.  Define the bounded linear functional $\lambda _f \in L^2_q(\Xi)^*$ by 
\[
\lambda_f (\alpha) := (\alpha , f )_{\omega, h}.
\]
We already know that $\lambda_f \in \sh _q ^*$, but the estimate 
\[
|\lambda _f (\alpha)|^2 \le \left ( \int _{\Omega} |f|^2_{\Xi, \omega ; h} dV_{\omega} \right ) \left ( \int _{\Omega} \left < \Xi \alpha, \alpha \right >_{\omega, h} dV_{\omega}\right ) \le \left ( \int _{\Omega} |f|^2_{\Xi, \omega ; h} dV_{\omega} \right ) ||\alpha ||_{\sh _q}^2
\]
tells us that 
\[
||\lambda _f||^2_{\sh_q ^*}\le  \int _{\Omega} |f|^2_{\Xi, \omega ; h} dV_{\omega} .
\]
By the Riesz Representation Theorem there exists $\beta \in \sh_q$ representing $\lambda _f$, which is to say 
\[
( \dbar \alpha  , \dbar \beta)_{\omega , h} + (\dbar ^*_h \alpha , \dbar ^*_h \beta)_{\omega, h} = (\alpha , f)_{\omega , h}, \quad \alpha \in \sh _q.
\]
The last identity defines the notion of a weak solution $\beta$ to the equation 
\[
(\dbar ^*_h \dbar +\dbar \dbar ^*_h) \beta = f.
\]
Thus we have found a weak solution $\beta$ satisfying the estimate
\[
\int _{\Omega} \left < \Xi \beta , \beta \right >_{\omega , h} dV_{\omega} \le ||\beta||^2_{\sh_q} \le \int _{\Omega} |f|^2_{\Xi, \omega ; h} dV_{\omega},
\]
where in the first inequality we have used the appropriate modification of the basic estimate adapted to $\Xi$.  Notice that this is the case even if $\dbar f $ does not vanish identically.

Finally, assume that $\dbar f = 0$.  Then $f$ is orthogonal to the image of $\dbar ^*_h$, and thus we have 
\[
0 = (\dbar ^*_h \dbar \beta , f) = ( \dbar \dbar ^*_h \dbar \beta  , \dbar \beta)_{\omega , h} + (\dbar ^*_h \dbar ^*_h \dbar \beta , \dbar ^*_h \beta)_{\omega, h} = ||\dbar ^*_h \dbar \beta||^2_{\omega , h},
\]
and therefore 
\[
(\dbar ^*_h \alpha , \dbar ^*_h \beta)_{\omega, h} = (\alpha , f)_{\omega , h}.
\]
But the latter precisely says that the $(0,q-1)$-form 
\[
u := \dbar ^*_h \beta
\]
is a weak solution of  the equation $\dbar u = f$.  Moreover, since $\dbar ^*_h \dbar \beta = 0$, we have 
\[
||\dbar \beta ||^2_{\omega,h}= (\dbar ^*_h \dbar \beta ,\beta )_{\omega,h} = 0,
\]
and therefore we obtain the estimate 
\[
||u||^2_{\omega , h} = ||\dbar ^*_h\beta||^2_{\omega,h} = ||\dbar ^*_h\beta||^2_{\omega,h}+||\dbar \beta||^2_{\omega,h} \le \int _{\Omega} |f|^2_{\Xi, \omega ; h} dV_{\omega},
\]
thus completing the proof.
\end{proof}

\begin{rmk}[Regularity of the Kohn Solution]
Before moving on, let us make a few remarks on the solution $u$ obtained in the proof of Theorem \ref{h-type-thm}.  This solution was of the form $u = \dbar ^*_h \beta$ for some $(0,q)$-form $\beta$.  Since any two solutions of the equation $\dbar u = f$ differ by a $\dbar$-closed $E$-valued $(0,q)$-form, the solution $u$ is actually the one of minimal norm.  Indeed, it is clearly orthogonal to all $\dbar$-closed $E$-valued $(0,q)$-forms.

This solution, being minimal, is unique, and is known as the Kohn solution.  It was shown by Kohn that if furthermore the boundary of $\Omega$ is strictly pseudoconvex, then $u$ is smooth up to the boundary  if this is the case for $f$.  
We shall use this fact in the second twisted method below.

Finally, we should note that the approach of H\"ormander, namely using the Lax-Milgram Lemma instead of passing through solutions of the $\dbar$-Laplace-Beltrami operator, does not necessarily produce the minimal solution, but it does produce a solution with the same estimate.  Therefore, this estimate also bounds the minimal solution, so that the outcome of the two methods is the same, as far as existence and estimates of weak solutions is concerned.
\red
\end{rmk}

\subsection{Singular Hermitian metrics}

In many problems in analysis and geometry, there is much to be gained be relaxing the definition of Hermitian metric for a holomorphic line bundle.  Let us discuss this more general notion, often called a {\it singular Hermitian metric} (though perhaps the name {\it possibly singular Hermitian metric} is more appropriate).

\begin{defn}
Let $X$ be a complex manifold and $L \to X$ a holomorphic line bundle.  A {\it possibly singular Hermitian metric} is a measurable section $h$ of the the line bundle $L^* \tensor \overline{L^*} \to X$ that is symmetric and positive definite almost everywhere, and with the additional property that, for any nowhere-zero smooth section $\xi$ of $L$ on an open subset $U$, the function 
\[
\vp ^{(\xi)} := - \log h(\xi,\overline{\xi}) 
\]
is upper semi-continuous and lies in $L^1_{\ell oc}(U)$.  In particular, if $\xi$ is holomorphic, the $(1,1)$-current 
\[
\Theta _h := \di \dbar \vp ^{(\xi)}
\]
is called the curvature current of $h = e^{-\vp}$, and it is independent of the section $\xi$.
\red
\end{defn}

If $\Theta_h$ is non-negative, then the local functions $\vp^{(\xi)}$ are plurisubharmonic.  More generally, if $\Theta _h$ is bounded below by a smooth $(1,1)$-form then the local functions are quasi-plurisubharmonic, i.e., a sum of a smooth function and a plurisubharmonic function.  Thus possibly singular Hermitian metrics are subject to the results of pluripotential theory, including regularization.  If one can regularize a singular Hermitian metric in the right way, then many of the results we have stated, and will state, can be extended to the singular case.

We shall not be too precise about this point here; it is well-made in many other articles and texts, and though it is fundamental, focusing on it will take us away from the main goal of the article.  Suffice it to say that there are good regularizations available on the following kinds of spaces:
\begin{enumerate}
\item[(i)] Stein manifolds, i.e., properly embedded submanifolds of $\C^N$,
\item[(ii)] Projective manifolds, and 
\item[(iiI)] manifolds with the property that there is a hypersurface whose complement is Stein.
\end{enumerate}
We should mention that the recent resolution of the openness conjecture by Guan and Zhou \cite{gz-openness} has opened the door to new types of approximation techniques that we will not have time to go into here.  An interesting example can be found in \cite{cao}.  Though the strong openness conjecture deserves a more elaborate treatment, we have to make some hard choices of things to leave out, lest this article continue to grow unboundedly.  The reader should consult any of a number of articles on this important topic, including for example \cite{bo-oc,lempert,hiep} and references therein.

\subsection{Twisted estimates: method I}

A look at the twisted basic estimate \eqref{tbest} shows that there are two positive functions we must choose, namely $\tau$ and $A$ (with $\tau$ smooth).  In this section, we will always assume that $A = \frac{\tau}{\delta}$ for some constant $\delta$.  With this choice, the twisted basic estimate becomes 
\begin{eqnarray}
\nonumber &&\frac{1+\delta}{\delta} \int _{\Omega}\tau |\dbar ^*_{\psi}\beta |_{\omega} ^2 e^{-\psi} dV_{\omega} + \int _{\Omega} \tau |\dbar \beta |^2_{\omega} e^{-\psi} dV_{\omega}\\
\label{bmo-est} &&\qquad  \qquad  \ge \int _{\Omega} \left < \{ \tau \ii (\di \dbar \psi +{\rm Ricci}(\omega) )- \ii \di \dbar \tau - \delta \ii \tau^{-1} \di \tau \wedge \dbar \tau \}\beta , \beta \right > e^{-\psi}dV_{\omega}.
\end{eqnarray}

Using the estimate \eqref{bmo-est}, we shall prove the following theorem.   

\begin{d-thm} \label{twist-1-thm}
Let $(X, \omega)$ be a Stein K\"ahler manifold, and $L \to X$ a holomorphic line bundle with possibly singular Hermitian metric $e^{-\kappa}$.  Suppose there exists a smooth function $\eta :X \to \R$ and a $q$-positive, a.e. strictly $q$-positive Hermitian $(1,1)$-form $\Theta$ such that 
\[
\ii \left ( \di \dbar \kappa +{\rm Ricci}(\omega) + (1-\delta) \di \dbar \eta + (1+\delta)(\di \dbar \eta - \di \eta \wedge \dbar \eta)\right ) - \Theta 
\]
is $q$-positive for some $\delta \in (0,1)$.  Then for all $L$-valued $(0,q)$-forms $\alpha$ such that 
\[
\dbar \alpha = 0 \quad \text{and} \quad \int _X |\alpha |^2_{\Theta , \omega} e^{-\kappa} dV_{\omega} < +\infty
\]
there exists an $L$-valued $(0,q-1)$-form $u$ such that 
\[
\dbar u= \alpha  \quad \text{and} \quad \int _X |u|^2_{\omega} e^{-\kappa} dV_{\omega} \le \frac{1+\delta}{\delta} \int _X |\alpha |^2_{\Theta , \omega} e^{-\kappa} dV_{\omega}.
\]
\end{d-thm}

\begin{proof}
By the usual technique of approximation, we may replace $X$ by a pseudoconvex domain $\Omega \subset X$, and we may assume that all metrics are smooth.

Let 
\[
\tau = e^{-\eta} \quad \text{and} \quad \kappa = \psi - \eta.
\]
Define the operators 
\[
T = \sqrt{\frac{1+\delta}{\delta}}\dbar \circ\sqrt{\tau} \quad \text{and} \quad S = \sqrt{\tau} \circ \dbar.
\]
Then 
\[
e^{-\psi} \left (\tau (\di \dbar \psi +{\rm Ricci}(\omega) )- \di \dbar \tau - \frac{\delta}{\tau} \di \tau \wedge \dbar \tau \right ) = e^{-\kappa} \left ( \di \dbar \kappa + {\rm Ricci}(\omega) +2\di \dbar \eta - (1+\delta) \di \eta \wedge \dbar \eta  \right ),
\]
so by \eqref{bmo-est} and the hypotheses we have the a priori estimate 
\[
||T^*_{\psi}\beta||^2_{\psi} + ||S\beta||^2 _{\psi}\ge \int _{\Omega}\tau \left < \ii \Theta \beta , \beta \right >_{\omega} e^{-\psi} dV_{\omega}
\]
for all smooth $\beta$ in the domain of $T^*_{\psi}$ (which coincides with the domain of $\dbar ^*_{\psi}$).  Since the smooth forms are dense, the result holds for all $\beta$ in the domain of $T^*_{\psi}$.

If we now apply the proof of Theorem \ref{h-type-thm}, {\it mutatis mutandis}, to the operators $T$ and $S$ in place of $\dbar_q$ and $\dbar_{q+1}$ respectively, we obtain a solution $U$ of the equation 
\[
TU = \alpha
\]
with the estimate 
\[
\int _{\Omega} |U|^2e^{-\psi} dV_{\omega} \le \int _X |\alpha |^2_{\Theta , \omega} e^{-\kappa} dV_{\omega}.
\]
Letting $u = \sqrt{\frac{1+\delta}{\delta}\tau}U$, we have $\dbar u = \alpha$ and 
\[
\int _{\Omega}|u|^2e^{-\kappa} dV_{\omega} = \frac{1+\delta}{\delta} \int _{\Omega} |U|^2 e^{-\psi} dV_{\omega} \le \frac{1+\delta}{\delta} \int _{\Omega} |\alpha |^2_{\Theta , \omega} e^{-\kappa} dV_{\omega},
\]
which is what we claimed.
\end{proof}

\subsection{Twisted estimates: method II}

\begin{d-thm} \label{twist-2-thm}
Let $(X, \omega)$ be a Stein K\"ahler manifold, and $L \to X$ a holomorphic line bundle with Hermitian metric $e^{-\kappa}$.  Suppose there exists a smooth function $\eta :X \to \R$ and a $q$-positive, a.e. strictly $q$-positive Hermitian $(1,1)$-form $\Theta$ such that 
\[
\ii \left ( \di \dbar \kappa +{\rm Ricci}(\omega) +  \di \dbar \eta + (\di \dbar \eta - \di \eta \wedge \dbar \eta) - \Theta \right )
\]
is $q$-positive.  Let $\alpha$ be an $L$-valued $(0,q)$-form such that $\alpha = \dbar u$ for some $L$-valued $(0,q-1)$-form $u$ satisfying 
\[
\int _X |u|^2_{\omega} e^{-(\kappa+\eta)} dV_{\omega} < +\infty.
\]
Then the solution $u_o$ of $\dbar u_o = \alpha$ having minimal norm satisfies the estimate 
\[
\int _X |u_o|^2_{\omega} e^{-\kappa} dV_{\omega} \le  \int _X |\alpha |^2_{\Theta , \omega} e^{-\kappa} dV_{\omega}.
\]
\end{d-thm}

\begin{proof}
Since $X$ is Stein, it can be exhausted by strictly pseudoconvex domains.  If we prove the result for a strictly pseudoconvex domain $\Omega \relcomp X$ then the uniformity of the estimates will allow us, using Alaoglu's Theorem, to increase $\Omega$ to cover all of $X$.  Therefore we may replace $X$ by $\Omega$.

Let $\psi = \kappa +\eta$ and $\tau = e^{-\eta}$.  Then 
\[
\left (\tau (\di \dbar \psi +{\rm Ricci}(\omega)-\Theta) - \di \dbar \tau\right )e^{-\psi} = \left ( \di \dbar \kappa + {\rm Ricci}(\omega) + \di \dbar \eta + (\di \dbar \eta -\di \eta \wedge \dbar \eta )\right ) e^{-\kappa}.
\]
By Kohn's work on the $\dbar$-Neumann problem, on $\Omega$ the solution $u_{o,\Omega}$ of minimal norm is of the form 
\[
u _{o,\Omega} = \dbar ^*_{\psi}\beta
\]
for some $L$-valued $\dbar$-closed $(0,q)$-form $\beta$ that is smooth up to the boundary of $\Omega$ and satisfies the $\dbar$-Neumann boundary conditions.  It follows from \eqref{twisted-apest} and the hypotheses, we obtain the estimate 
\[
\int _{\Omega} |\alpha|^2_{\Theta, \omega}e^{-\kappa} dV_{\omega} \ge \int _{\Omega} |u_{o,\Omega}|^2_{\omega}e^{-\kappa} dV_{\omega},
\]
and the proof is finished by taking the aforementioned limit as $\Omega \nearrow X$.
\end{proof}

\subsection{Functions with self-bounded gradient}

We denote by $W^{1,2}_{\ell oc}(M)$ the set of locally integrable functions on a manifold $M$ whose first derivative, computed in the sense of distributions, is locally integrable.  In \cite{mcn-sbg} the following definition was introduced.

\begin{defn}
Let $X$ be a complex manifold.  A function $\eta \in W^{1,2}_{\ell oc}(X)$ is said to have self-bounded gradient if there exists a positive constant $C$ such that the $(1,1)$-current 
\[
\ii \di \dbar \eta - C \ii \di\eta \wedge \dbar \eta
\]
is non-negative.  We denote the set of functions with self-bounded gradient on $X$ by ${\rm SBG}(X)$.
\red
\end{defn}

Writing $\phi = C\eta$, we get 
\[
\ii \di \dbar \phi - \ii \di \phi \wedge \dbar \phi = C(\ii \di \dbar \eta - C \ii \di \eta \wedge \dbar \eta) \ge 0,
\]
and thus we can always normalize a function with self-bounded gradient so that $C=1$.  We write 
\[
{\rm SBG}_1(X) := \left \{ \eta \in W^{1,2}_{\ell oc}(X)\ ;\ \ii \di \dbar \eta \ge \ii \di \eta \wedge \dbar \eta\right \}.
\]

\begin{rmk}
The normalization $C=1$ has the minor advantage that one can then focus on the maximal positivity of $\ii \di \dbar \eta$ as $\eta$ varies over ${\rm SBG_1(X)}$.  On the other hand, for certain kinds of problems, such as regularity for the $\dbar$-Neumann problem, one expects to find local functions with self-bounded gradient and arbitrarily large Hessian, so in this case the normalization can have slight conceptual and notational  disadvantages.  In any case, it is easy to pass between the normalized and unnormalized notions, so we will not worry too much about this point.
\red
\end{rmk}

It might be hard to tell immediately whether one can have functions with self-bounded gradient on a given complex manifold.  Indeed, the condition that the square norm of the $(1,0)$-derivative of a function give a lower bound for its complex Hessian certainly appears to be a strong condition, but on the surface it does not immediately give a possible obstruction to the existence of such a function.  However, one can rephrase the property of self-bounded gradient.  To see how, note that 
\[
\ii \di \dbar (-e^{-\eta}) = e^{-\eta}(\ii \di \dbar \eta -  \di \eta \wedge \dbar \eta).
\]
Thus $\eta \in {\rm SBG}_1(X)$ if and only if $-e^{-\eta}$ is plurisubharmonic.  Since $-e^{-\eta} \le 0$, we see that if a complex manifold admits a function with self-bounded gradient if and only if it admits a negative plurisubharmonic function.  

\begin{ex}
${\rm  SBG}(\C ^n) = \{ \text{constant functions}\}.$
\red
\end{ex}

\begin{ex}\label{ball-example}
In the unit ball $\B _n \subset \C ^n$, one can take 
\[
\eta (z) = \log \frac{1}{1-|z|^2}.
\]
Then 
\[
\ii \di \dbar \eta - \ii \di \eta \wedge \dbar \eta = \frac{\ii dz \dot \wedge d\bar z}{1-|z|^2} 
\]
\red
\end{ex}

As one can see from H\"ormander's Theorem, if a complex manifold admits a bounded plurisubharmonic function, then this function can be added to any weight function without changing the underlying vector space of Hilbert space in which one is working, while doing so increases the complex Hessian of the weight, thus allowing H\"ormander's Theorem to be applied for a wider range of weights.  One of the main reasons for introducing functions with self-bounded gradient is that they achieve the same gain in the complex Hessian of the weight, but are not necessarily bounded.  

\begin{ex}
Let $X$ be a complex manifold and $Z \subset X$ a hypersurface.  Assume there exists a function $T \in \co (X)$ such that 
\[
Z = \{ x\in X\ ;\ T(x) = 0\}\quad \text{and} \quad \sup _X |T| \le 1.
\]
Then the function 
\[
\eta (x) = - \log (\log |T|^{-2})
\]
has self-bounded gradient.  Indeed, 
\[
\di \eta = \frac{1}{\log |T|^{-2}} \di \left ( \log |T|^2\right ),
\]
and
\[
\di \dbar \eta =  \frac{1}{\log |T|^{-2}} \di \dbar \left ( \log |T|^2\right )  +  \frac{dT \wedge d\bar T}{|T|^2(\log |T|^{-2})^2} = \frac{dT \wedge d\bar T}{|T|^2(\log |T|^{-2})^2},
\]
where the latter equality follows from the Poincar\'e-Lelong Formula.  Therefore
\[
\ii \di \dbar \eta - \ii \di \eta \wedge \dbar \eta = 0.
\]
To see that $\eta \in W^{1,2}_{\ell oc}(X)$, one argues as follows.  Obviously $\eta$ is smooth away from the zeros of $T$.  If the poles of $\ii \di \eta \wedge \dbar \eta$ have codimension $\ge 2$, then the Skoda-El Mir Theorem allows us to replace $\ii \di \eta \wedge \dbar \eta$ with the $0$ current.  Thus it suffices to check local integrability near the smooth points of $Z$.  At such a smooth point, one can take a local coordinate system whose first coordinate is $T$.  By Fubini's Theorem, we are therefore checking the local integrability of $|z|^{-2}(\log |z|^{-2})^{-2}$ near $0$ in $\C$ with respect to Lebesgue measure, and the latter follows from direct integration.  Thus $\eta \in {\rm SBG}_1(X)$.

Such a function, as well as some variants of it, will be used in the next section, when we discuss theorems on $L^2$ extension of holomorphic sections from $Z$ to $X$.
\red
\end{ex}

\subsection{How twist gives more}\label{SS:twist_more}

In this section, we elaborate how the the twisted $\dbar$ estimates given by Theorem \ref{twist-1-thm} are genuinely stronger than the $\dbar$ estimates given by H\" ormander's theorem, Theorem \ref{h-type-thm}. Of course, both theorems follow from
the same basic method: unravel the natural energy form associated to the complex being studied -- the left-hand side of \eqref{twisted-bk-id} for the twisted estimates and the left-hand side of \eqref{bk-gen} for H\" ormander's estimates -- via integration by parts.
So in a very general sense, both sets of estimates on $\dbar$ might be said to be ``equivalent'' to elementary calculus, and hence equivalent to each other. But such a statement is not illuminating, especially in regard to the positivity needed to invoke Theorems \ref{twist-1-thm} and 
 \ref{h-type-thm} -- the right-hand sides of  \eqref{twisted-bk-id} and \eqref{bk-gen}, respectively.
 
 In order to compare these two estimates, consider the simplest situation. Let $\Omega\subset\C^n$ be a domain with smooth boundary, equipped with the Euclidean metric, which is pseudoconvex. Let $\phi\in C^2(\Omega)$ be a function, variable at this point but to be determined soon. Let $f$ be an ordinary $(0,1)$-form on $\Omega$ satisfying $\dbar f=0$. [Thus, in Theorems \ref{twist-1-thm} and  \ref{h-type-thm}, $q=1$, $\omega=\text{Euclidean}$ (which we'll denote with a subscript $e$), $E\to X$ is the trivial bundle, and $h=e^{-\phi}$ globally.]
 
 Theorem  \ref{h-type-thm} guarantees a function $u$ solving $\dbar u=f$ and satisfying the estimate
 
 \begin{equation}\label{E:basic-hor-dbar}
 \int_\Omega |u|^2\, e^{-\phi}\, dV_e\leq \int_\Omega |f|^2_\Xi\, e^{-\phi}\, dV_e
 \end{equation}
 as long as
 
 \begin{equation}\label{E:basic-hor-hyp}
 \Xi=: \ii \partial\dbar\phi >0
 \end{equation}
(and the right-hand side of \eqref{E:basic-hor-dbar} $<\infty$). It seems to us that a $\dbar$-estimate can legitimately be said to hold ``by H\" ormander'' only if $\phi$ can be chosen such that \eqref{E:basic-hor-hyp} holds and then \eqref{E:basic-hor-dbar} is the resulting estimate.

On the other hand, Theorem  \ref{twist-1-thm} guarantees a solution to $\dbar v=f$ satisfying

 \begin{equation}\label{E:basic-twist-dbar}
 \int_\Omega |u|^2\, e^{-\phi}\, dV_e\leq C_\delta\, \int_\Omega |f|^2_\Theta\, e^{-\phi}\, dV_e
 \end{equation}
 as long as there exists a function $\eta$ and a constant $\delta\in (0,1)$ such that
 
 \begin{equation}\label{E:basic-twist-hyp}
 \Theta =: \ii\left[\partial\dbar\phi + (1-\delta) \di \dbar \eta + (1+\delta)(\di \dbar \eta - \di \eta \wedge \dbar \eta)\right] >0
 \end{equation}
(and the right-hand side of \eqref{E:basic-twist-dbar} $<\infty$). Inequality \eqref{E:basic-twist-hyp} is manifestly more general than  \eqref{E:basic-hor-hyp}. And there are two somewhat different ways in which the estimate \eqref{E:basic-twist-dbar} achieves more than estimate \eqref{E:basic-hor-dbar}: 
\begin{enumerate}
\item[(i)] when $\phi$ is specified, or
\item[(ii)] when the pointwise norm $|f|_{\Theta}$, appearing on the right-hand side of \eqref{E:basic-twist-dbar}, is specified. 
\end{enumerate}

As a very elementary illustration, suppose one seeks a $\dbar$ estimate in ordinary $L^2$ norms, i.e., $\phi =0$. No information directly follows ``by H\" ormander'' since  \eqref{E:basic-hor-hyp} fails. (Although, as we noted earlier, if $\Omega$ is {\it bounded}, we could add $|z|^2$ to $\phi$ and obtain a solution satisfying an $L^2$ estimate).  Notice, however, that if $\Omega$ supports a function $\eta$ with self-bounded gradient  such that

\begin{equation}\label{E:closed-range}
\ii\left(\di\dbar\eta -\di\eta\wedge\dbar\eta\right)\geq a \ii \di \dbar |z|^2 >0,
\end{equation}
then  \eqref{E:basic-twist-dbar} gives a solution to $\dbar v=f$ satisfying $\int_\Omega |v|^2\, dV_e\leq C\, \int_\Omega |f|^2\, dV_e$ as desired. And condition \eqref{E:closed-range} can hold on some unbounded domains $\Omega$.

More generally, one may seek an estimate with a specified $\phi$, where this function is not even weakly plurisubharmonic. This situation occurs in the $L^2$ extension theorems with ``gain'' discussed in Section \ref{S:extension} below. In these cases, positivity of 
$\ii\left(\di\dbar\eta -\di\eta\wedge\dbar\eta\right)$ can be used to compensate for negativity of $\ii\di\dbar\phi$ in order to achieve $\Theta >0$ and get estimate \eqref{E:basic-twist-dbar}.

However, the most significant feature of the twisted $\dbar$ estimates, to our mind, comes when one {\it needs} to specify the ``curvature'' term occurring in the pointwise norm of $f$ on the right-hand side of the estimate, in order to assure that this integral is uniformly finite.
We refer to expressions like $\Theta$ or $\Xi$ as "curvature terms" simply for convenient shorthand; by ``uniformly finite'' we mean the integrals are bounded independently of certain parameters built into the functions $\eta$ and/or $\kappa$. There are many natural problems where large enough curvature terms of the form $\Xi$ can {\it not} be constructed without re-introducing blow-up in the form of the density $e^{-\phi}$ in the integrals. The Maximum Principle for plurisubharmonic functions is the obstruction. 

To see this explicitly, consider the (simplest) set-up of the Ohsawa-Takegoshi extension theorem (stated below as Theorem \ref{ot-vanilla}): $H$ is a complex hyperplane in $\C^n$, $\Omega$ is a bounded pseudoconvex domain, and
$f$ is a holomorphic function on $H\cap\Omega$ with finite $L^2$ norm. The point discussed below is perhaps the most important difficulty in establishing $L^2$ extension, and the issue arises in other, more complicated extension problems as well. 

To prove the Ohsawa-Takegoshi theorem, one first notes that it suffices to consider the to-be-extended
function, $f$, to be $C^\infty$ in an open neighborhood of $\overline{H\cap\Omega}$ in $H$. This reduction is achieved
by exhausting $\Omega$ by pseudoconvex domains $\Omega_c$ with smooth boundaries (the domains $\Omega_c$ can be taken to be strongly pseudoconvex as
well, but this is inessential for the current discussion). This reduction is by now a standard result in the subject.
However, the size of this neighborhood, say $U$, is \textit{not}
uniform -- it depends on the parameter $c$ above or, equivalently, on the function $f$ to be extended. It is essential to obtain estimates that do not depend on the size of this neighborhood; this is the heart of the proof of the Ohsawa-Takegoshi theorem.

If coordinates are chosen so that $H=\{z_n=0\}$, it is natural to extend a given holomorphic $f(z_1,\dots z_{n-1})$ simply be letting it be constant in $z_n$. But note that extending $f$ in this way does not necessarily define a function on all of $\Omega$. The purpose of the first reduction is to circumvent this difficulty. 
If $f$ is assumed defined (and smooth) on $U$ above, which sticks out of $\overline\Omega$ somewhat, then there exists an $\epsilon >0$ such that all points in
$$\left\{ z=(z', z_n)\in\Omega: z'\in H\cap\Omega\text{ and } |z_n| <\epsilon\right\}$$
have the property that $(z',0)\in H\cap U$. Note that the size of $\epsilon$ depends on the unspecified neighborhood $U$, so can be small in an uncontrolled manner.
One then takes a cut-off function $\chi\left(\left|z_n\right|\right)$, whose support is contained in $\{|z_n| <\epsilon\}$ and which is $\equiv 1$ near $H$; a smooth extension of $f$ to $\Omega$ is then given by $\tilde f(z_1,\dots z_{n-1}, z_n)=\chi\left(\left|z_n\right|\right)\cdot f(z_1,\dots z_{n-1})$. 

Defining $\alpha =\dbar\left(\tilde f\right)$, we now seek to solve $\dbar u=\alpha$ with estimates on $u$ in terms of the $L^2$ norm of $f$ alone. Note that $|\dbar\chi|^2 \lesssim \frac 1{\epsilon^2}$, where $\epsilon$ is the thickness of the slab above. This is the enemy of our desired estimate. In order the kill this term on the right-hand side of the $\dbar$ inequality, we need a curvature term of size $\approx \frac 1{\epsilon^2}$ in a $\epsilon$ collar about $\{z_n=0\}$. Additionally, this curvature must be produced without introducing perturbation factors which cause the perturbed $L^2$ norms to differ essentially from the starting $L^2$ structure.
Using H\" ormander's $\dbar$ set-up, this can only done by introducing weights, $\phi_\epsilon=\phi$, which

\begin{itemize}
\item[(i)] are plurisubharmonic and
\item[(ii)] are \textit{bounded} functions, independently of $\epsilon$, while
\item[(iii)] $\ii\partial\dbar\phi\geq \frac{C}{\epsilon^2}\ii dz_n\wedge d\bar z_n$ on ${\rm Support}(\chi)$,  for all sufficiently small $\epsilon>0$ and some constant $C>0$ independent of $\epsilon$.
\end{itemize}

\noindent These requirements are incompatible, as we now show.

\subsubsection{An extremal problem}\label{SSS:extremal}

Let ${\rm SH}(\Omega)$ denote the subharmonic functions on a domain $\Omega\subset\C$.
Let $D(p;a)$ denote the disc in $\C^1$ with center $p$ and radius $a$. Define the set of functions
\[
\sg=\left\{u\in {\rm SH}(D)\cap C^2(D):\,\, 0\leq u(z)\leq 1, z\in D\right\},
\]
where $D=D(0;1)$. For $0<\epsilon <1$, consider the following extremal problem: how large can $K>0$ be such that
\begin{equation}\label{E:lower}
\triangle u(z)\geq K\qquad\forall\, z\in D(0;\epsilon)\,\, 
\end{equation}
for $u\in\sg$?

We first observe that it suffices to consider \textit{radial} elements in $\sg$.

\begin{lem} Let $u\in {\rm SH}(D)\cap C^2(D)$ satisfy \eqref{E:lower}. There exists a radial $v\in {\rm SH}(D)\cap C^2(D)$ such that
\begin{itemize}
\item[(i)] $\|v\|_{L^\infty(D)}\leq \|u\|_{L^\infty(D)}$
\item[(ii)] $\triangle v(r)\geq K$ if $0\leq r\leq\epsilon$.
\end{itemize}
\end{lem}

\begin{proof}
Define
\[
v(r)=\frac 1{2\pi}\int_0^{2\pi} u\left(r e^{\ii\alpha}\right)\, d\alpha.
\]
The function $v$ is clearly radial and satisfies (i). It is also a standard fact, ``Hardy's convexity theorem'', see e.g. \cite[Page 9]{duren}, that $v\in {\rm SH}(D)$.

To see (ii), recall that in polar coordinates $(r,\theta)$
\[
\triangle =\frac{\partial^2}{\partial r^2} +\frac 1r\, \frac{\partial}{\partial r} +\frac 1{r^2}\, \frac{\partial^2}{\partial\theta^2}.
\]
Thus
\begin{align*}
\triangle v(r)&=\left[\frac{\partial^2}{\partial r^2} +\frac 1r\, \frac{\partial}{\partial r}\right] v(r) \\
&=\frac 1{2\pi}\int_0^{2\pi} \left[\frac{\partial^2}{\partial r^2} +\frac 1r\, \frac{\partial}{\partial r}\right] \left(u\left(re^{\ii\alpha}\right)\right)\, d\alpha \\
(*)\quad &= \frac 1{2\pi}\int_0^{2\pi} \left[\frac{\partial^2}{\partial r^2} +\frac 1r\, \frac{\partial}{\partial r}+\frac 1{r^2}\frac{\partial^2}{\partial\alpha^2}\right] \left(u\left(re^{\ii\alpha}\right)\right)\, d\alpha \\ &\geq K\qquad\text{if  } re^{\ii\alpha}\in D(0;\epsilon),
\end{align*}
since $u$ satisfies \eqref{E:lower}. Note that to obtain equality (*), the angular part of the Laplacian was added to the integrand; that this is zero follows using integration by parts:
\[
\frac 1{2\pi}\frac 1{r^2} \int_0^{2\pi}\frac{\partial}{\partial\alpha}\left[\frac{\partial}{\partial\alpha}\, u\left(r e^{\ii\alpha}\right)\right]\, d\alpha \stackrel{\text{IBP}}{=}
\left.\frac 1{2\pi r^2}\frac{\partial}{\partial\alpha}\, u\left(r e^{\ii\alpha}\right)\right|^{\alpha= 2\pi}_{\alpha =0} =0.
\]
\end{proof}
\medskip

Let $\sg_{\text{rad}}$ denote the radial functions in $\sg$.

\begin{prop}\label{P:extremal} Suppose $u\in\sg_{\text{rad}}$, $0 <\epsilon <1$, and $\triangle u(z)\geq K$ for all $z\in D(0;\epsilon)$. Then
$$K\lesssim \frac 1{\epsilon^2}\, \left(\log \frac 1\epsilon\right)^{-1},$$
\medskip
where the estimate $\lesssim$ is uniform in $\epsilon$.
\end{prop}

\begin{proof} 
We use the standard notation $f_r=\frac{\partial f}{\partial r}$.  Since $u$ is radial and subharmonic on $D(0;1)$, 
\[
\frac\partial{\partial r}\left(r\, u_r(r)\right)\geq 0\qquad\text{for all }\, 0\leq r <1.
\]
In particular, for $\epsilon\leq s< 1$ we have
\[
\int_\epsilon^s \frac\partial{\partial r}\left(r\, u_r(r)\right)\, dr \geq 0,
\]
which implies
\begin{equation}\label{E:1}
s\, u_r(s)\geq \epsilon\, u_r(\epsilon)\qquad\forall\,\, \epsilon\leq s <1.
\end{equation}

On the other hand, $\triangle u\geq K$ on $D(0;\epsilon)$ implies
\begin{equation}\label{E:2}
\frac\partial{\partial r}\left[r\, u_r(r)\right] =r\, \triangle u\geq K\, r\quad\text{for  } 0<r<\epsilon.
\end{equation}
Integrate both sides of \eqref{E:2} from $0$ to $\epsilon$ to obtain
\begin{equation}\label{E:3}
\epsilon\, u_r(\epsilon)\geq \frac {K\epsilon^2}2.
\end{equation}
Now combine \eqref{E:1} with \eqref{E:3} to get
\begin{equation}\label{E:4}
u_r(s)\geq \frac {K\epsilon^2}2\cdot \frac 1s \qquad\forall\,\, \epsilon\leq s <1.
\end{equation}
However
\begin{align*}
u(s)&=\int_\epsilon^s u_r(t)\, dt +u(\epsilon) \\
&\geq  \frac {K\epsilon^2}2\cdot\log\left(\frac s\epsilon\right)+u(\epsilon)
\end{align*}
follows from \eqref{E:4}. This, plus the fact that $u\leq 1$, gives
\[
\frac {K\epsilon^2}2\cdot\log\left(\frac s\epsilon\right)\leq 2,
\]
and thus 
\[
K \lesssim \frac 1{\epsilon^2}\, \left(\log \frac 1\epsilon\right)^{-1}.
\]
as claimed.
\end{proof}

\bigskip

\begin{s-rmk}
Note that $\frac 1{\epsilon^2}\, \left(\log \frac 1\epsilon\right)^{-1} << \frac 1{\epsilon^2}$ as $\epsilon\to 0$.
\red
\end{s-rmk}

Now return to the discussion before Subsection \ref{SSS:extremal}. The $\dbar$ data, $\alpha$, associated to the smooth extension of $f$, is large as $\epsilon\to 0$: $|\alpha|^2 =\left|\dbar\chi\right|^2\, |f|^2 \approx \frac 1{\epsilon^2}$ on the support of $\dbar\chi$.
It follows from Proposition \ref{P:extremal} that the $\ii dz_n\wedge d\bar z_n$ component of $\ii\di\dbar\phi$ is $\leq C \epsilon^{-2}\log\left(\frac 1\epsilon\right)^{-1}$ for any bounded plurisubharmonic function on $\Omega$. Thus, there is {\bf no} bounded psh function $\phi$ such that
$$|\alpha|^2_\Xi\, e^{-\phi} <K,$$
for $K$ independent of $\epsilon$. Consequently, the Ohsawa-Takegoshi theorem does not follow ``by H\" ormander'' in the sense described earlier.

As another example where the twisted estimates yield more than H\" ormander, consider the Poincare metric. Let $D\subset\C$ be the unit disc. The Poincare metric on $D$ (up to a constant) has K\" ahler form

\begin{equation}\label{E:poincare_1d}
P=\frac {\ii dz \wedge d\bar z}{\left(1-|z|^2\right)^2},
\end{equation}
i.e., the pointwise Poincare length of a form $f dz$ is $|fdz|_P=|f| \left(1-|z|^2\right)$, where $|\cdot|$ is ordinary absolute value.

A simple argument shows that $P$ cannot arise from a {\it bounded} potential:

\begin{prop} There is no  $\lambda\in L^\infty(D)$ such that
$$\ii\di\dbar\lambda (z)\geq \frac {\ii dz \wedge d\bar z}{\left(1-|z|^2\right)^2},\qquad z\in D.$$
\end{prop}

\begin{proof}
Suppose there were such a $\lambda$. For $0<r<1$, let  $D_r=\{ z: |z|<r\}$ and $\lambda_r(z)=\lambda (rz)$.  Integration by parts gives

\begin{align*}
\int_{D_r}
\frac{\partial^2}{\partial z\partial\bar z}\left(\lambda_r(z)\right) (r^2-|z|^2)
&=\left|\int_{D_r} \frac{\partial\lambda_r}{\partial\bar z} \bar z\right| \\
&=\left|-\int_{D_r}\lambda -\int_{bD_r}\lambda \frac{|z|^2}{|z|+|\bar z|} \right|\\
&\leq 2\pi ||\lambda ||_{\infty}.
\end{align*}

But the lower bound on $\ii\partial\bar\partial\lambda$ implies

\begin{align*}
\int_{D_r}
\frac{\partial^2}{\partial z\partial\bar z}\left(\lambda_r(z)\right) (r^2-|z|^2)
&\geq \int_{D_r}\frac{r^2}{(r^2-|z|^2)^2} (r^2-|z|^2) \\
&\geq 2\pi r^2\int_0^r\frac 1{(r^2-\rho^2)}\rho\, d\rho
\\ &= +\infty,\end{align*}
which is a contradiction.
\end{proof}

Therefore, it is not possible to conclude ``by H\" ormander'' that we can solve $\dbar u =f$ with the estimate

\begin{equation}\label{E:dbar-poincare-1d}
\int_D |u|^2\, dV_e\leq C\, \int_D |f|^2_P\, dV_e.
\end{equation}
But \eqref{E:dbar-poincare-1d} is true and follows easily from \eqref{E:basic-twist-dbar}: take $\phi =0$ and $\eta=-\log\left(1-|z|^2\right)$, and compute that $\Theta= P$ in \eqref{E:basic-twist-hyp}.

Estimates like \eqref{E:dbar-poincare-1d} for classes of domains in $\C^n$ will be discussed in Section \ref{S:invariant} below.

\subsection{Some examples of estimates for $\dbar$ under weakened curvature hypotheses}

In this section, we demonstrate the sort of improvements that we get from the twisted estimates for $\dbar$ in a number of situations.

\subsubsection{The unit ball}
Let us begin with the unit ball $\B_n$.  We write 
\[
\omega _P := \ii \di \dbar \log \frac{1}{1-|z|^2}
\]
for the Poincar\'e metric.  We begin by applying the twisted estimates of Method I.  From Theorem \ref{twist-1-thm} and Example \ref{ball-example} we have the following theorem.

\begin{d-thm}\label{ball-euc-dbar-1}
Let $\psi \in L^1_{\ell oc}(\B_n)$ be a weight function, and assume there exists a positive constant $\delta$ such that 
\[
\ii \di \dbar \psi \ge -(1-\delta)\omega _P.
\]
Then for any $(0,1)$-form $\alpha$ such that 
\[
\dbar \alpha = 0 \quad \text{and} \quad \int _{\B_n}|\alpha|^2_{\omega _P}e^{-\psi} dV < +\infty
\]
there exists a locally integrable function $u$ such that 
\[
\dbar u = \alpha \quad \text{and} \quad \int _{\B _n} |u|^2 e^{-\psi} dV \le \frac{2(2+\delta)}{\delta^2}  \int _{\B_n}|\alpha|^2_{\omega _P}e^{-\psi} dV.
\]
\end{d-thm}

\begin{proof}
In Theorem \ref{twist-1-thm}, we let $\omega = \frac{\ii}{2} \di \dbar |z|^2$, $\kappa = \psi$, $\Theta = \frac{\delta}{2} \omega _P$  and $\eta = \log \frac{1}{1-|z|^2}$.  Then 
\[
\ii (\di \dbar \kappa + {\rm Ricci}(\omega) +(1-\tfrac{\delta}{2})\di \dbar \eta + (1+\tfrac{\delta}{2})(\di \dbar \eta - \di \eta \wedge \dbar \eta ))- \Theta  \ge  \ii \di \dbar \psi + (1-\delta )\omega _P \ge 0.
\]
Thus the hypotheses of Theorem \ref{twist-1-thm} hold, and we have our proof.
\end{proof}

Next, we turn to the application of Method II, i.e., Theorem \ref{twist-2-thm}.

\begin{d-thm}\label{ball-euc-dbar-2}
Let $\psi \in L^1_{\ell oc}(\B_n)$ be a weight function, and assume there is a positive constant $\delta$ such that 
\[
\ii \di \dbar \psi \ge -(1-\delta)\omega _P.
\]
Fix any $(0,1)$-form $\alpha$ such that 
\[
\dbar \alpha = 0 \quad \text{and} \quad \int _{\B_n}|\alpha|^2_{\omega _P}e^{-\psi} dV < +\infty.
\]
Assume there exists a measurable function $\tilde u$ on $\B _n$ such that 
\[
\dbar \tilde u = \alpha \quad \text{and} \quad \int _{\B _n} |\tilde u|^2 e^{-\psi} (1-|z|^2)dV < +\infty.
\]
Then there is a measurable function $u$ on $\B _n$ such that 
\[
\dbar u = \alpha \quad \text{and} \quad \int _{\B_n}|u|^2e^{-\psi} dV \le \frac{1}{\delta} \int _{\B _n} |\alpha|^2_{\omega _P}e^{-\psi} dV.
\]
\end{d-thm}

\begin{proof}
One chooses $\omega = \frac{\ii}{2} \di \dbar |z|^2$, $\kappa = \psi$, $\Theta = \delta \omega _P$  and $\eta = \log \frac{1}{1-|z|^2}$ in Theorem \ref{twist-2-thm}.
\end{proof}

If we want to reduce further the lower bounds on the complex Hessian of $\psi$, we have to pay for it by restricting the forms $\alpha$ for which the $\dbar$-equation can be solved.  We have the following theorem.

\begin{d-thm}\label{ball-euc-dbar-2+-}
Let $\psi \in L^1_{\ell oc}(\B_n)$ be a weight function such that 
\[
\ii \di \dbar \psi \ge -\omega _P.
\]
Fix any $(0,1)$-form $\alpha$ such that 
\[
\dbar \alpha = 0 \quad \text{and} \quad \int _{\B_n}|\alpha|^2e^{-\psi} dV < +\infty.
\]
Assume there exists a measurable function $\tilde u$ on $\B _n$ such that 
\[
\dbar \tilde u = \alpha \quad \text{and} \quad \int _{\B _n} |\tilde u|^2 e^{-\psi} (1-|z|^2)dV < +\infty.
\]
Then there is a measurable function $u$ on $\B _n$ such that 
\[
\dbar u = \alpha \quad \text{and} \quad \int _{\B_n}|u|^2e^{-\psi} dV \le e \int _{\B _n} |\alpha|^2e^{-\psi} dV.
\]
\end{d-thm}

\begin{proof}
Choose $\omega = \Theta = \frac{\ii}{2} \di \dbar |z|^2$, $\kappa = \psi +|z|^2$, and $\eta = \log \frac{1}{1-|z|^2}$ in Theorem \ref{twist-2-thm}.  We then compute that 
\[
\ii \di \dbar \kappa +{\rm Ricci}(\omega) + \ii \di \dbar \eta +(\ii \di \dbar \eta -\ii \di\eta \wedge \dbar \eta )-\Theta = \ii \di \dbar \psi +\omega _B \ge 0.
\]
We thus obtain a function $u$ such that $\dbar u = \alpha$ and 
\[
\int _{\B _n} |u|^2e^{-(\xi +|z|^2)}dV \le \int _{\B_n}|\alpha|^2e^{-(\psi + |z|^2)} dV \le \int _{\B_n}|\alpha|^2e^{-\psi} dV.
\]
Since 
\[
\int _{\B _n} |u|^2e^{-\xi }dV \le e \int _{\B _n} |u|^2e^{-(\xi + |z|^2)}dV,
\]
the proof is complete.
\end{proof}

\subsubsection{Strictly pseudoconvex domains in $\C ^n$}

To a large extent, the situation in the unit ball carries over to strictly pseudoconvex domains.  The key is the Bergman kernel, and the celebrated theorem of Fefferman on its asymptotic expansion.

To state and prove our result, let us recall some basic facts about the Bergman kernel of a smoothly bounded domain $\Omega \subset \C ^n$.  Consider the spaces 
\[
L^2(\Omega) := \left \{ f : \Omega \to \C \ ;\ \int _{\Omega} |f|^2 dV < +\infty \right \} \quad \text{and}\quad \sa ^2(\Omega) := L^2(\Omega) \cap \co (\Omega).
\]
By Bergman's Inequality, $\sa ^2(\Omega)$ is a closed subspace, hence a Hilbert space, and thus the orthogonal projection $P _{\Omega} : L^2(\Omega) \to \sa ^2(\Omega)$ is a bounded operator.  This projection operator, called the Bergman projection, is an integral operator:
\[
(P_{\Omega}f)(z) = \int _{\Omega}K_{\Omega}(z,\bar w) f(w) dV(w).
\]
The kernel $K_{\Omega}$ is called the Bergman kernel, and it is a holomorphic function of $z$ and $\bar w$.  One has the formula 
\[
K_{\Omega}(z,\bar w) = \sum _{j=1} ^{\infty} f_j(z)\overline{f_j(w)}
\]
where $\{f_1,f_2,...\}\subset \sa ^2(\Omega)$ is any orthonormal basis.  In the special case of the unit ball, the Bergman kernel can be computed explicitly:
\[
K_{\B_n}(z,\bar w) = \frac{c_n}{(1-z\cdot \bar w)^{n+1}}.
\]

The Bergman kernel can be used to define a K\"ahler metric $\omega _B$ on $\Omega$, called the Bergman metric.  The definition is
\[
\omega _B (z) := \ii \di \dbar \log K_{\Omega}(z,\bar z).
\]

The theorem of Fefferman states that, near a give point $P\in \di \Omega$, the Bergman metric is asymptotic to the Bergman metric of a ball whose boundary closely osculates $\di \Omega$ at $P$.  With Fefferman's theorem, Example \ref{ball-example}, and a little more work, one can prove the following result.

\begin{d-thm}\label{spcvx-cn-sbg}
Let $\Omega \relcomp \C ^n$ be a domain with strictly pseudoconvex boundary.  Then there exists a positive constant $c$ such that 
\[
z \mapsto c\log K_{\Omega}(z,\bar z) \in {\rm SBG}_1(\Omega).
\]
Moreover, any such constant $c$ can be at most $\frac{1}{n+1}$.
\end{d-thm}

\begin{rmk}
There exist strictly pseudoconvex domains $\Omega$ for which the largest possible constant $c$ that can be chosen in Theorem \ref{spcvx-cn-sbg} is strictly less than $\frac{1}{n+1}$.
\red
\end{rmk}

But in fact, one can do a little better.  

\begin{d-thm}\label{spcvx-cn-qsbg}
Let $\Omega \relcomp \C ^n$ be a domain with strictly pseudoconvex boundary, and write 
\[
\eta (z) := \frac{1}{n+1} \log K_{\Omega}(z,\bar z).
\]
Then there exists a positive constant $C$ such that 
\[
\ii \di \dbar \eta - \ii \di \eta \wedge \dbar \eta \ge - C \ii \di \dbar |\cdot |^2.
\]
Moreover, the result fails if one replaces $\frac{1}{n+1}$ by a larger constant.
\end{d-thm}

\begin{proof}[Idea of proof]
From Fefferman's Theorem, we know that in the complement of a sufficiently large compact subset $K \relcomp \Omega$, one can achieve the conclusion of the theorem with $C$ arbitrarily small.  Compactness of $K$ and smoothness of the Bergman kernel in the interior of $\Omega$ takes care of the estimate on $K$.
\end{proof}

In the current state of the art, we know that Theorem \ref{spcvx-cn-sbg} also holds for domains of finite type in $\C ^2$, and convex domains of finite type in arbitrary dimension, but the conclusion of Theorem \ref{spcvx-cn-sbg} is not known to be true (resp. false) in every (resp. any) smoothly bounded pseudoconvex domain.  We also don't have such a precise version of Theorem \ref{spcvx-cn-qsbg} for domains that are not strictly pseudoconvex.  And given our current understanding of domains of finite type, the latter problem could be very difficult.

Let us now return to our H\"ormander-type theorems in the setting of strictly pseudoconvex domains.  We have the following analogues of the results for the ball.

\begin{d-thm}\label{spcvx-cn-dbar-1}
Let $\Omega \relcomp \C ^n$ be a domain with smooth, strictly pseudoconvex boundary.  Let $\psi \in L^1_{\ell oc}(\Omega)$ be a weight function, and assume there exists a positive constant $\delta$ such that 
\[
\ii \di \dbar \psi \ge -(1-\delta)\frac{1}{n+1}\omega _B.
\]
Then for any $(0,1)$-form $\alpha$ such that 
\[
\dbar \alpha = 0 \quad \text{and} \quad \int _{\Omega}|\alpha|^2_{\omega _B}e^{-\psi} dV < +\infty
\]
there exists a locally integrable function $u$ such that 
\[
\dbar u = \alpha \quad \text{and} \quad \int _{\Omega} |u|^2 e^{-\psi} dV \le \frac{M}{\delta^2}  \int _{\Omega}|\alpha|^2_{\omega _B}e^{-\psi} dV,
\]
where the constant $M$ depends only on the constant $C$ in Theorem \ref{spcvx-cn-qsbg} and the diameter of $\Omega$.
\end{d-thm}

\begin{proof}
Let $B$ be the smallest Euclidean ball containing $\Omega$ and let $P $ denote the center of $B$.  In Theorem \ref{twist-1-thm}, we let $\omega = \frac{\ii}{2} \di \dbar |z|^2$, $\kappa = \psi + C(1+\tfrac{\delta}{2})|z-P|^2$ where $C$ is as in Theorem \ref{spcvx-cn-qsbg}, $\Theta = \frac{\delta}{2(n+1)} \omega _B$ and $\eta (z) = \frac{1}{n+1}\log K_{\Omega}(z,\bar z)$.  Then 
\[
\ii (\di \dbar \kappa + {\rm Ricci}(\omega) +(1-\tfrac{\delta}{2})\di \dbar \eta + (1+\tfrac{\delta}{2})(\di \dbar \eta - \di \eta \wedge \dbar \eta ))- \Theta  \ge \ii \di \dbar \psi + (1-\delta )\frac{\omega _B}{n+1} \ge 0.
\]
Thus once again the hypotheses of Theorem \ref{twist-1-thm} hold, and we have a function $u$ satisfying $\dbar u =\alpha$ and 
\[
\int _{\Omega} |u|^2 e^{-(\psi+C(1+\tfrac{\delta}{2})|z-P|^2)} dV \le \frac{2(2+\delta)}{\delta^2}  \int _{\Omega}|\alpha|^2_{\omega _B}e^{-(\psi +C(1+\tfrac{\delta}{2})|z-P|^2)} dV \le \frac{2(2+\delta)}{\delta^2}  \int _{\Omega}|\alpha|^2_{\omega _B}e^{-\psi} dV .
\]
It follows that 
\[
\int _{\Omega} |u|^2 e^{-\psi} dV \le M' \int _{\Omega} |u|^2 e^{-(\psi+C(1+\tfrac{\delta}{2})|z-P|^2)} dV \le \frac{M}{\delta^2}  \int _{\Omega}|\alpha|^2_{\omega _B}e^{-\psi} dV.
\]
Obviously $M$ depends only on $C$ and the diameter of $\Omega$, and thus the proof is complete.
\end{proof}

Next, let us apply Method II.

\begin{d-thm}\label{spcvx-cn-dbar-2}
Let $\Omega \relcomp \C ^n$ be a domain with smooth, strictly pseudoconvex boundary, and denote by $\rho$ any smooth function with values in $(0,1)$, that agrees with the distance to $\di \Omega$ near $\di \Omega$.  Let $\psi \in L^1_{\ell oc}(\Omega)$ be a weight function, and assume there is a positive constant $\delta$ such that 
\[
\ii \di \dbar \psi \ge -(1-\delta)\omega _B.
\]
Fix any $(0,1)$-form $\alpha$ such that 
\[
\dbar \alpha = 0 \quad \text{and} \quad \int _{\Omega}|\alpha|^2_{\omega _B}e^{-\psi} dV < +\infty.
\]
Assume there exists a measurable function $\tilde u$ on $\Omega$ such that 
\[
\dbar \tilde u = \alpha \quad \text{and} \quad \int _{\Omega} |\tilde u|^2 e^{-\psi} \rho dV < +\infty.
\]
Then there is a measurable function $u$ on $\Omega$ such that 
\[
\dbar u = \alpha \quad \text{and} \quad \int _{\Omega}|u|^2e^{-\psi} dV \le \frac{M}{\delta} \int _{\Omega} |\alpha|^2_{\omega _B}e^{-\psi} dV,
\]
where the constant $M$ depends only on the constant $C$ in Theorem \ref{spcvx-cn-qsbg} and the diameter of $\Omega$.
\end{d-thm}

\begin{proof}
First let us note that Fefferman's Theorem (and in fact, a much softer argument) implies that, with $\eta (z) = \frac{1}{n+1} \log K_{\Omega}(z,\bar z)$,  
\[
A^{-1}\log \rho \le - \eta \le A\log \rho 
\]
for some constant $A$.  

Once again let $B$ be the smallest Euclidean ball containing $\Omega$ and $P$ the center of $B$.  In Theorem \ref{twist-2-thm}, we let $\omega = \frac{\ii}{2} \di \dbar |z|^2$, $\kappa = \psi + C|z-P|^2$ with $C$ as in Theorem \ref{spcvx-cn-qsbg}, $\Theta = \frac{\delta}{(n+1)} \omega _B$ and, as already mentioned, $\eta (z) = \frac{1}{n+1}\log K_{\Omega}(z,\bar z)$.  
Then we have 
\[
\int _{\Omega} |\tilde u|^2 e^{-(\psi+\eta)} dV \sim \int _{\Omega} |\tilde u|^2 e^{-\psi}\rho dV < +\infty.
\]
We calculate that
\[
\ii (\di \dbar \kappa + {\rm Ricci}(\omega) +\di \dbar \eta + (\di \dbar \eta - \di \eta \wedge \dbar \eta ))- \Theta  \ge  \ii \di \dbar \psi + (1-\delta )\frac{\omega _B}{n+1} \ge 0.
\]
Thus the hypotheses of Theorem \ref{twist-2-thm} hold, and we have a function $u$ satisfying $\dbar u =\alpha$ and 
\[
\int _{\Omega} |u|^2 e^{-(\psi+C|z-P|^2)} dV \le \frac{1}{\delta}  \int _{\Omega}|\alpha|^2_{\omega _B}e^{-(\psi +C|z-P|^2)} dV \le \frac{1}{\delta}  \int _{\Omega}|\alpha|^2_{\omega _B}e^{-\psi} dV .
\]
It follows that 
\[
\int _{\Omega} |u|^2 e^{-\psi} dV \le M \int _{\Omega} |u|^2 e^{-(\psi+C|z-P|^2)} dV \le \frac{M}{\delta}  \int _{\Omega}|\alpha|^2_{\omega _B}e^{-\psi} dV,
\]
which completes the proof.
\end{proof}

Finally, if we want to reduce further the lower bounds on the complex Hessian of $\psi$, we may do so, as in the case of the unit ball, at the cost of restricting further the forms for which we can solve $\dbar$.  We have the following generalization of Theorem \ref{ball-euc-dbar-2+-}.

\begin{d-thm}\label{spcvx-cn-dbar-2+-}
Let $\Omega$ and $\rho$ be as in Theorem \ref{spcvx-cn-dbar-2}, and let $\psi \in L^1_{\ell oc}(\Omega)$ be a weight function satisfying
\[
\ii \di \dbar \psi \ge -\omega _B.
\]
Fix any $(0,1)$-form $\alpha$ such that 
\[
\dbar \alpha = 0 \quad \text{and} \quad \int _{\Omega}|\alpha|^2e^{-\psi} dV < +\infty.
\]
Assume there exists a measurable function $\tilde u$ on $\Omega$ such that 
\[
\dbar \tilde u = \alpha \quad \text{and} \quad \int _{\Omega} |\tilde u|^2 e^{-\psi} \rho dV < +\infty.
\]
Then there is a measurable function $u$ on $\Omega$ such that 
\[
\dbar u = \alpha \quad \text{and} \quad \int _{\Omega}|u|^2e^{-\psi} dV \le M \int _{\Omega} |\alpha|^2e^{-\psi} dV,
\]
where $M$ depends only on the constant $C$ in Theorem \ref{spcvx-cn-qsbg} and the diameter of $\Omega$.
\end{d-thm}

\begin{proof}
Let $P \in \C^n$ be as in the proofs of Theorems \ref{spcvx-cn-dbar-1} and \ref{spcvx-cn-dbar-2}.  In Theorem \ref{twist-2-thm}, let $\omega = \frac{\ii}{2} \di \dbar |z|^2$, $\kappa = \psi +(C+1)|z-P|^2$, $\Theta = \ii \di \dbar |z|^2$ and $\eta (z) =\frac{1}{n+1} \log K_{\Omega}(z,\bar z)$.  Then as before,  
\[
\int _{\Omega} |\tilde u|^2 e^{-(\psi+\eta)} dV \sim \int _{\Omega} |\tilde u|^2 e^{-\psi}\rho dV < +\infty,
\]
and
\[
\ii (\di \dbar \kappa + {\rm Ricci}(\omega) +\di \dbar \eta + (\di \dbar \eta - \di \eta \wedge \dbar \eta ))- \Theta  \ge \ii \di \dbar \psi + \frac{\omega _B}{n+1} \ge 0.
\]
Thus the hypotheses of Theorem \ref{twist-2-thm} hold, and we have a function $u$ satisfying $\dbar u =\alpha$ and 
\[
\int _{\Omega} |u|^2 e^{-(\psi+(C+1)|z-P|^2)} dV \le \frac{1}{\delta}  \int _{\Omega}|\alpha|^2e^{-(\psi +(C+1)|z-P|^2)} dV \le   \int _{\Omega}|\alpha|^2e^{-\psi} dV .
\]
Again we conclude that 
\[
\int _{\Omega} |u|^2 e^{-\psi} dV \le M \int _{\Omega} |u|^2 e^{-(\psi+(C+1)|z-P|^2)} dV \le M  \int _{\Omega}|\alpha|^2e^{-\psi} dV,
\]
as desired.
\end{proof}

\section{Extension theorems}\label{S:extension}

\subsection{Extension from a hypersurface cut out by a holomorphic function}

The following result is the main theorem of \cite{ot}.

\begin{d-thm}[Ohsawa-Takegoshi]\label{ot-vanilla}
Let $\Omega$ be a bounded pseudoconvex domain in $\C^n$, $H \subset \C^n$ a complex hyperplane, and $\psi :\Omega \to \R \cup \{-\infty\}$ a plurisubharmonic function.  Then there exists a constant $C$ depending only on the diameter of $\Omega$ such that for any holomorphic function $f$ on $\Omega \cap H$ satisfying 
\[
\int _{\Omega \cap H}e^{-\psi}|f|^2 dV_{n-1} <\infty
\]
where $dV_{n-1}$ denotes the $(2n-2)$-dimensional Lebesgue measure, there exists a holomorphic function $F$ on $\Omega$ satisfying $F|\Omega \cap H=f$ and 
\[
\int _{\Omega} e^{-\psi}|F|^2 dV_n \le C\int _{\Omega \cap H} e^{-\psi}|f|^2 dV_{n-1}.
\]
\end{d-thm}

\begin{s-rmk}
Theorem \ref{ot-vanilla} was given new proofs by McNeal \cite{jeff-ot} (who did not state the theorem, but did construct $L^2$ extensions for his purposes in that article) and Berndtsson \cite{bo-96} at around the same time, using different methods.  Siu \cite{s-96} also gave a proof at about the same time, that on the one hand was more general, but on the other hand had stronger assumptions on the curvature, which he later removed in \cite{s-02}.  We will come back to Siu's Theorem shortly.
\red
\end{s-rmk}

In \cite{ot}, Theorem \ref{ot-vanilla} is established as an immediate corollary of the following result.

\begin{d-thm}[Ohsawa-Takegoshi]\label{ot-stein}
Let $X$ be a Stein manifold of dimension $n$, $\psi$ a plurisubharmonic function on $X$ and $s$ a holomorphic function on $X$ such that $ds \neq 0$ on any branch of $s^{-1}(0)$.  Let $Y := s^{-1}(0)$ and $Y_o:= \{ x\in Y\ ;\ ds(x) \neq 0\}$.  Let $g$ be a holomorphic $(n-1)$-form on $Y_o$ with 
\[
\int _{Y_o} e^{-\psi}\ii ^{n(n-1)} g\wedge \bar g < +\infty.
\]
Then there exists a holomorphic $n$-form $G$ on $X$ such that 
\[
G= g \wedge ds \quad \text{ on }Y_o
\]
and 
\[
\int _X \frac{e^{-\psi}}{(1+|s|^2)^2}\ii ^{(n+1)n}G\wedge \bar G \le 1620 \pi \int _{Y_o} e^{-\psi}\ii ^{n(n-1)} g\wedge \bar g < +\infty.
\]
\end{d-thm}

Manivel was the first to generalize Theorem \ref{ot-stein} to extension of holomorphic sections of a holomorphic vector bundle, from a subvariety cut out by a global section of a holomorphic vector bundle.  In \cite{manivel}, he established the following result.

\begin{d-thm}[Manivel]\label{ot-manivel}
Let $X$ be a Stein manifold of dimension $n$, $E$ a vector bundle of rank $d$ on $X$, $s \in H^0(X,E)$ a section of $E$ that is generically transverse to the zero section, and 
\[
Y := \left \{ x\in X\ ;\ s(x)=0, \wedge ^d ds(x) \neq 0 \right \}.
\]
Let $\pi : \p (E) \to X$ denote the projectivization of $E$.  The section $s$ defines a section $\sigma$ of $\co _{E^*}(-1) \to \p (E)$ over $\pi ^{-1} (X-s^{-1}(0))\subset \p (E)$.

We assume that $\co _{E^*}(-1)$ is equipped with a Hermitian metric $e^{-\gamma}$, and $X$ with a positive closed $(1,1)$-form $\Omega$, such that 
\[
\pi ^*\Omega \ge \ii \di \dbar \gamma \quad \text{on }\p (E).
\]
Let $L \to X$ be a holomorphic line bundle with Hermitian metric $e^{-\psi}$ such that 
\[
\frac{1}{d}\pi ^* \ii \di \dbar \psi \ge \alpha \pi ^* \Omega + \ii \di \dbar \gamma.
\]
We suppose also that $E$ admits a Hermitian metric such that $|s|\le \kappa |\sigma|$, where $\kappa$ is real and strictly positive, and we equip $L\tensor \det E^*$ with the associated Hermitian metric.

Then for each plurisubharmonic function $\xi$ on $X$, each positive real number $\beta$, and each holomorphic sections $g$ of $K_Y \tensor L \tensor \det E^* \to Y$ such that 
\[
\int _Y e^{-\xi}|g|^2e^{-\psi} < +\infty,
\]
there exists a holomorphic section $G$ of $K_X\tensor L \to X$ such that 
\[
G|_Y = g \wedge (\wedge ^d ds) \quad \text{and} \quad \int _X \frac{e^{-\xi} |G|^2e^{-\psi}}{(|\sigma|^2e^{-\gamma})^{d-1}(1+|\sigma|^2e^{-\gamma})^{1+\beta}} \le M \int _Y e^{-\xi}|g|^2e^{-\psi},
\]
where the constant $M$ depends only on $d$, $\alpha$, $\kappa$ and $\beta$.
\end{d-thm}

\begin{s-rmk}
In Theorems \ref{ot-stein} and \ref{ot-manivel}, one uses the canonical bundle (whose sections are holomorphic forms of top degree) in order to avoid using a volume form for the $L^2$-norms that arise.  If one does not work with canonical forms, then the Ricci curvature enters the hypotheses.  It is easy to pass back and forth between the two cases, since the square of the canonical bundle is the determinant of the (real) tangent bundle, and a volume form is just a metric for the dual of the canonical bundle.  In a K\"ahler manifold, if the volume form is the determinant of the K\"ahler metric, the curvature of the metric for the canonical bundle induced by the reciprocal of the volume form  is precisely the negative of the Ricci curvature.
\red
\end{s-rmk}

In his work on the deformation invariance of plurigenera for complex projective manifolds \cite{s-02}, Siu gave another form of the $L^2$ extension theorem, in which he introduced a new perspective on the twisted technique, which is the perspective we took in the discussion in Section \ref{twisted-prelims-section}, of twisting the metric of the line bundle, rather than replacing $\dbar$ with $\dbar \circ \sqrt{\tau}$ for some function $\tau$.  The statement of Siu's Theorem is as follows.

\begin{d-thm}[Siu \cite{s-02}]\label{siu-ot}
Let $X$ be a complex manifold and $L \to X$ a holomorphic line bundle with singular Hermitian metric $e^{-\vp}$ having non-negative curvature current.  Let $w \in \co (X)$ be a bounded holomorphic function with non-singular zero set $Z$, so that $dw$ is nowhere zero at any point of $Z$.  Assume there exists a hypersurface $V$ in $Y$ such that no component of $Z$ is contained in $V$, and $X-V$ is a Stein manifold.  If $f$ is an $L$-valued holomorphic $(n-1)$-form on $Z$ satisfying 
\[
\int _Z |f|^2e^{-\vp} < +\infty
\]
then there is a holomorphic $n$-form $F$ on $Y$ such that 
\[
F|_{Z} = f \wedge dw \quad \text{and} \quad \int _Y |F|^2e^{-\vp} \le 8\pi \sqrt{2+\tfrac{1}{e}}\left (\sup _Y |w|^2 \right )\int _Z |f|^2e^{-\vp}.
\]
\end{d-thm}

Siu's version of the $L^2$ extension theorem has been widely used in many applications to algebraic geometry.  Siu himself used it as one of two key tools in the proof of the deformation invariance of plurigenera, the second tool being Skoda's ideal membership theorem.  P\u aun \cite{mihai-pluri} was later able to simplify Siu's proof of the deformation invariance of plurigenera by eliminating the need for Skoda's Theorem, and at the same time improve the pluricanonical extension theorem itself.  Let us digress briefly to discuss the theorems of Siu and Paun. 

The $m^{\rm th}$ plurigenus of a compact complex manifold $Y$ is the dimension $h^0(Y, K_Y ^{\tensor m})$ of the space $H^0(Y, K_Y^{\tensor m})$ of global sections of the $m^{\rm th}$ tensor power of the canonical bundle $K_Y$ of $Y$.  The problem of the deformation invariance of plurigenera, which asks whether the plurigenera are invariant in families, is a long-standing question.  Let us make the statement more precise.  Recall that a holomorphic family is a proper holomorphic immersion $\pi : X \to \D$ from a complex manifold $X$ to the unit disk.  Since $\pi$ is proper, each fiber $X_t := \pi ^{-1}(t)$ is a compact complex variety, and since $\pi$ is immersion, the $X_t$ are pairwise-diffeomorphic complex manifolds.  However, their complex structures might vary.  Fundamental work of Griffiths on deformation of Hodge structures showed that the genera, i.e., the dimensions $h^0(X_t,K_{X_t})$ of the global sections of the canonical bundle of $X_t$, are independent of $t \in \D$ (and we will see in a moment that this invariance also follows from the Ohsawa-Takegoshi Extension Theorem \ref{siu-ot}).  The deformation invariance of plurigenera is precisely the statement that the dimensions $h^0(X_t, K_{X_t}^{\tensor m})$ are independent of $t\in \D$.

By Montel's Theorem (together with Bergman's Inequality, a.k.a., the sub-mean value property for plurisubharmonic functions), one can see that if we have a sequence $t_j \to t_o$ and sections $s_j \in H^0(X_{t_j}, K_{X_{t_j}}^{\tensor m})$, then there is a subsequence converging to a section $s \in H^0(X_{t_o}, K_{X_{t_o}}^{\tensor m})$.  Therefore $t \mapsto h^0(X_t , K_{X_t}^{\tensor m})$ is upper semi-continuous.  In order to show that the dimension does not jump, it suffices to show that given a section $s \in H^0(X_o, K_{X_o}^{\tensor m})$, there is a section $S \in H^0(X, K_X^{\tensor m})$ such that 
\[
S|_{X_o} = s\tensor d\pi ^{\tensor m}.
\]
Indeed, then the sections $(S|_{X_t})/(d\pi ^{\tensor m})$ are sections of $H^0(X_t, K_{X_t} ^{\tensor m})$, so $t \mapsto h^0(X_t , K_{X_t}^{\tensor m})$ is lower semi-continuous, hence continuous, hence, since it is integer-valued, constant.  In this way, we see that the deformation invariance of plurigenera follows from an extension theorem for pluricanonical sections.

\begin{rmk}
Note that if $m =1$, then Siu's Extension Theorem \ref{siu-ot}, with $L\to X$ taken to be the trivial bundle, shows that the genus (i.e., the first plurigenus) is invariant in families.  This result was previously known through important work of Griffiths using the deformation of Hodge structures.
\red
\end{rmk}

In the late 1960s and early 1970s Iitaka showed that the plurigenera of surfaces are invariant in families.  In 1986 Nakayama showed that the plurigenera are not invariant in families that are not K\"ahler.  At that point, it was conjectured only that plurigenera were invariant in families of projective manifolds of general type, and this result was proved by Siu in his celebrated paper \cite{s-98}.  A short while later, Siu proved that the plurigenera are invariant for \emph{any} family of projective manifolds.  (A projective family is a proper holomorphic immersion $\pi :X \to \D$ together with a line bundle $A \to X$ that admits a smooth metric of strictly positive curvature.)

In fact, the $L^2$ Extension Theorem suggests a twisted version of the problem of deformation invariance of plurigenera:  If $Y$ is a compact complex manifold and $L \to Y$ is a holomorphic line bundle, we can define the $L$-twisted plurigenera 
\[
h^0(Y, K_Y ^{\tensor m}\tensor L).
\]
This was indeed done by Siu, who also proposed the result.  Far from being an unmotivated generalization, the methods of Siu were the catalyst for a flurry of incredible activity in binational geometry.  We will not discuss these results, as they lie well outside the scope of this article.  We can however, state Siu's extension theorem, and P\u aun's generalization, which was itself conjectured by Siu.

\begin{thm}[Siu \cite{s-02}]
Let $\pi : X \to \D$ be a projective holomorphic family, and let $L \to X$ be a holomorphic line bundle admitting a singular Hermitian metric $e^{-\vp}$ whose curvature current $\ii \di \dbar \vp$ is non-negative.  (Such an $L$ is called \emph{pseudoeffective}.) Assume, moreover, the the metric $e^{-\vp}$ restricts to $X_o = \pi ^{-1}(o)$ as a singular Hermitian metric, and moreover, $e^{-\vp}$ is locally integrable on $X_o$.  Then for any section $s \in H^0(X_0, K_{X_o} \tensor L)$ there exists a section $H^0(X, K_X^{\tensor m} \tensor L)$ such that 
\[
S|_{X_o} = s \tensor (d\pi)^{\tensor m}.
\]
\end{thm}

\begin{thm}[Paun \cite{mihai-pluri}]
Let $\pi : X \to \D$ be a projective holomorphic family, and let $L \to X$ be a holomorphic line bundle admitting a singular Hermitian metric $e^{-\vp}$ whose curvature current $\ii \di \dbar \vp$ is non-negative.  (Such an $L$ is called \emph{pseudoeffective}.)  Fix any smooth K\"ahler metric $\omega$ for $X$.  Then for any section $s \in H^0(X_0, K_{X_o} \tensor L)$ such that 
\[
\int _{X_o} \frac{|s|^2 e^{-\vp}}{\omega ^{m-1}} < +\infty,
\]
there exists a section $H^0(X, K_X^{\tensor m} \tensor L)$ such that 
\[
S|_{X_o} = s \tensor (d\pi)^{\tensor m}.
\]
\end{thm}

Both of these results involve extension to $X$ of sections on a hypersurface $X_o$ that is cut out by a bounded holomorphic function, namely $\pi$.  By adjunction, the normal bundle of such hypersurfaces are trivial. Paun's Theorem (and hence Siu's Theorem) extends to more general hypersurfaces than these, as was proved by one of us \cite{dv-tak}.  One requires a version of Siu's Extension Theorem \ref{siu-ot} for such hypersurfaces, and the curvature of the (holomorphically extended) normal bundle of the hypersurface enters the picture, as we shall see below (cf. Theorem \ref{ot-tak}).

\subsection{Extension with ``gain"}

The term ``gain" refers to having extension with respect to weights that are not necessarily plurisubharmonic.  The largest class of weights for which extension, with $L^2$ estimates in terms of those weights, occurs is unknown and would be difficult to
precisely define. This class would certainly depend on the underlying geometry where one seeks extension. But there are several situations where extension with respect to not-necessarily-plurisubharmonic weights is known, to which we now turn.

\subsubsection{Ohsawa's Theorem: negligible weights}  Motivated by issues surrounding estimates for the Bergman kernel using induction on dimension, Ohsawa \cite{o-95} established the following result.

\begin{d-thm}[Ohsawa]\label{neg-wts}
Let $\Omega$ be a bounded pseudoconvex domain in $\C ^n$ containing the origin, and let $\Omega ' $ be the intersection of $\Omega$ with the complex hyperplane $\{z_n = 0\}$.  Then for any plurisubharmonic function $\psi$ on $\Omega$ such that 
\[
A_{\psi}:= \sup _{z\in \Omega} \psi(z) + 2\log |z_n| < +\infty,
\]
there exists a constant $C$ depending only on $A_{\psi}$, such that for any plurisubharmonic function $\nu$, and any function $f \in \co (\Omega ')$ satisfying 
\[
\int _{\Omega '} e^{-(\nu+\psi)}|f|^2 dV_{n-1} < +\infty,
\]
there exists $F \in \co (\Omega)$ such that $F|\Omega '= f$ and 
\[
\int _{\Omega} e^{-\nu} |F|^2dV_{n} \le C \int _{\Omega '} e^{-(\nu+\psi)}|f|^2 dV_{n-1}.
\]
\end{d-thm}

Thus one can "gain" some positivity for the weight $\nu$, in the sense that extension holds for $\nu$, even though $\nu$ is not as positively curved as the weight $\nu +\psi$.

It is also worth remarking that in his work \cite{o-01}, Ohsawa proved a much more general extension theorem that includes Theorem \ref{neg-wts}, and has a number of applications.  However, that theorem involves a somewhat non-constructive $L^2$ norm and some difficult-to-compute spaces of weights, so we will not discuss it here, sacrificing the generality of Ohsawa's result to stay as concrete as possible.

\subsubsection{Berndtsson's Theorem: Integrable algebraic singularity}

In the $L^2$ extension problem, one would like to make the extension as small as possible.  This smallness can be captured not only in the constant (i.e., the norm of the linear extension operator that is a consequence of the $L^2$ extension theorem), but also if one can carry out $L^2$ extension with respect to weights in the ambient space that are singular (albeit integrable) on the subvariety from which we are extending.  The first example of this sort of result was proved by Berndtsson \cite{bo-96},  though he did not state it as an explicit theorem.

\begin{d-thm}[Berndtsson]\label{bo-gain}
Let $\Omega \relcomp \C ^n$ be a pseudoconvex domain, and let $\vp$ be a plurisubharmonic function in $\Omega$.  Let $h \in \co (\Omega)$ satisfy $||h||_{\infty} \le 1$ and write $Z = h^{-1}(0)$ with $dh\neq 0$ on any component of $Z$.  Then for any holomorphic function $f \in \co (Z)$ such that 
\[
\int _{Z_{\rm reg}} |f|^2 e^{-\vp} < +\infty
\]
there exists a function $F \in \co (\Omega)$ such that $F|_{Z}=f$ and 
\[
\int _{\Omega} \frac{|F|^2e^{-\vp}}{|h|^{2s}} \le \frac{2\pi}{1-s} \int _{Z_{\rm reg}} \frac{|f|^2 e^{-\vp}}{|dh|^2}.
\]
\end{d-thm}

\begin{s-rmk}
Recall that the definition of a holomorphic function on a singular variety already implies that it is locally in some (ambient) neighborhood of each of the points of the variety.  Thus there is a local extension to begin with, and the point is to get global estimates.
\red
\end{s-rmk}

\begin{proof}[Proof of Theorem \ref{bo-gain}]
First we are going to obtain two a priori estimates from Theorem \ref{tbe}.  In the first, let $\tau = 1-|h|^{2(1-s)}$, $\psi = \vp$,  and $A = 2|h|^{2(1-s)}$.  Then 
\[
- \di \dbar \tau = \frac{(1-s)^2}{|h|^{2s}} dh \wedge d\bar h \quad \text{and} \quad \frac{\di \tau \wedge \dbar \tau}{A}  = \frac{(1-s)^2}{2|h|^{2s}} dh \wedge d\bar h
\]
Substituting into Theorem \ref{tbe} and using the positivity of $\ii \di \dbar \vp$ and pseudoconvexity of $\Omega$, we get the estimate 
\begin{equation}\label{first-est}
\frac{(1-s)^2}{2} \int _{\Omega} \frac{|d\bar h(\alpha)|^2}{|h|^{2s}}e^{-\vp} dV \le \int _{\Omega} (1+|h|^{2(1-s)})|\dbar ^*_{\vp}\alpha|^2e^{-\vp} dV
\end{equation}
Next we apply Theorem \ref{tbe} again with the function $\tau = \frac{1}{\pi} \log |h|^{-2(1-s)}$ and $A=\frac{2}{\pi |h|^{2(1-s)}}$.  Then 
\[
- \ii \di \dbar \tau = 2 (1-s)[Z] \quad \text{and} \quad \frac{\di \tau \wedge \dbar \tau}{A} = \frac{(1-s)^2}{2\pi|h|^{2s}}dh \wedge d\bar h,
\]
where $[Z]$ denotes the current of integration over $Z$ and we have used the Lelong-Poincar\'e formula.  We therefore get the estimate 
\begin{eqnarray}
\nonumber && 2 (1-s)\int _{Z} |d\bar h(\alpha)|^2 e^{-\vp} dV \\
\label{second-est} && \qquad \le \frac{1}{\pi} \int _{\Omega}(2|h|^{-2(1-s)}+\log |h|^{-2(1-s)}) |\dbar ^*_{\vp}\alpha|^2e^{-\vp} dV + \frac{(1-s)^2}{2\pi}\int _{\Omega} \frac{|dh (\alpha)|^2}{|h|^{2s}} e^{-\vp} dV
\end{eqnarray}
Substituting \eqref{first-est} into the last term on the right of \eqref{second-est}, we get 
\begin{eqnarray*}
&& 2(1-s) \int _{Z} |d\bar h(\alpha)|^2 e^{-\vp} dV\\
&& \qquad  \le \frac{1}{\pi} \int _{\Omega} \left ( 2 + |h|^{2(1-s)}\log (|h|^{-2(1-s)}) + |h|^{2(1-s)} + |h|^{4(1-s)} \right ) \frac{|\dbar ^*_{\vp} \alpha|^2}{|h|^{2(1-s)}} e^{-\vp} dV
\end{eqnarray*}

Now, by calculus,  $x(\log x^{-1}+2 +x +x^2) \le 4$ for $x \in (0,1]$, and thus we get 
\begin{equation}\label{final-est}
\int _{Z} |d\bar h(\alpha)|^2 e^{-\vp} dV  \le \frac{2}{(1-s)\pi} \int _{\Omega} \frac{|\dbar ^*_{\vp} \alpha|^2}{|h|^{2(1-s)}} e^{-\vp} dV.
\end{equation}

We define the $(0,1)$-current $g$ on $\Omega$ by 
\[
g = f \dbar \frac{1}{h}.
\]
Fix a domain $\Omega _o \relcomp \Omega$ with strictly pseudoconvex boundary.  By definition of distributional solution, 
\[
\dbar u = g \iff \int _{\Omega} \left < g , \alpha \right > e^{-\vp} = \int _{\Omega} u \overline{\dbar ^*_{\vp} \alpha} e^{-\vp} 
\]
for all smooth, $\dbar$-closed $(0,1)$-forms $\alpha$ with compact support.  Now, for all smooth forms $\alpha$ in the domain of $\dbar ^*_{\vp}$, \eqref{final-est} implies  
\begin{eqnarray*}
\left | \int _{\Omega} \left < g , \alpha \right > e^{-\vp} dV \right | & \le & \pi ^2 \int _{Z} \frac{|f|}{|dh|} d\bar h (\alpha)| e^{-\vp} dV \\
&\le & \frac{2\pi}{1-s} \int _{Z} \frac{|f|^2}{|dh|^2} e^{-\vp} dV \int _{\Omega} \frac{|\dbar ^*_{\vp} \alpha|^2}{|h|^{2(1-s)}} e^{-\vp} dV.
\end{eqnarray*}
It follows that there is a distribution $u$ satisfying 
\[
\dbar u = g \quad \text{and} \quad \int _{\Omega} |u|^2 |h|^{2(1-s)}e^{-\vp}dV \le  \frac{2\pi}{1-s} \int _{Z} \frac{|f|^2}{|dh|^2} e^{-\vp} dV.
\]
Finally, let 
\[
F= hu = h(u - \tfrac{f}{h}) +f.
\]
Then $F$ is holomorphic and satisfies 
\[
\int _{\Omega} \frac{|F|^2e^{-\vp}}{|h|^{2s}}dV \le  \frac{2\pi}{1-s} \int _{Z} \frac{|f|^2}{|dh|^2} e^{-\vp} dV.
\]
Observe also that $u - f/h$ is holomorphic away from $Z$, and since the singularities of $u$ and $f/h$ are the same, $u-f/h$ extends holomorphically across $Z$.  It follows that $F|_Z = f$, and the proof is finished.
\end{proof}

\subsubsection{Demailly's Theorem:  Logarithmic singularity}

In his paper \cite{dem-ot}, Demailly established a rather general result, a special case of which is the following theorem.  

\begin{d-thm}\label{dem-ot-thm}
Let $(X,\omega)$ be an $n$-dimensional Stein K\"ahler manifold, let $(L, e^{-\vp}) \to X$ be Hermitian holomorphic line bundles, and let $Z \subset X$ be a smooth complex hypersurface cut out by a holomorphic function $h \in \co (X)$.  Assume that
\[
\sup _X |h|^2 \le 1 \quad \text{and} \quad \ii \di \dbar \vp +{\rm Ricci}(\omega) \ge 0.
\]
Then for any $f \in H^0(Z, L|_Z)$ satisfying 
\[
\int _Z \frac{|f|^2 e^{-\vp}}{|dh|^2_{\omega}} \frac{\omega ^{n-1}}{(n-1)!}  < +\infty
\]
there exists a section $F \in H^0(X, L)$ such that 
\[
F|_{Y} = f \quad \text{and} \quad \int _X \frac{|F|^2e^{-\vp}}{|h|^2 \left ( \log \tfrac{e}{|h|}\right )^2} \frac{\omega ^n}{n!} \le C \int _Z \frac{|f|^2 e^{-\vp}}{|dh|^2_{\omega}} \frac{\omega ^{n-1}}{(n-1)!},
\]
where $C$ is a universal constant.
\end{d-thm}

\begin{rmk}
Note that since $r\log (e/r) \le 1$,  Demailly's Theorem gives an extension with better estimates than those of the Ohsawa-Takegoshi Theorem \ref{ot-vanilla}.  

On the other hand, one can also use the estimate
\[
s\log \frac{e}{r} \le r^{-s}
\]
to get and extension $F$ with estimates on 
\[
\int _X \frac{|F|^2e^{-\vp}}{|h|^{2-2s}} \frac{\omega ^n}{n!},
\]
but the estimates one gets from Demailly's Theorem are $O(s^{-2})$, whereas in Berndtsson's Theorem \ref{bo-gain} the estimate is better: it is $O(s^{-1})$.

More interestingly, the weights appearing on the left side of the estimate are not plurisubharmonic, and thus Demailly's Theorem is another example of a gain-type result.
\red
\end{rmk}

\subsubsection{Theory of denominators, and a general $L^2$ extension theorem with gain}\label{SS:denominators}

In \cite{mv-gain}, the authors introduced an approach to $L^2$ extension that encompassed all of the gain-type results discussed so far.  At the heart of the result is the notion of denominators, which we now present.

\begin{defn}
Functions in the class $\sd$, called {\it denominators}, are non-negative functions on $[0,\infty)$ with the following three properties.
\begin{enumerate}
\item[(i)] Each $g \in \sd$ is continuous and increasing.
\item[(ii)] For each $g \in \sd$ the improper integral 
\[
C(g) := \int _1 ^{\infty} \frac{dt}{g(t)} 
\]
is finite.
\end{enumerate}
For each $\delta > 0$, set 
\[
G_{\delta} (x) := \frac{1}{1+\delta} \left ( 1 + \frac{\delta}{C(g)} \int _1 ^x \frac{dt}{g(t)} \right ),
\]
and note that $0 < G_{\delta}(x) \le 1$.  Let 
\[
h_{\delta}(x) := \int _1 ^x \frac{1-G_{\delta}(y)}{G_{\delta}(y)} dy.
\]
\begin{enumerate}
\item[(iii)] For each $g \in \sd$ there exists a constant $\delta > 0$ such that 
\[
K_{\delta}(g) := \sup _{x \ge 1} \frac{x+h_{\delta}(x)}{g(x)}
\]
is finite.
\red
\end{enumerate}
\end{defn}

With the notion of denominators in hand, we can now state the main result of \cite{mv-gain}.

\begin{d-thm}\label{mv-ot}
Let $(X,\omega)$ be a Stein K\"ahler manifold, and $w \in \co (X)$ a holomorphic function such that
\[
\sup _X |w| \le 1 \quad \text{and, with } Z:= w^{-1}(0), \quad dw|_Z \text{ is nowhere zero}.
\]
Let $H \to X$ be a holomorphic line bundle with singular Hermitian metric $e^{-\kappa}$ such that 
\[
\ii \di \dbar \kappa + {\rm Ricci}(\omega) \ge 0.
\]
Let $g \in \sd$.  Suppose $R : X \to \R$ is a function such that for all $\gamma > 1$, and all sufficiently small $\ve > 0$ (depending on $\gamma -1$),   
\begin{enumerate}
\item[(a)] $ \alpha - g^{-1}(e^{-R}g(\alpha)$ is subharmonic, and
\item[(b)] $g^{-1}(e^{-R}g(1 - \log |w|^2)) \ge 1$,
\end{enumerate}
where $\alpha := \gamma - \log (|w|^2+\ve ^2)$.  Then for every holomorphic section $f \in H^0(Z,H|_Z)$ such that 
\[
\int _Z \frac{|f|^2e^{-\kappa}}{|dw|^2_{\omega}} \frac{\omega ^{n-1}}{(n-1)!} < +\infty
\]
there exists a holomorphic section $F \in H^0(X, H)$ such that 
\[
F|_Z = f \quad \text{and} \quad \frac{1}{2\pi} \int _X \frac{|F|^2e^{-\kappa}}{|w|^2g\left (\log \frac{e}{|w|^2}\right )} \frac{\omega ^n}{n!} \le 4 \left ( K_{\delta}(g) + \frac{1+\delta}{\delta}C(g)\right ) \int _Z \frac{|f|^2e^{-\kappa}}{|dw|^2_{\omega}} \frac{\omega ^{n-1}}{(n-1)!} .
\]
\end{d-thm}

The class of denominators is rather rich, though there has not been a careful study of just how rich.  Here are a few interesting examples.
\begin{enumerate}
\item[(I)] $g_s(x) = s^{-1} e^{s(x-1)}$, $s \in (0,1]$
\item[(II)] $g_s(x) = s^{-1} x^{1+s}$, $s\in (0,1]$
\item[(III)] $g_{N,s}(x) = s^{-1} xL_1(x)L_2(x)\cdots L_{N-2}(x)(L_{N-1}(x))^{1+s}$, $s\in (0,1]$, $N\in [2,\infty)\cap \Z$,

where $E_j = \exp ^{(j)}(1)$ and $L_j (x) = \log ^{(j)}(E_jx)$.
\end{enumerate}
With the function $g_1$ of type (I), Theorem \ref{mv-ot} recovers Theorem \ref{ot-vanilla} of Ohsawa-Takegoshi (set $R=0$), as well as a generalization of Theorem \ref{neg-wts} of Ohsawa, in which the function $\psi$ (which is $R$ in Theorem \ref{mv-ot}) is less restricted.  With the functions $g_s$ of type (II), we recover Berndtsson's Theorem \ref{bo-gain}.  Finally, with the function $g_{1,2}$ of type (III) we recover Demailly's Theorem \ref{dem-ot-thm}.

\subsection{Extension from hypersurfaces with non-trivial normal bundle}

Many of the $L^2$ extension theorems presented in the previous paragraph involved extension from a hypersurface $Z$ cut out of $X$ by a bounded holomorphic function.  One exception is Manivel's Theorem \ref{ot-manivel}, in which the submanifold $Z$ is cut out by a section of a holomorphic vector bundle that is generically transverse to the zero section.  In fact, the $L^2$ extension technique works for much more general complex subvarieties of $X$, though if the subvariety is not cut out by a section of a vector bundle that is generically transverse to the zero section, the constants become much less controlled.

Demailly's Theorem \ref{dem-ot-thm} that we quoted above is actually a special case of his result, where he considers the same setup as Manivel.  

In \cite{dv-tak}, one of us established the following result on $L^2$ extension from smooth hypersurface cut out by a non-trivial line bundle.  

\begin{d-thm}\label{ot-tak}
Let $(X,\omega)$ be a Stein K\"ahler manifold, and let $Z \subset X$ be a smooth hypersurface.  Assume there exists a section $T \in H^0(X,L_Z)$ and a metric $e^{-\lambda}$ for the line bundle $L_Z \to X$ associated to the smooth divisor $Z$, such that $e^{-\lambda}|_Z$ is still a singular Hermitian metric, and 
\begin{equation}\label{section-bound}
\sup _X |T|^2e^{-\lambda} \le 1.
\end{equation}
Let $H \to X$ be a holomorphic line bundle with singular Hermitian metric $e^{-\psi}$ such that $e^{-\psi}|_Z$ is still a singular Hermitian metric.  Assume that 
\[ 
\ii (\di \dbar \psi  +{\rm Ricci}(\omega)) \ge \ii \di \dbar \lambda _Z
\]
and
\[
\ii (\di \dbar \psi +{\rm Ricci}(\omega)) \ge(1+ \delta) \ii \di \dbar \lambda _Z
\]
for some positive constant $\delta \le 1$.  Then for any section $f \in H^0(Z,H)$ satisfying 
\[
\int _Z \frac{|f|^2e^{-\psi}}{|dT|_{\omega}^2e^{-\lambda }}dA_{\omega} <+\infty 
\]
there exists a section $F\in H^0(X,H)$ such that 
\[
F|_Z=f \quad \text{and} \quad \int _X |F|^2e^{-\psi} dV_{\omega} \le \frac{24\pi}{\delta}\int _Z \frac{|f|^2e^{-\psi}}{|dT|_{\omega}^2e^{-\lambda }}dA_{\omega}.
\]
\end{d-thm}

Theorem \ref{ot-tak} seems to have been the first result in which the metric $e^{-\lambda}$ is allowed to be singular, though we point out that the hypothesis \eqref{section-bound} puts rather a strong constraint on just how singular the metric could be.

\begin{rmk}
Let $Z \subset X$ be a smooth complex hypersurface, defined on a given coordinate chart $U_j$ as the zero set of a holomorphic function $f_j\in \co (U_j)$.  If $U_j$ and $U_k$ are two such coordinate charts, then the function $g_{jk} := f_j/f_k$ is holomorphic and nowhere zero on $U_j \cap U_k$.  Thus $\{ g_{jk}\}$ define transition functions for a line bundle on $X$, which in Theorem \ref{ot-tak} is denoted $L_Z$.  Moreover, notice that 
\begin{equation}\label{eff-id}
f_j = g_{jk} f_k,
\end{equation}
which means that the $f_j$ fit together to form a global holomorphic section of $L_Z \to X$.  Any other section whose zero set, counting multiplicity, is $Z$, differs from this section by a nowhere-vanishing holomorphic function whose domain of definition is $X$.  (This construction applies to all complex manifolds containing a smooth divisor, and does not use the Stein structure of $X$ is any way.)

If we differentiate the identity \eqref{eff-id}, we get 
\[
df_j = g_{jk}df _k + dg_{jk} f_k.
\]
Restricting to $Z$, we find that $df_j = g_{jk}df _k$, which means that$\{df _j\}$ define a section of the vector bundle $(T^*_X \tensor L_Z)|_Z \to Z$.  Moreover, this section annihilates $T_Z$.  It follows that $\{df_j\}$ is in fact a section of the line bundle $N^*_{X/Z} \tensor (L_Z)|_Z \to Z$.  Finally, since $Z$ is smooth, the section $\{df_j\}$ is nowhere zero.  Therefore the line bundle 
\[
N^*_{X/Z}\tensor (L_Z)|_Z \to Z
\]
is trivial, which is to say, the restriction to $Z$ of the line bundle $L_Z$ is isomorphic to the normal bundle of $Z$ in $X$.  This latter fact is called the {\it Adjunction Formula}.

The Adjunction Formula lends some geometric insight to the curvature hypotheses in Theorem \ref{ot-tak}.  If the normal bundle of $Z$ is very positive, then $H \to Z$ might have a lot of sections, and in order for those sections to extend, we are going to need enough curvature from $H \to X$.
\red
\end{rmk}

It is not difficult to adapt the theory of denominators to the setting of extension from hypersurfaces with non-trivial normal bundle; it is simply a matter of modifying the technique of proof of Theorem \ref{mv-ot} to the setting of Theorem \ref{ot-tak}.

\subsection{An elementary example of extension for divisors with non-trivial normal bundle}

The considerations of the last section might seem rather abstract to the more analytically minded of our readers, but in fact that condition appears rather naturally in very concrete problems.  One such problem, first considered by Seip and Wallsten \cite{sw} and then studied further by Berndtsson and Ortega Cerd\`a, is the problem of interpolation sequences for the generalized Bargmann-Fock space, which we now describe.

The underlying manifold we work on is the complex plane $\C$.  On this space, we have the trivial bundle with nontrivial metric $e^{-\vp}$, and we use it, together with Lebesgue measure, to define the Hilbert space 
\[
\sh ^2 (\C, e^{-\vp} dA) := \left \{ f \in \co (\C)\ ;\ \int _{\C} |f|^2 e^{-\vp} dA < +\infty \right \}.
\]
We say that $\sh ^2 (\C, e^{-\vp} dA)$ is a generalized Bargmann-Fock space if there exists a constant $M$ such that 
\[
M^{-1} \ii \di \dbar |z|^2 \le \ii \di \dbar \vp \le M \ii \di \dbar |z|^2.
\]
Now let $\Gamma \subset \C$ be a closed discrete subset.  To this subset we can attach another Hilbert space, namely the space
\[
\ell ^2 (\Gamma , e^{-\vp}) := \left \{ f : \Gamma \to \C\ ;\ \sum _{\gamma \in \Gamma} |f(\gamma)|^2 e^{-\vp(\gamma)} < +\infty\right \}.
\]
The basic problem is then as follows: find necessary and sufficient conditions on $\Gamma$ to guarantee that the restriction map 
\[
\sr _{\Gamma} : \sh ^2 (\C, e^{-\vp} dA)\to \ell ^2 (\Gamma, e^{-\vp})
\]
is surjective.  If this happens, we say that $\Gamma$ is an interpolation set (for the data $(\C, \vp, dA)$).

The central result of the subject, which in this generality is due to Berndtsson, Ortega Cerd\`a and Seip \cite{quimbo, quimseep}, can be stated as follows.

\begin{d-thm}\label{interp-C-char}
Assume $\sh ^2 (\C, e^{-\vp} dA)$ is a generalized Bargmann-Fock space.  Then a closed discrete subset $\Gamma$ is an interpolation set if and only if 
\begin{enumerate}
\item[(i)] $\Gamma$ is uniformly separated, i.e., 
\[
\inf \{ |\gamma - \mu|\ ;\ \gamma, \mu \in \Gamma,\ \gamma \neq \mu\}  > 0,
\]
and
\item[(ii)] the upper density 
\[
D^+_{\vp} (\Gamma) := \limsup _{r \to \infty} \sup _{z \in\C} \frac{\# (D_r(z)\cap \Gamma)}{\int _{D_r(z)} \ii \di \dbar \vp} 
\]
is strictly less than $1$.
\end{enumerate}
\end{d-thm}

The necessity of the conditions (i) and (ii) for $\Gamma$ to be an interpolation set, which was established in \cite{quimseep}, uses techniques that lie somewhat outside the scope of the present article.  The sufficiency, which preceded necessity by about 3 years, was established in \cite{quimbo} using $L^2$ techniques.  However, recently Pingali and the second author \cite{pv} found a rather direct argument for obtaining this result from Theorem \ref{ot-tak}.  The argument also illuminates the meaning of the curvature of the normal bundle of the hypersurface $\Gamma$ (especially in higher dimensions, which are treated there).  We shall now give the argument.

First, we wish to apply Theorem \ref{ot-tak} to the problem at hand.  To do so, we take $X=\C$, $\omega = \frac{\ii}{2} dz \wedge d\bar z$, and $Z=\Gamma$.  We fix any holomorphic function $T \in \co (\C)$ such that 
\[
\Gamma = {\rm Ord}(T),
\]
i.e., $T$ vanishes to order $1$ along the points of $\Gamma$, and has no other zeros.  We set 
\[
\lambda (z) := \frac{1}{\pi r^2} \int _{D_r(z)} \log |T(\zeta)|^2 dA(\zeta)= \frac{1}{\pi r^2} \int _{D_r(0)} \log |T(z-\zeta)|^2 dA(\zeta).
\]
Then by the sub-mean value property for subharmonic functions,
\[
\log |T|^2 \le \lambda, \quad \text{i.e.}, \quad |T|^2 e^{-\lambda} \le 1,
\]
and by the Poincar\'e-Lelong Formula, 
\[
\Delta \lambda = \frac{\#(\Gamma \cap D_r(z))}{\pi r^2}
\]
Therefore by Theorem \ref{ot-tak}, we have the following result.

\begin{prop}\label{ot-for-interp}
Let $\Gamma \subset \C$ be a closed discrete subset, and assume $\vp$ is a subharmonic weight function such that 
\[
\limsup _{r \to \infty} \sup _{z \in \C} \frac{\#(\Gamma \cap D_r(z))}{\ii \di \dbar \vp(z)} < 1 .
\]
Then for any $f : \Gamma \to \C$ satisfying 
\[
\sum _{\gamma \in \Gamma} \frac{|f(\gamma)|^2e^{-\vp(\gamma)}}{|T'(\gamma)|^2e^{-\lambda(\gamma)}} < +\infty
\]
there exists $F \in \sh ^2 (\C , e^{-\vp} dA)$ such that $F|_{\Gamma} = f$.
\end{prop} 

The `if' part of Theorem \ref{interp-C-char} then follows immediately from the following two lemmas.

\begin{lem}\label{unif-sep}
Let $\Gamma \subset \C$ be a closed discrete subset.  Then $\Gamma$ is uniformly separated with respect to the Euclidean distance if and only if for any $r > 1$ there exists $C_r>0$ such that 
\[
\inf _{\gamma \in \Gamma} |T'(\gamma)|^2e^{-\lambda(\gamma)} \ge C_r.
\]
\end{lem}

\noi The proof of Lemma \ref{unif-sep} was established in \cite{pv}, but it is rather more elementary in dimension $1$.  The reader can find an elementary proof in the $1$-dimensional case in \cite{v-rs1}.

\begin{lem}
Let $\vp$ be a weight function satisfying 
\[
- M\ii \di \dbar |z|^2  \le \Delta \vp \le M \ii \di \dbar |z|^2 ,
\]
and let 
\[
\vp _r (z) :=  \frac{1}{\pi r^2}\int _{D^o_r(z)} \vp (\zeta) \log \frac{r^2}{|\zeta -z|^2} dA (\zeta)= \frac{1}{\pi r^2} \int _{D^o_r(0)}  \vp (\zeta+z) \log \frac{r^2}{|\zeta|^2}dA(\zeta) , \qquad z \in\C.
\]
Then 
\[
-  M \ii \di \dbar |z|^2  \le \Delta \vp_r \le M  \ii \di \dbar |z|^2,
\]
and there is a constant $C_r>0$ such that for all $z \in \C$, 
\[
|\vp (z) - \vp _r(z)| \le C_r.
\]
In particular, we have the following quasi-isometries 
\[
\sh ^2 (\C , e^{-\vp}dA) \asymp \sh ^2 (\C , e^{-\vp_r} dA) \quad \text{and} \quad \ell ^2(\Gamma, e^{-\vp}) \asymp \ell ^2 (\Gamma, e^{-\vp_r}).
\]
of Hilbert spaces given by the identity map.
\end{lem}

\noi Again, for a proof see \cite{v-rs1}.

\subsection{Higher forms}

The $L^2$ extension theorems discussed so far have all treated the problem of extending holomorphic sections.  It is natural to ask whether extension is possible for $\dbar$-closed forms of higher bi-degree.  

First, let us mention briefly that, a priori, there are two possible definitions for the restriction of a $\dbar$-closed form with values in a holomorphic line bundle. To explain these, let $X$ be a complex manifold, $\iota : Z \emb X$ a complex submanifold (or subvariety), and $L \to X$ a holomorphic line bundle.  

The first type of restriction of a $(p,q)$-form $\alpha$ with values in $L$ is the pullback $\iota ^*\alpha$.  The resulting object is a differental form on $Z$ (or the regular part of $Z$ if $Z$ is not smooth).  We call this restriction the {\it intrinsic restriction}.

The second type of restriction is a section of the restricted vector bundle 
\[
(\Lambda ^{p,q}_X \tensor L)|_Z \to Z.
\]
We call this second type of restriction the {\it ambient restriction}.  While ambient restriction is a little less natural than intrinsic restriction, since the restriction to $Z$ of a $(p,q)$-form on $X$ is no longer a $(p,q)$-form on $Z$, it is nevertheless a useful notion of restriction in certain contexts.

In his paper \cite{manivel}, Manivel claimed that the methods used to prove $L^2$ extension of holomorphic sections carry over to $\dbar$-closed forms of higher bi-degree.  It was later pointed out by Demailly \cite{dem-ot} that Manivel's deduction was not correct, because the proofs of all the $L^2$ extension techniques above use the interior ellipticity of the $\dbar$-operator, and this ellipticity fails for $(p,q)$-forms as soon as $q \ge 1$.  (In the extreme case, if $\alpha$ has bi-degree $(p,n)$ on an $n$-dimensional complex manifold, then $\dbar \alpha = 0$, so there is no regularity whatsoever.) There is, however, a related elliptic problem of solving $\dbar$ with minimal norm, and Demailly asked whether this problem has some small amount of regularity when the metrics of the line bundles in question are singular.  Demailly's questions remain unsolved at the time of writing of this article.  

A breakthrough came in the work of V. Koziarz \cite{koz}, who was able to show that in fact, one can extend cohomology classes from smooth hypersurfaces.  This amounts to saying that if one is given a $\dbar$-closed twisted form $u$ on the smooth hypersurface, then there is an ambiently defined $\dbar$-closed form $U$ whose restriction to the hypersurface differs from $u$ by a form that is $\dbar$-exact on the hypersurface.  Koziarz actually made use of the Ohsawa-Takegoshi extension theorem for the case $q=0$ by passing to the sheaf-theoretic realization of cohomology; in other words, instead of using forms to represent cohomology, one uses \v Cech cocycles, and these are locally given by holomorphic functions.  One can have these functions be defined on Stein domains, where the Ohsawa-Takegoshi Theorem (or its proof) can be applied.  The functions can be extended, and one proceeds to show that the extensions define a cocycle.  The trouble with the proof is that necessarily in the sheaf-theoretic formulation one must choose a cover, and the constants of $L^2$ extension, which should be universal, end up depending on the cover.

The next step was taken by Berndtsson.  Using his method of $\dbar$ on currents, Berndtsson was able to solve the $L^2$ extension problem for $\dbar$-closed $(n,q)$-forms with values in a holomorphic line bundle, from smooth hypersurfaces in {\it compact} K\"ahler manifolds.  Berndtsson establishes his theorem by using the method of solving $\dbar$ for a current, as developed in \cite{bs}. 

Most recently, the authors have solved the $L^2$ extension problem for $\dbar$-closed forms on a Stein manifold.  Both types of restrictions were considered, though the two extension problems turn out to be essentially equivalent.  Moreover, the techniques are easily adapted to the compact setting (where they are in fact a little easier to establish).

Let us state the main results of \cite{mv-ot}.  To set notation, let $X$ be a K\"ahler manifold of complex dimension $n$ with smooth K\"ahler metric $\omega$, and $Z \subset X$ a smooth complex hypersurface.  Let $L \to X$ be a holomorphic line bundle with a possibly singular Hermitian metric $e^{-\vp}$ whose singular locus does not lie in $Z$, i.e., such that $e^{-\vp}|_{Z}$ is a metric for $L|_Z$.  Assume also that the line bundle $E_Z \to X$ associated to the divisor $Z$ has a holomorphic section $f_Z$ such that $Z = \{ x \in X\ ;\ f_Z(x) = 0\}$, and a singular Hermitian metric $e^{-\lambda _Z}$, such that 
\[
\sup _X |f_Z|^2e^{-\lambda _Z} = 1.
\]

The first result of \cite{mv-ot} is as follows.
\begin{d-thm}[Ambient $L^2$ extension]\label{ambient-main}
Let the notation be as above, and denote by $\iota :Z \emb X$ the natural inclusion.  Assume that 
\[
\ii \left ((\di \dbar (\vp - \lambda_Z) + {\rm Ricci}(\omega))\right ) \wedge \omega ^{q}\ge 0
\]
and
\[
\ii (\di \dbar (\vp - (1+\delta)\lambda_Z) + {\rm Ricci}(\omega)) \wedge \omega ^{q} \ge 0
\]
for some constant $\delta > 0$.  Then there is a constant $C>0$ such that for any smooth section $\xi$ of the vector bundle $(L\tensor \Lambda ^{0,q}_X)|_Z \to Z$ satisfying
\[
\dbar (\iota ^* \xi )= 0 \quad \text{and} \quad \int _Z\frac{ |\xi|_{\omega}^2 e^{-\vp}}{|df_Z|^2e^{-\lambda _Z}} \omega^{n-1} <+\infty ,
\]
there exists a smooth $\dbar$-closed $L$-valued $(0,q)$-form $u$ on $X$ such that
\[
u|_Z = \xi \quad \text{and} \quad \int _X|u|_{\omega}^2 e^{-\vp} \frac{\omega^n}{n!} \le \frac{C}{\delta}\int _Z\frac{ |\xi|_{\omega}^2 e^{-\vp}}{|df_Z|^2e^{-\lambda _Z}} \frac{\omega^{n-1}}{(n-1)!} .
\]
The constant $C$ is universal, i.e., it is independent of all the data.
\end{d-thm}

Given a smooth section $\xi$ of $L\tensor \Lambda ^{0,q}_X)|_{Z}$, the pullback $\iota ^* \xi$ is a well-defined $L$-valued $(0,q)$-form on $Z$.  Now, if $\eta$ is an $L$-valued $(0,q)$-form on $Z$ then the orthogonal projection $P: T^{0,1}_X|_Z \to T^{0,1}_Z$ induced by the K\"ahler metric $\omega$ maps $\eta$ to a section $P^*\eta$ of $(L \tensor \Lambda ^{0,q}_X)|_Z \to Z$, by the formula
\[
\left < P^*\eta , \bar v_1 \wedge ... \wedge \bar v_q\right > := \left < \eta ,  (P\bar v_1) \wedge ... \wedge (P\bar v_q) \right> \quad \text{in }L_z
\]
for all $v_1,...,v_q \in T^{*0,1}_{X,z}$.  The map $P^*$ is an isometry for the pointwise norm on $(0,q)$-forms induced by $\omega$, and since $\iota ^*P^*\eta = \eta$, the hypotheses of Theorem \ref{ambient-main} apply to $\xi = P^*\eta$, and we obtain the following theorem.

\begin{d-thm}[Intrinsic $L^2$ extension] \label{intrinsic-main}
Suppose the hypotheses of Theorem \ref{ambient-main} are satisfied.  Then there is a universal constant $C>0$ such that for any smooth $\dbar$-closed $L$-valued $(0,q)$-form $\eta$ on $Z$ satisfying 
\[
\int _Z\frac{ |\eta|_{\omega}^2 e^{-\vp}}{|df_Z|^2e^{-\lambda _Z}} \omega^{n-1} <+\infty ,
\]
there exists a smooth $\dbar$-closed $L$-valued $(0,q)$-form $u$ on $X$ such that, with $\iota :Z \emb X$ denoting the natural inclusion,
\[
\iota ^* u = \eta \quad \text{and} \quad \int _X|u|_{\omega}^2 e^{-\vp} \frac{\omega^n}{n!} \le \frac{C}{\delta}\int _Z\frac{ |\eta|_{\omega}^2 e^{-\vp}}{|df_Z|^2e^{-\lambda _Z}} \frac{\omega^{n-1}}{(n-1)!} .
\]
\end{d-thm}

\subsection{Optimal constants}\label{S:optimal_constants}

There has been some interest in obtaining the best constant in the $L^2$ extension theorem.  The main motivation (and at present, essentially the only motivation known to the authors) was linked to the Suita conjecture \cite{suita}, as we now explain.

Let $X$ be a Riemann surface and assume $X$ admits a non-constant bounded subharmonic function.  (Such Riemann surfaces are called {\it hyperbolic}, or sometimes {\it potential theoretically-hyperbolic}.)  It is well known that such a Riemann surface admits a Green's function $G: X\times X \to [-\infty, 0)$, i.e., a function uniquely characterized by the following properties:
\begin{enumerate}
\item[(i)] If we write $G_x(y) = G(y,x)$, then for each $x\in X$, 
\[
\frac{\ii}{\pi} \di \dbar G_x = \delta _x,
\]
and
\item[(ii)] if $H: X\times X \to [-\infty, 0)$ is another function with property (i), then $G \ge H$.
\end{enumerate}
Using the Green's Function, one can construct a conformal metric for $X$ as follows:
\[
\omega _F (x) := \lim _{y \to x} \frac{\ii}{2} \di (e^{G_x}) (y)\wedge \dbar (e^{G_x})(y).
\]
The metric $\omega _F$ is called the {\it fundamental metric}. 
\begin{rmk}
The metric $\omega _F$ can be computed from the Green's Function at a point $x$ as follows.  Choose any holomorphic function $f \in \co (X)$ such that ${\rm Ord}(f) = \{x\}$, i.e., $f$ has exactly one zero of multiplicity $1$, and this zero is at the point $x$.  Since $X$ is an open Riemann surface, and thus Stein, such a function $f$ exists.  From the definition of Green's function, the function 
\[
h_x := G_x - \log |f|
\]
is harmonic.  Then 
\[
\omega _F(x) = e^{2h_x(x)} \frac{\ii}{2} df (x) \wedge d\bar f (x).
\]
\red
\end{rmk}

\begin{ex}
Let $X$ be the unit disk.  Then $G(z,\zeta) = \log \left | \frac{\zeta - z}{1-\bar \zeta z}\right |$, and we have 
\[
\omega _F (z) = \frac{\ii dz \wedge d\bar z}{2(1-|z|^2)^2}.
\]
Thus in the unit disk the fundamental metric agrees with the Poincar\'e metric.  A similar calculation shows that on a bordered Riemann surface with no punctures, the fundamental metric and the Poincar\'e metric are asymptotic at the boundary.  
\red
\end{ex}

Suita's conjecture can be stated as follows:

\begin{conj}\cite{suita}
Let $X$ be a hyperbolic Riemann surface.  Then the Gaussian curvature of the fundamental metric of $X$ is at most $-4$. Moreover, it is exactly $-1$ if and only if $X$ is the unit disk.
\end{conj}

Suita's Conjecture was proved fairly recently by B\l ocki \cite{blocki-suita} for the case where $X$ is a domain in $\C$, and more generally by Guan and Zhou \cite{gz-suita}.  We now sketch the proof.

A theorem of M. Schiffer \cite{sch} states that the curvature form of $\omega _F$ is 
\[
{\rm R}(\omega _F)(z) = - \pi B_X(z,\bar z),
\]
where $B_X(z,\bar w)$ is the Bergman kernel of the Riemann surface $X$, i.e., 
\[
B_X (z,\bar z) := \ii \alpha _j (z) \wedge \overline{\alpha _j(z)},
\]
where $\{\alpha _1,\alpha _2,...\}$ is an orthonormal basis of holomorphic $1$-forms for the Hilbert space 
\[
\sa ^2_X := \left \{ \alpha \in H^0(X, T^{*1,0}_X)\ ;\ \int _X \frac{\ii}{2} \alpha \wedge \overline{\alpha} < +\infty \right \}
\]
of square-integrable holomorphic $1$-forms on $X$.
\begin{rmk}
The following properties of $B_X$ are well-known.

\begin{enumerate}

\item[(B1)] The series defining $B_X$ converges locally uniformly.

\item[(B2)] $B_X$ is independent of the choice of orthonormal basis for $\sh ^2 _X$.

\item[(B3)] $B_X$ is characterized by its holomorphicity (in the complex structure of $X\times \overline{X}$) together with the following reproducing property:  for any $\alpha \in \sh ^2 _X$, 
\[
\alpha (z) = \int _X \frac{\ii}{2} \alpha(w) \wedge \frac{B_X(z,w)}{\ii}.
\]

\item[(B4)] For any smooth area form $dA$ on $X$, 
\[
\frac{B_X (z,z)}{dA} = \sup _{||\alpha ||=1} \frac{\ii \alpha (z)\wedge \overline{\alpha (z)}}{dA}.
\]
\end{enumerate}
\end{rmk}

It was first realized by Ohsawa \cite{o-01} that one need only prove show that 
\[
\omega _F(z) \le \frac{\pi}{4} B_X(z,\bar z),
\]
and that an the $L^2$ extension theorem with optimal constant could be used to produce holomorphic $1$-forms that would give the needed estimate for the curvature of $\omega _F$.

Such an $L^2$ extension theorem was first proved by B\l ocki for domains in $\C$, and more generally by Guan and Zhou.  Recently Ohsawa \cite{o-2014} has given a significantly more elementary proof of the optimal constant extension theorem.  We now state Guan-Zhou's version of this theorem.  

\begin{d-thm}\cite{blocki-suita,gz-suita} \label{ot-canon}
Let $X$ be a Stein manifold of complex dimension $n$ and $Y \subset X$ a smooth hypersurface.  Let $f \in H^0(X,L_Y)$ be the canonical section of the line bundle associated to the smooth divisor $Y$.  Assume there exists a metric $e^{-\lambda}$ for $L_Y$ such that 
\[
\sup _X |f|^2e^{-\lambda} \le 1.
\]
Let $L \to X$ be a holomorphic line bundle with singular Hermitian metric $e^{-\psi}$ such that for some $\delta \le 1$, 
\[
\ii \di \dbar \psi \ge 0 \quad \text{and} \quad \ii \di \dbar \psi \ge \delta \ii \di \dbar \lambda.
\]
Then for any $L$-valued holomorphic $(n-1)$-form $\alpha_o \in H^0(Y, K_Y)$ such that 
\[
\int _Y \frac{\ii ^{(n-1)^2}}{2^{n-1}} \alpha_o \wedge \bar \alpha_o e^{-\psi} < +\infty
\]
there exists a holomorphic $(L+L_Y)$-valued $n$-form $\alpha \in H^0(X, K_X+L+L_Y)$ such that 
\[
\alpha |_Y = \alpha _o \wedge df \quad \text{and} \quad \int _X \frac{\ii ^{n^2}}{2^n}\alpha \wedge \bar \alpha e^{-\psi-\lambda} \le \frac{\pi}{\delta} \int _Y\frac{\ii ^{(n-1)^2}}{2^n} \alpha_o \wedge \bar \alpha_o e^{-\psi}.
\]
\end{d-thm}

Next we take $X$ to be our hyperbolic Riemann surface and $Y$ to be  any point $x \in X$.  Let $f \in \co (X)$ satisfy ${\rm Ord}(f) = x$.  The function
\[
h_x := \log |f| - G_x
\]
is therefore harmonic.  Moreover 
\[
|f|^2e^{-2h_x} = e^{2G_x} \le 1.
\]
We make the choice $\lambda = 2h_x$.  Then $\Delta \lambda = 0$, so we can take the metric $e^{-\psi} = e^{\lambda}$ and set $\delta = 1$.  Theorem \ref{ot-canon} then tells us there exists a holomorphic $1$ form $\alpha$ on $X$ such that 
\[
\alpha (x) = df(x) \quad \text{and} \quad c_{\alpha} := \int _X \frac{\ii}{2} \alpha \wedge \bar \alpha \le \frac{\pi}{2} e^{-2h_x(x)}.
\]
(Here we have extended the $0$-form $\alpha _o = 1$.)

\begin{proof}[Proof of Suita's Conjecture]
Consider the holomorphic $1$-form $\beta = \frac{1}{\sqrt{c_{\alpha}}}\alpha$.  Then we have 
\[
\int _X \frac{\ii}{2} \beta \wedge \bar \beta =1
\]
while 
\[
\frac{\ii\beta (x) \wedge \overline{\beta(x)}}{\omega _F(x)} = 2 c_{\alpha}^{-1}e^{-2h_x(x)}\ge \frac{4}{\pi}.
\]
It follows from (B4) that 
\[
\frac{B_X(x,x)}{\omega _F(x)} \ge \frac{4}{\pi},
\]
which is what we wanted to show.
\end{proof}

\section{Invariant metric estimates}\label{S:invariant}

\subsection{The Bergman kernel again}
Let us recall the definitions of the classical invariant metrics, starting with the Bergman metric. If $\Omega\subset\C^n$ is a domain, the Bergman kernel function, $B_\Omega(z,w)$,  is the Schwarz kernel of the orthogonal projection operator
$B:L^2(\Omega)\longrightarrow A^2(\Omega)$, where $A^2(\Omega)$ denotes the holomorphic functions in $L^2(\Omega)$. That is,

\begin{equation*}
Bf(z)=\int_\Omega B_\Omega(z,w) f(w)\, dV_e(w),\qquad f\in L^2(\Omega).
\end{equation*}

\begin{rmk}
The link with the Bergman kernel discussed in the previous section is that on a domain in $\C^n$, the canonical bundle is trivial, and moreover the nowhere-zero section $dz^1 \wedge ... \wedge dz ^n$ squares to Lebesgue measure.
\red
\end{rmk}

Let $B(z) =B_\Omega(z,z)$ denote the Bergman kernel of $\Omega$ restricted to the diagonal
of $\Omega\times\Omega$. Define a Hermitian matrix $\left(g_{k\bar l}\right)$ of functions
by setting $$g_{k\bar l}(z)=\frac{\partial^2}{\partial z_k\partial\bar z_l}\log B(z).$$
Then if $\alpha=\sum \alpha_k d\bar z_k$ a $(0,1)$-form, the pointwise Bergman length of $\alpha$
is defined
$$|\alpha|_B=\left(\sum_{k,l=1}^n g^{k\bar l} \alpha_k\bar\alpha_l\right)^{1/2},$$
where $\left(g^{k\bar l}\right)=\left(g_{k\bar l}\right)^{-1}$ as matrices. 

\subsection{Invariant metrics}
An often useful way to determine the Bergman metric is as a ratio of extreme value problems. Since Hermitian metrics are usually defined on
tangent vectors, we formulate this alternate definition accordingly; to measure the Bergman length of a co-vector like $\alpha$ above, one merely uses duality in $\C^n$.

It is elementary to show that $B(z)$ itself solves an $L^2$ extremal problem:

$$B(z)=\sup\{ |f(z)|^2:f\in \co(\Omega)\text{ and } ||f||_2\leq 1\},$$
where $\co(\Omega)$ denotes holomorphic functions. If $X\in T^{1,0}_{\Omega}$ is a tangent vector, thought of here as a derivation, define a second extreme-value problem by

$$N(z;X)=\sup\{|Xf(z)|^2:f\in \co(\Omega), f(z)=0,
\text{ and } ||f||_2\leq 1\}.$$
The norm $\|\cdot\|_2$ in both problems is the euclidean $L^2$ norm $\Omega$. The Bergman
length of of $X$ at $z$ is then given by
\[
M_B(z;X)=\left(\frac{N(z;X)}{D(z,z)}\right)^{1/2}.
\]

To define the Caratheodory and Kobayashi metrics, let $H(U_1,U_2)$ denote the set of holomorphic mappings from $U_1$ to $U_2$, if $U_i\subset{\C}^{n_i}, i=1,2$, are open sets.  Let $\D$ denote the unit disk in $\C$. Fix a domain $\Omega\subset {\C}^n$, a point $z$, and a tangent vector $X \in T^{1,0}_{\Omega, z}$.

\begin{enumerate}
\item[(i)]The Caratheodory length of $X$ at $z$ is by definition the number 
\[
M_C(z;X)=\sup\{ |df(z)X|: f\in H(\Omega, \D), f(z)=0\}.
\]
\item[(ii)]
The Kobayashi length of $X$ at $z$ is by definition the number 
\[
M_K(z;X)=\inf\{|a| :\exists f\in H(\D,\Omega)\text{ with } f(0)=z\text{ and } f'(0) =X/a\}.
\]
\end{enumerate}

\begin{rmk}
The definitions we have given for the three metrics above fits the notion of a Finsler metric.  As we saw earlier, the Bergman metric is actually (the norm obtained from) a K\"ahler metric.  On the other hand, in general neither the Caratheodory nor the Kobayashi metrics are not induced by a Riemannian metric.
\red
\end{rmk}

\subsection{Estimates}
It is essentially impossible to find formulas that give the values of the above metrics, except in special cases of domains with high degrees of symmetry. However, a great deal of work in the last thirty
years has led to results showing how these metrics behave, approximately, as $z$ approaches $b\Omega$,
for wide classes of domains $\Omega$. These results show that the invariant metrics (and various derivatives
of $B_\Omega(z,w)$) can be bounded from above and below by an explicit pseudometric defined
in terms of the geometry of $\di \Omega$.

The following types of domains are ones to which we can apply the method above; they are all
{\it finite type} domains, as defined by D'Angelo \cite{dang-annals}, which means that the Levi
form associated to these domains degenerates to at most finite order, in a certain sense.
We refer to \cite{dang-annals} or \cite{dang-book} for the definition of finite type.

\begin{defn}\label{D:simple} Call a smoothly bounded, finite type domain $\Omega\subset{\C}^n$
{\it simple} if it is one of the following types:
\begin{itemize}
\item[(i)] $\Omega$ is strongly pseudoconvex,
\item[(ii)] $n=2$,
\item [(iii)]$\Omega$ is convex,
\item [(iv)] $\Omega$ is decoupled, i.e. $\Omega=\left\{ z:\text{Re }z_n+\sum_{k=1}^{n-1}
f_k(z_k)<0\right\}$ for some subharmonic functions $f_k$ of one complex variable.
\item[(v)] The eigenvalues of the Levi form associated to $b\Omega$ are all comparable.
\end{itemize}\end{defn}

The notion of a simple domain is ad hoc and merely refers to domains where the Bergman kernel and its derivatives have known estimates which are essentially sharp.  These kernel estimates are derived, for the various classes of domains in Definition \ref{D:simple}, in \cite{cat_c2, mcn_c2, mcn_decoup, mcn_convex, nrsw_c2, koe_comp}. 
See \cite{mcn_greene} for an expository account of these estimates.

One corollary of the estimates is that, on a simple domain, the invariant metrics $M_B, M_C,$ and $M_K$ are all comparable to each other as $z\to \di \Omega$. This asymptotic equivalence of all three invariant metrics is definitely known to be false, in general. (See, for example, \cite{died-forn_notcompare}). The comparability on simple domains means that getting $L^2$ estimates on $\dbar$ in either the Caratheodory or
Kobayashi metric, which are only Finsler metrics, can be obtained by estimating $\dbar$ in the Bergman metric, which is Hermitian and has a globally defined potential function (because we're on a domain in $\C^n$). Hermitian metrics given by global potentials are much more amenable to either the weighted or twisted approach for estimating $\dbar$.  We saw an example of such estimates at the end of the last section, in which the Bergman kernel played the role of a curvature rather than a potential. 

The known estimates on the Bergman kernel are sharp enough to show that, when $\Omega$ is simple, the potential function $\eta := c \log B$ satisfies condition \eqref{E:basic-twist-hyp}, for some constant $c>0$. The following $\bar\partial$ theorem then results.

\begin{d-thm}\label{T:invariant} Let $\Omega\subset\subset{\C}^n$ be a simple domain and let $\phi$
be a plurisubharmonic function on $\Omega$. 
There exists a constant $C>0$ so that, if $\alpha$ 
is a $\bar\partial$-closed $(0,1)$-form on $\Omega$, there exists a solution to $\bar\partial
u=\alpha$ which satisfies
$$\int_\Omega |u|^2\, e^{-\phi}\leq C \int_\Omega |\alpha|^2_B\, e^{-\phi},$$
assuming the right hand side is finite.\end{d-thm}

For a detailed proof of Theorem \ref{T:invariant}, and further information about the invariant metrics, see \cite{mcn_invariant}.


\section{Estimates for $\dbar$-Neumann}
\subsection{Compactness and subelliptic estimates}

Let $\square =\bar\partial\bar\partial^* +\bar\partial^*\bar\partial$ denote the ordinary, un-weighted complex 
Laplacian. The $\bar\partial$-Neumann problem is the following: given
$f\in L^2_{p,q}(\Omega)$, find $u\in L^2_{p,q}(\Omega)$ such that

$$\left\{ \aligned &\left(\bar\partial\bar\partial^*+\bar\partial^*\bar\partial\right)
u=f \\ &u\in\text{Dom }(\bar\partial^*)\cap\text{Dom }(\bar\partial) \\
&\bar\partial u\in \text{Dom }(\bar\partial^*) \\
&\bar\partial^* u\in \text{Dom }(\bar\partial).\endaligned\right.$$
When the problem is solvable, the $\bar\partial$-Neumann operator, $N$, maps square-integrable forms into the domain of
$\square$ and inverts $\square$. See \cite{Fol-Koh} for
details about $N$ and many other aspect of the $\bar\partial$-Neumann problem.

The $\bar\partial$-Neumann problem is not an elliptic boundary value problem. There are, however,
two analytic estimates on the $\bar\partial$-Neumann 
problem which serve as substitutes for elliptic estimates,
and have been extensively studied: the compactness estimate and the subelliptic estimate. 
Kohn and Nirenberg, \cite{Koh-Nir},
showed how these estimates can substitute for elliptic estimates, especially with regard to
proving up to the boundary regularity theorems on $N$.
Neither the compactness estimate nor the subelliptic
estimate hold on a general domain; the geometry of the boundary $\di \Omega$ of $\Omega$ plays a crucial role in whether these estimates hold. And, while it is known that the geometry of $\di \Omega$ is intimately connected with whether these estimates hold, it is not yet  understood exactly what geometric conditions imply these estimates, or the ``strength'' of these estimates when they do hold, for example, what sorts of geometric conditions are implied by the estimates. Catlin's theorem characterizing when subelliptic estimates hold, recalled below, is a remarkable result, in that, aside from being a tour de force of ideas and techniques, it is the closest thing we have to a complete picture.  Nevertheless, there remain interesting and difficult questions that are not addressed by Catlin's Theorem.

To state these estimates, we recall some notation, specialized to $(0,1)$-forms.  For an open set $U\subset{\C}^n$,
let ${\cd}^{0,1}(U\cap\overline\Omega)$ be the forms in Dom $(\bar\partial^*)$  which are also smooth on $U\cap\overline\Omega$,
and let $Q(u,u)=||\bar\partial u||^2 +||\bar\partial^* u||^2$ be the Dirichlet form associated to $\square$, defined on ${\cd}^{0,1}(\Omega)$. We say the $\dbar$-Neumann problem is {\it compact} if every sequence $\{u_n\}\in{\cd}^{0,1}(\Omega)$ such that $Q(u_n,u_n)\leq 1$ has a subsequence which converges with respect to the ordinary $L^2$ norm. Equivalently, the $\dbar$-Neumann problem, or $N$, is compact if and only if the following family of estimates, usually called the compactness estimate(s), hold: for every $\eta >0$, there exists a constant $C(\eta)$ such that
\begin{equation}\label{E:compactness}
\|u\|^2\leq \eta Q(u,u) +C(\eta)||u||^2_{-1}, \qquad u\in{\cd}^{0,1}(\Omega),
\end{equation}
where $||\cdot||_{-1}$ denotes the $L^2$ Sobolev norm of order $-1$. Although the estimates \eqref{E:compactness} are stated
globally, i.e. for forms in ${\cd}^{0,1}(\Omega)$, compactness is a local property: $N$ is compact on $\Omega$ if and only if every boundary point of $\Omega$ has a neighborhood $U$ such that $N_U$ -- the $\bar\partial$-Neumann operator associated to $U\cap\Omega$ -- is compact. See \cite{fu-str} for a survey of results on compactness of the $\bar\partial$-Neumann problem.

A subelliptic estimate is a quantified form of compactness. Let $p\in \di \Omega$ and $U$ be a neighborhood of $p$ in ${\C}^n$. A subelliptic estimate of order $\epsilon >0$ holds in $U$ if there exists a constant
$C>0$ such that

\begin{equation}\label{E:subelliptic}
\|u||^2_\epsilon\leq CQ(u,u),\qquad u\in{\cd}^{0,1}(U\cap\overline\Omega),
\end{equation}
where $||\cdot||_\epsilon$ denotes the $L^2$ Sobolev norm of order $\epsilon$.

There are potential-theoretic conditions that imply \eqref{E:compactness}. The first general condition
was given by Catlin \cite{cat_P}.

\begin{defn}\label{D:propP}A pseudoconvex domain $\Omega\subset{\C}^n$ satisfies
Property P if for every $M>0$ there exists $\psi=\psi_M\in C^2(\Omega)$ such that
\begin{itemize}
\item[(i)] $|\psi|\leq 1$ on $\Omega$
\item [(ii)] $i\partial\bar\partial\psi(p)\big(\xi,\xi\big)\geq M|\xi|^2$ for $p\in \di \Omega$ and $\xi\in{\C}^n$.
\end{itemize}
\end{defn}

\begin{d-thm}\label{T:P_implies_compact} If $\Omega$ satisfies Property $P$, then $N$ is compact.
\end{d-thm}

Catlin's proof of Theorem \ref{T:P_implies_compact} follows from Theorem \ref{basic-est-gen}, using the functions in Definition \ref{D:propP} as weight functions there. For a form $u\in{\cd}^{0,1}(\Omega)$, write $u=u_1+u_2$, where $u_1$ is supported near $\di\Omega$
and $u_2$ is compactly supported in $\Omega$. It follows from (ii) of Definition \ref{D:propP} that, for arbitrarily large $M$, 
\[
||u_1||^2\leq\frac 1M Q(u_1,u_1).
\]
But $u_2$ satisfies elliptic estimates, since its support is disjoint from $b\Omega$. Together, these estimates
imply that \eqref{E:compactness} holds.

In \cite{mcn-sbg}, a generalization of Property $P$ is given.

\begin{defn}\label{D:prop_tildeP} A domain $\Omega$ is said to satisfy Property $\tilde P$ if for every $M>0$ there exists $\tilde\psi=\tilde\phi_M\in C^2(\overline\Omega)$ such that
\begin{itemize}
\item[(i)] $\tilde\psi$ has self-bounded gradient
\item[(ii)] $i\partial\bar\partial\tilde\psi(p)\big(\xi,\xi\big)\geq M|\xi|^2$ 
for $p\in \di \Omega$ and $\xi\in{\Bbb C}^n$.
\end{itemize}\end{defn}

\begin{d-thm}\label{T:tildeP_implies_compact}If $\Omega$ satisfies Property $\tilde P$, then $N$ is compact.
\end{d-thm}

Property $P$ implies Property $\tilde P$.  Indeed, given a family of functions $\psi_M$
satisfying the conditions in Definition \ref{D:propP}, the functions $\tilde\psi_M=\exp (\psi_M)$ satisfy Definition \ref{D:prop_tildeP}. But Property $\tilde P$ is more general.  As stated earlier, (i) in Definition \ref{D:prop_tildeP} does not imply that $\tilde\psi_M$ is bounded independent of $M$, e.g. the functions $\tilde\psi_M =-\log(-f+\frac 1M)$ for $f<0$ and strictly plurisubharmonic, have self-bounded gradient but are not uniformly bounded near $\{f=0\}$.

The proof of Theorem \ref{T:tildeP_implies_compact} proceeds by a duality argument. First, the functions in Definition  \ref{D:prop_tildeP} are used as the weight functions $\psi$ in Theorem \ref{tbe};  $\tau$ and $A$ are set equal to $e^{-\psi}$ and $2e^{-\psi}$, respectively. Splitting $u\in{\cd}^{0,1}(\Omega)$, $u=u_1+u_2$, as before, we obtain
\[
||u_1||^2_{2\psi}\leq \frac 1M\left(||\bar\partial u_1||^2_{2\psi} +
||\bar\partial^*_\psi u_1||^2_{2\psi}\right).
\]
Note the weight $\psi$ in $\bar\partial^*_\psi$, while the norms are with respect to $2\psi$. A Riesz representation argument, in the same spirit as that which proves Theorem \ref{h-type-thm}, shows that if $\alpha$ is a $\bar\partial$-closed $(0,1)$-form
\begin{equation}\label{E:P_aux}
||\bar\partial^* N\alpha||^2\leq K\left(\frac 1M||\alpha||^2 +C(M)||\alpha||^2_{-1}\right),
\end{equation}
for a constant $K$ independent of $M$. The compactness of the operator $\bar\partial^*N$ follows from \eqref{E:P_aux}; the compactness of $N$ itself follows from this and a little functional analysis.

Turning to subelliptic estimates, Catlin showed that a quantified version of Property $P$ implies that subelliptic estimates hold.

\begin{d-thm}\label{T:catlin_sub} 
Let $\Omega\subset{\C}^n$ be a smoothly bounded, pseudoconvex domain.  Let $p\in \di \Omega$ and $W$ a neighborhood of $p$. Suppose that, for all sufficiently small $\delta >0$, there exists $\psi_\delta\in C^\infty(\overline\Omega)$ such that
\begin{itemize}
\item[(i)] $\psi_\delta$ is plurisubharmonic on $\Omega\cap W$,
\item[(ii)] $\left|\psi_\delta(z)\right| \leq 1$   for $z\in\Omega\cap W$,
\item[(iii)] For $z\in\left\{ z\in W: -\delta <r(z)<0\right\}$,
\[
\ii\partial\bar\partial\psi_\delta(z)\big(\xi,\xi\big)
\geq c \, \delta^{-2\epsilon}|\xi|^2,\qquad\xi\in{\C}^n
\]
for some positive constant $c>0$ independent of $z$, $\xi$ and $\delta$.
\end{itemize}
Then there exists a neighborhood $U\subset W$ of $p$ and positive constants $C$ such that \eqref{E:subelliptic} holds with $\epsilon$ appearing in (iii) above.
\end{d-thm}

This theorem is, therefore, the stunning equivalence of the geometric condition that $\di \Omega$ has finite type and the analytic condition that \eqref{E:subelliptic} holds for some $\epsilon >0$.

The difficult part of Catlin's paper \cite{cat_sub} is {\it construction} of the functions
$\psi_\delta$ satisfying the hypotheses of Theorem \ref{T:catlin_sub}, in a neighborhood of a point of finite type $p\in \di \Omega$. 
Once the functions $\psi_\delta$ are in hand, Theorem \ref{T:catlin_sub} precisely
connects the rate of blow-up of the Hessians $\ii\partial\bar\partial\psi_\delta$ to the strength
of the subelliptic estimate. However the connection between the type $T(p)$ of a point $p \in \di \Omega$ (see, for example, \cite{dang-annals, dang-book}) and the (best possible) $\epsilon$ in \eqref{E:subelliptic} is not known -- the construction given in \cite{cat_sub} provides an $\epsilon$ such that $\frac{1}{\epsilon}$ is doubly exponential in $T(p)$. Determining
the exact relationship between $T(p)$ and $\epsilon$ is an intriguing and difficult
problem that remains to be solved.

Parallel to the way Theorem \ref{T:tildeP_implies_compact} extends Theorem \ref{T:P_implies_compact}, a more general sufficient condition for subellipticity than Theorem \ref{T:catlin_sub} was established.

\begin{d-thm}{(\cite{herbig_thesis})}\label{T:herbig_thesis} Let $\Omega\subset{\C}^n$ be a smoothly bounded, pseudoconvex domain.
Let $p\in b\Omega$ and $W$ a neighborhood of $p$. Suppose there exists a constant $c>0$ so that,  for all sufficiently small $\delta >0$, there exists $\tilde\psi_\delta\in C^\infty(\overline\Omega)$ such that
\begin{itemize}
\item[(i)] $\tilde\psi_\delta$ is plurisubharmonic on $\Omega\cap W$,
\item[(ii)] $\tilde\psi_\delta$   has self-bounded gradient on $\Omega\cap W$,
\item[(iii)] For $z\in\left\{ z\in W: -\delta <r(z)<0\right\}$,
\[
\ii\partial\bar\partial\tilde\psi_\delta(z)\big(\xi,\xi\big)
\geq c \, \delta^{-2\epsilon}|\xi|^2,\qquad\xi\in{\C}^n.
\]
\end{itemize}
Then there exists a neighborhood $U\subset W$ of $p$ and a constant $C>0$ such that \eqref{E:subelliptic} holds with $\epsilon$ as in (iii) above.
\end{d-thm}

As in Theorem \ref{T:tildeP_implies_compact}, this extension relaxes the requirement that the plurisubharmonic functions be bounded to only having self-bounded
gradient. Herbig's proof of Theorem \ref{T:herbig_thesis} also
starts with the inequalities behind Theorem \ref{tbe}, though new arguments are required, including a careful analysis of a partition of unity in the phase variable. She also obtains
relatively simple constructions of the functions $\tilde\psi_\delta$ on some finite type domains, leading to the hope of
establishing sharp subelliptic estimates on families of domains where these estimates are not yet known.



\end{document}